\documentclass[12pt]{amsart}
\usepackage{amsthm,amsmath,amsxtra,amscd,amssymb,xspace,esint}
\usepackage[cp1251]{inputenc}
\usepackage[T2A]{fontenc}

\numberwithin{equation}{section}

\renewcommand{\Im}{{\operatorname{Im}\,}}
\renewcommand{\Re}{{\operatorname{Re}\,}}
\newcommand{\ord}{\operatorname{ord}\nolimits}
\newcommand{\CC}{{\mathbb{C}}}

\newcommand{\PP}{{\mathbb{P}}}

\newcommand{\RR}{{\mathbb{R}}}
\newcommand{\ZZ}{{\mathbb{Z}}}
\newcommand{\SSS}{{\mathbb{S}}}

\newcommand{\calC}{{\mathcal C}}
\newcommand{\calK}{{\mathcal K}}

\newcommand{\calM}{{\mathcal M}}

\newcommand{\calE}{{\mathcal E}}

\newcommand{\calS}{{\mathcal S}}

\newcommand\Res{\operatorname{Res}}
\newcommand\us{{\underline{s}}}
\newcommand\urho{{\underline{\rho}}}

\newcommand\tC{{\tilde{C}}}

\newcommand\uE{{\underline{{\mathcal E}}}}
\newcommand\uc{{\underline{c}}}

\newcommand\uf{{\underline{f}}}
\newcommand\uh{{\underline{h}}}
\newcommand\ug{{\underline{g}}}

\newcommand\uphi{{\underline{\phi}}}

\newcommand\PS[1][N-1]{{\SSS_+^{#1}}}
\newcommand\pbs{positive blow-up of the sphere}
\newcommand\RHP{jump problem\xspace}
\newcommand\PsiDP{\Psi_{D_k}^+}
\newcommand\PsiDM{\Psi_{D_k}^-}

\theoremstyle{plain}
\newtheorem{thm}{Theorem}[section]
\newtheorem{lm}[thm]{Lemma}
\newtheorem{prop}[thm]{Proposition}
\newtheorem{cor}[thm]{Corollary}
\newtheorem{fact}[thm]{Fact}

\theoremstyle{definition}
\newtheorem{df}[thm]{Definition}
\newtheorem{notat}[thm]{Notation}
\newtheorem{rem}[thm]{Remark}

\begin{document}
\title{Real-normalized differentials: limits on stable curves}
\author{Samuel Grushevsky}
\address{Mathematics Department, Stony Brook University,
Stony Brook, NY 11794-3651, USA}
\email{sam@math.stonybrook.edu}
\thanks{Research of the first author was supported in part by National Science Foundation under the grant DMS-15-01265, and by a Simons Fellowship in Mathematics (Simons Foundation grant \#341858 to Samuel Grushevsky)}
\author{Igor Krichever}
\address{Columbia University, New York, USA, and Skolkovo Institute for Science and Technology, and National Research University Higher School of Economics, and Kharkevich Institute for Information Transmission Problems, and
Landau Institute for Theoretical Physics, Moscow, Russia}
\email{krichev@math.columbia.edu}
\author{Chaya Norton}
\address{Concordia University, Montreal, QC, Canada, and Centre de Recherches Math\'ematiques (CRM), Universit\'e de Montr\'eal, QC, Canada}
\email{chaya.norton@concordia.ca}

\begin{abstract}
We study the behavior of real-normalized (RN) meromorphic differentials on Riemann surfaces under degeneration. We describe all possible limits of RN differentials on any stable curve. In particular we prove that the residues at the nodes are solutions of a suitable Kirchhoff problem on the dual graph of the curve. We further show that the limits of zeroes of RN differentials are the divisor of zeroes of a twisted differential --- an explicitly constructed collection of RN differentials on the irreducible components of the stable curve, with higher order poles at some nodes.

Our main tool is a new method for constructing differentials (in this paper, RN differentials, but the method is more general) on smooth Riemann surfaces, in a plumbing neighborhood of a given stable curve. To accomplish this, we think of a smooth Riemann surface as the complement of a neighborhood of the nodes in a stable curve, with boundary circles identified pairwise. Constructing a differential on a smooth surface with prescribed singularities is then reduces to a construction of a suitably normalized holomorphic differential  with prescribed ``jumps'' (mismatches along the identified circles). We solve this additive analog of the multiplicative Riemann-Hilbert problem in a new way, by using iteratively the Cauchy integration kernels on the irreducible components of the stable curve, instead of using the Cauchy kernel on the plumbed smooth surface. As the stable curve is fixed, this provides explicit estimates for the differential constructed, and allows a precise degeneration analysis.
\end{abstract}

\maketitle

\section*{Introduction}
A smooth {\em jet curve} $X$ is a Riemann surface $C$ with distinct marked points $p_1,\dots,p_n\in C$, and with prescribed singular parts $\sigma_1,\dots,\sigma_n$ of a meromorphic differential at these points. If each prescribed residue $r_\ell$ at each $p_\ell$ is purely imaginary, and $\sum r_\ell=0$, then there exists a unique meromorphic differential $\Psi$ on $C$ with singular part $\sigma_\ell$ at each $p_\ell$, holomorphic on $C\setminus\lbrace p_1,\ldots,p_n\rbrace$, and such that all periods of~$\Psi$ are real. This differential is called the {\em real-normalized (RN)} meromorphic differential, and this paper is one in a series investigating its properties and using it to study the geometry of the moduli space of curves. Here we focus on the behavior of the RN differential as the Riemann surface degenerates to a stable curve.

\subsection*{The jump problem}
Our main technical tool is a new analytic method for studying the behavior of differentials on Riemann surfaces under degeneration. This is done by working explicitly in plumbing coordinates, and we need to introduce some notation to describe it; this setup will be defined in full detail in section~\ref{sec:plumbing}. Fix a nodal curve $C$; its dual graph $\Gamma$ has vertices, denoted $v$, corresponding to irreducible components $C^v$ of the normalization of~$C$, and (unoriented) edges, which we denote $|e|$, corresponding to nodes $q_{|e|}\in C$. We will write $e$ for edges of $\Gamma$ together with a choice of orientation. If an oriented edge $e$ starts from a vertex $v$, we say that it corresponds to a preimage $q_e\in C^v$ of a node $q_{|e|}\in C$. We write $-e$ for the edge $e$ with the opposite orientation, $|e|$ for the corresponding unoriented edge, and write $E$ and $|E|$ for the sets of oriented and unoriented edges, respectively.

Plumbing gives a way to understand versal deformations of $C$ in the Deligne-Mumford compactification --- that is, coordinates on $\overline{\calM}_{g,n}$ transverse to the boundary stratum containing~$C$.
To define plumbing coordinates $\us=(s_1,\ldots,s_{\#|E|})$, one fixes once and for all a local coordinate $z_e$ on the normalization of $C$ near each $q_e$. Let $\widehat C_\us$ be the complement in $C$ of the union of the disks $\lbrace |z_e|<\sqrt{|s_{|e|}|}\rbrace$ around each~$q_e$. Then $\widehat C_\us$ is a Riemann surface with boundary components $\gamma_e:=\lbrace |z_e|=\sqrt{|s_e|}\rbrace$. The compact Riemann surface $C_\us$ is obtained from~$\widehat C_\us$ by identifying each pair of  boundaries $\gamma_e$ and $\gamma_{-e}$ via the map $I_e:z_e\mapsto s_{|e|}/z_e$, to form the seam $\gamma_{|e|}\subset C_\us$. The complex structure on $C_\us$ is obtained by declaring a function on $C_\us$ to be holomorphic if it is holomorphic outside of all seams, and continuous on each seam;  $C_\us$ is smooth if and only if each $s_{|e|}$ is non-zero.

Then a differential on $C_\us$ is the same as a differential on $\widehat C_\us$ such that its boundary values on $\gamma_e$ and $\gamma_{-e}$ match under the pullback by $I_e$; this is to say, there is no ``jump'' on the seam $\gamma_{|e|}$. Our approach to constructing such a differential on $C_\us$ with prescribed singular parts is novel. We start with a collection of meromorphic differentials on the irreducible components $C^v$ with prescribed singular parts. Then of course the boundary values of this collection on the seams do not agree, so there are non-zero ``jumps''. We then construct explicitly a suitably normalized collection of holomorphic differentials on $C^v$, such that their jumps are precisely equal to those of the original collection of meromorphic differentials. Subtracting this collection of holomorphic differentials from the collection of meromorphic differentials then gives a differential on $\widehat C_\us$ with no jumps, i.e.~a meromorphic differential on the smooth surface $C_\us$ with prescribed singularities.

The problem of constructing a differential with prescribed jumps is an additive analog of a well-known general problem, known variously as the Riemann-Hilbert problem, or the Riemann boundary value problem. We use the name {\em jump problem} for the version of the problem that is relevant for us. It is the question of constructing a suitably normalized differential on a Riemann surface with boundary, with prescribed differences of boundary values on pairwise identified boundary circles. Equivalently, this is the problem of constructing a differential on a compact Riemann surface, defined on the complement of a set of disjoint closed loops, with prescribed jumps from one side of each loop to the other side. As a matter of language, we will talk about the jump problem either on $\widehat C_\us$ or on $C_\us$, as is more convenient in each case.

Classically (see~\cite{rodin},\cite{zverovich}), the jump problem is solved by integrating the jumps with respect to the suitably normalized Cauchy integration kernel on the surface, i.e.~on $C_\us$ in our case. In this classical approach it appears very difficult to determine the behavior of the solution under degeneration, as the Cauchy kernel varies with $\us$, and degenerates as $\us\to 0$. The technical core of our paper is a new method for solving the jump problem, which allows explicit estimates for the solution under degeneration.

Instead, we view $\widehat C_\us$ as a subset of the normalization of $C$, and try to obtain the solution of the jump problem by integrating with respect to the suitably normalized Cauchy kernels on $C^v$, which are thus {\em independent} of $\us$. We then compute the jumps on the seams of $C_\us$ of the convolution of arbitrary initial data on the seams with the Cauchy kernels on $C^v$. While the jumps on $\gamma_{|e|}$ of this convolution are of course not equal to the initial data, the condition for the jumps to be equal to the prescribed data amounts to an integral equation on the functions that are convolved with the Cauchy kernels on $C^v$. We show that this integral equation can be solved, by showing that the norm of the corresponding integral operator is sufficiently close to zero, so that it can be formally inverted, as a sum of an iteratively defined series.


Our main technical result is this construction in plumbing coordinates, Proposition~\ref{prop:RH}, and the bound for it, Proposition~\ref{ARNPJP}. This allows us to construct and estimate the RN differentials in an entire neighborhood of $C$ in the moduli space.

While we apply this machinery to study the limits of RN differentials and their zeroes, it can also be used for example to study the behavior of a normalized basis of holomorphic differentials. Hu and the third author~\cite{hunorton} used our approach to the jump problem to extend and reprove the results of Yamada on degenerations of period matrices.

\subsection*{Limits of RN differentials}
We use this explicit construction of RN differentials to understand their degenerations, in plumbing coordinates. Our first result is on limits of RN differentials with arbitrary residues.  As our setup is real-analytic, we state it for degenerating sequences (not families). Let $\lbrace X_k\rbrace$ be a sequence of smooth jet curves converging to a stable jet curve $X$, with underlying smooth curves $C_k$, with plumbing coordinates $\us_k$, converging to a nodal curve $C$, whose dual graph is $\Gamma$. The rough version of our first result is as follows.

\begin{thm}[=Theorem~\ref{thm:limits}+Proposition~\ref{prop:compactness}]\label{thm:limitsrough}
Let $\lbrace X_k\rbrace$ be a sequence of smooth jet curves converging to a stable jet curve $X$. The limit RN differential $\Psi:=\lim_{k\to\infty}\Psi_k$ exists if and only if the solutions of the flow Kirchhoff problem on~$\Gamma$, with inflows $ir_{\ell,k}$ and resistances $\log|\us_k|$, converge. If the limit $\Psi$ exists, then  $\Psi|_{C^v}$  is the RN differential with prescribed singularities at those marked points $p_\ell$ that lie on $C^v$, and with simple poles at the preimages of the nodes, with residues given by the limit of solutions of the flow Kirchhoff problem on $\Gamma$, with  inflows $ir_{\ell}$ and resistances $\log |\us_k|$.
\end{thm}
Thus the existence of the limit RN differential is controlled by the existence of the limit of solutions of the flow Kirchhoff problem on the dual graph of the stable curve (see Definition~\ref{df:kirchhoff} for the precise general statement of the Kirchhoff problem). Surprisingly, it seems that the classical problem of determining and parameterizing all possible limits of solutions of the flow Kirchhoff problem as some resistances approach zero has not been addressed previously. In lemma~\ref{lm:multiKirchhoff} we show that if the resistances (i.e., in our case, $\log|\us_k|$) converge in a suitable iterated real oriented blowup $\PS[\#|E|-1]$ of the non-negative sector of the real sphere $S^{\#|E|-1}:=\left(\RR_{\ge 0}^{\#|E|}\setminus\lbrace 0\rbrace \right)/\RR_{>0}$ (see definition~\ref{df:PS}), then the solutions of the flow Kirchhoff problem converge. We will call such degenerating sequences {\em admissible}, and will show that in an admissible sequence the limit of solutions of the flow Kirchhoff problem is given by the solution of what we call the multi-scale Kirchhoff problem (definition~\ref{df:multiKirchhoff}), with resistance given as a point in $\PS[\#|E|-1]$. The full notation is rather involved, and we thus postpone the precise statement of the theorem to section~\ref{sec:resultslimits}, where it appears as theorem~\ref{thm:limits}, stated using the notation on the Kirchhoff problem, developed in section~\ref{sec:kirchhoff}.

This theorem on limits of RN differentials is proven by applying the jump problem. We start with a collection of RN differentials on the irreducible components $C^v$ of the nodal curve $C$, whose residues at the nodes are given by solutions to the limiting multi-scale flow Kirchhoff problem. We then construct the solution of the jump problem with a normalization condition which ensures that all periods over cycles which do not intersect a neighborhood of the nodes is real. We finally find an explicit perturbation of the residues, as a series expansion in plumbing parameters (see~\ref{c-series}), such that the resulting differential is actually RN. The final part of the argument relies on the estimates for the solution of  the jump problem ensuring these terms disappear in the limit.

\subsection*{Limits of zeroes of RN differentials}
The zeroes of differentials play a crucial role  in various questions on moduli. In Teichm\"uller dynamics one studies the orbits on the stratum: the moduli of Riemann surfaces together with a holomorphic differential with a prescribed configuration of zeroes. For possible applications to Teichm\"uller dynamics, and for applications of common zeroes of RN differentials to cusps of plane curves in our upcoming work~\cite{grkrcusps},  it is natural to study the limits of zeros of differentials under degeneration. The difficulty is that the limit differential $\Psi$ may be identically zero on some irreducible component $C^v$. Algebro-geometrically, one approaches this by considering aspects of limit linear series --- which, however, are not yet fully developed for an arbitrary stable curve, though see~\cite{osserman} for recent progress. In~\cite{strata} the problem is dealt with by deforming, in plumbing coordinates, or using flat surface constructions, differentials on irreducible components of the stable curves that have zeroes as prescribed.

Our approach to locating the zeroes of RN differentials is again via the jump problem, with a further improvement resulting from starting from a better approximation of the solution. Indeed, to determine limits of RN differentials, we started with a collection of RN differentials on $C^v$ that we postulate the limit to be, and then construct the differentials in a neighborhood by using the jump problem. In doing this, we could be starting with an identically zero differential on some~$C^v$, if that is what the limit RN differential on $C^v$ is. Instead, we now start with a collection of non-identically zero RN differentials on $C^v$ which provide a better approximation to the full RN differential on $C_\us$. The key estimate~\ref{boundomegaelll} implies that the solution of the jump problem posed with these improved approximations vanishes to higher order than scaling required to determine the first non-zero term, and thus the location of the zeroes in the limit.

As before, we then need to obtain a bound showing that the solution of the jump problem is smaller than the original RN differentials, and thus disappears in the limit.
For this to be the case, we need to ensure that the suitably rescaled RN differentials locally near $q_e$ and $q_{-e}$ are such that the singular part of the differential on one side cancels the lowest order terms of the (holomorphic) differential on the other. This is the concept of so-called balanced differentials, developed in section~\ref{sec:zeroes}.

We will call a sequence of degenerating smooth jet curves {\em jet-convergent} if the singular parts at the nodes of the differentials constructed using this balanced condition converge, after suitable rescaling (see definition~\ref{df:jetconv} for the precise statement). As the limits of differentials are unchanged under rescaling, it will follow that the limits of zeroes of RN differentials exist in jet-convergent sequences, and are the zeroes of the originally taken collection of not identically zero RN differentials on $C^v$.

A rough statement of our main result is thus the following.
\begin{thm}[=Theorem~\ref{thm:zeroes}+Corollary~\ref{cor:cpt}]\label{thm:zeroesrough}
Any admissible sequence $\lbrace X_k\rbrace$ of smooth jet curves converging to a stable jet curve $X$ has a jet-convergent subsequence. For any jet-convergent sequence the limits of {\em zeroes} of RN differentials exist. These limits of zeroes are the divisor of zeroes on $X$ of a twisted RN differential constructed from the jet-convergent subsequence. In particular, the residues of the twisted RN differential arise from a suitable force Kirchhoff problem.
\end{thm}
By a twisted RN differential here we mean a collection of RN differentials $\Phi^v$ on the irreducible components $ C^v$, with prescribed singularities at $p_\ell$, and with higher order poles at some preimages of the nodes. The divisor of zeroes of such a twisted differential is the set of all its zeroes, with multiplicity, away from the nodes, together with the set of nodes counted with suitable multiplicities. The precise statement of this results requires developing the notion of balanced differentials, and related machinery, and obtaining the necessary bounds for the solution of the jump problem. This is done in section~\ref{sec:zeroes}, where the precise version of our main result is given as theorem~\ref{thm:zeroes}.

Since our construction approximates the RN differential on any jet curve in the neighborhood of a given stable jet curve, it describes {\em all} possible limits of zeroes, and thus in fact constructs a compactification of the moduli space of jet curves onto which the limits of zeroes of RN differentials exists. This compactification can be described as a suitable real blowup, and merits an independent study.

\subsection*{Related work}
The question of describing the closures of strata of Riemann surfaces together with a meromorphic differential is currently under intense investigation, eg.~in~\cite{gendron,chen,fapa,strata}, and the answer there is also in terms of twisted meromorphic differentials on the stable curve.

However, our analytic approach via the jump problem is completely different from the methods employed there, and in particular allows us to describe the RN differential with arbitrary precision on {\em any} smooth jet curve in a plumbing neighborhood. As a byproduct we get for example an explicit description of the residues in the limit via the Kirchhoff problem.

\subsection*{History of the project}
The integral $F:=\Im\int\Psi$ of the RN differential is a single-valued harmonic function on $C\setminus \lbrace p_1,\ldots,p_n\rbrace$. In this guise, as harmonic functions which are potentials of the electromagnetic field created by point charges at the marked points, the RN differentials with simple poles have been studied since at least the time of Maxwell.
A variant of the general notion of RN differentials already appears in \cite{schifferbook}, while their study in full generality was initiated by the second author in~\cite{krfinitegap} and in  \cite{kraveraging}, where the relationship with the Whitham perturbation theory of soliton equations was also established. In~\cite{grkrwhitham,grkrcm} the first and second author applied RN differentials to obtain a new proof of the theorem of Diaz on complete complex subvarieties of the moduli space of curves $\calM_g$, and established a relationship with the loci of spectral curves of the elliptic Calogero-Moser system, while in~\cite{krarba} the second author used RN differentials to prove a conjecture of Arbarello on subvarieties of $\calM_g$.

The rough version of the results of this paper, with an approach not using the \RHP, and not yielding the full statement of theorem~\ref{thm:zeroes}, appeared in the third author's Stony Brook PhD dissertation defended in August 2014. We then developed the current approach to the problem using our solution of the \RHP. The current paper, and our proof, are completely independent of the concurrent and independent progress on the compactifications of strata of differentials with prescribed zeroes, in~\cite{gendron,chen,fapa,strata}.

\subsection*{Structure of the paper}
First, in section~\ref{sec:kirchhoff} we give the statement of the Kirchhoff problem on a general graph, and investigate the properties of its solutions, proving that they are a priori bounded, and constructing the blowup $\PS[\#E-1]$ such that convergence of resistances there implies convergence of solutions of the flow Kirchhoff problem. This section is elementary and does not deal with Riemann surfaces and differentials. The setup and the lemmas from it are essential to stating the main results of the paper.

In section~\ref{sec:notation} we recall the notation for the spaces of jet curves and RN differentials. In section~\ref{sec:resultslimits} we develop the notation for degenerating sequences and give the precise statement of the main theorem~\ref{thm:limits} on limits of RN differentials. In section~\ref{sec:plumbing} we recall the plumbing coordinates and notation. Section~\ref{sec:rhp} contains the technical core of our construction: we pose the \RHP, and use Cauchy kernels independent of plumbing parameters to construct an almost real-normalized (ARN) solution, with a bound on its norm. In section~\ref{sec:RNplumbing} we use the ARN solution of the \RHP to construct the RN differential explicitly in plumbing coordinates, as a sum of a recursively defined series, and effectively bound the terms of these series. In section~\ref{sec:prooflimits} we determine the behavior of this construction of the RN differential in a degenerating sequence, proving the main theorem~\ref{thm:limits} on limit RN differentials.

In section~\ref{sec:zeroes} we introduce the notion of two differentials balancing (canceling up to order $m$ under the map $z\mapsto sz^{-1}$) at a node to construct a better approximation to the RN differential recursively. Starting from a collection of balanced differentials on $C^v$, we show that the ARN solution of the corresponding \RHP is smaller than the differentials themselves, and thus in the limit the balanced differentials dominate --- this yields the main theorem~\ref{thm:zeroes} on limits of zeroes of RN differentials.

In the appendix we formalize this notion of a collection of differentials on $C^v$ that are close to a differential on the plumbed surface, by introducing the notion of an $m$-th order approximation. While this setup is not necessary for our main proofs, the method can be used to study behavior of degenerating differentials with arbitrary precision, as will be investigated elsewhere.

\subsection*{Acknowledgements}
The second author thanks Columbia University for hospitality in January-March 2016, when much work on this paper was done. We are grateful to Scott Wolpert for carefully reading the third author's PhD dissertation, and for many useful discussions and comments on the topics surrounding plumbing.

\section{Limits of solutions of the Kirchhoff problem}\label{sec:kirchhoff}
In this section we pose the Kirchhoff problem on an arbitrary graph, in the generality that we require, and investigate the limits of its solutions. This setup will be used to state our main results.

\begin{notat}
We denote by $\Gamma$ a graph, which is a collection of vertices $v\in V(\Gamma)$, and a collection of edges $|e|\in |E|(\Gamma)$, allowing loops and parallel edges. We further denote by $E(\Gamma)$ the set of oriented edges $e$, writing $-e$ for the same edge as $e$, but with the opposite orientation, and writing $|e|=|-e|$ for the corresponding unoriented edge. Our graphs also have legs, i.e.~half-edges attached to some vertices.
For an oriented edge $e$ we denote $v(e)$ the vertex that is its target, and for a vertex $v$ denote by $E_v$ the set of all edges pointing to it, i.e.~$E_v=\lbrace e\in E(\Gamma): v(e)=v\rbrace$. We denote by $\#E,\#|E|,\#V$ the cardinalities of the corresponding sets. We will use the underline for the elements of $\#E$-dimensional or $\#|E|$-dimensional vector spaces, for example $\uc$ will mean the collection of numbers $c_e$ for all $e\in E$ and $\urho$ will mean the collection of numbers $\rho_{|e|}$ (we'll specify in each case, whether the oriented or unoriented edges are taken).
\end{notat}
\begin{notat}[The Kirchhoff problem]\label{df:kirchhoff}
The {\em general Kirchhoff problem} for a graph $\Gamma$ is the following. As initial data, to every leg $\ell$, one assigns a real number $f_\ell\in\RR$, thought of as the in/outflow of current, and to every unoriented edge $|e|$ of $\Gamma$, one assigns a positive real number $\rho_{|e|}\in\RR_+$, thought of as resistance. In addition one chooses a class in the first cohomology group of the graph, $\calE\in H^1(\Gamma, \RR)$, thought of as the electromotive force.

The Kirchhoff problem is then to find for each oriented edge $e\in E(\Gamma)$ a real number $c_e$ (the electrical current) such that the set of all $c_e$ satisfies the following three conditions:
\begin{itemize}
\item[(0)] $c_e=-c_{-e}$ for any $e\in E(\Gamma)$;

\item[(1)] The total current flow at any vertex is zero: for any  $v\in V(\Gamma)$,
\begin{equation}\label{noflow}
 \sum_{e\in E_v}c_e=-\sum_{\lbrace {\rm  leg\ } \ell\, \mid\, \ell {\rm\ connects\ to\ } v\rbrace} f_\ell;
\end{equation}

\item[(2)] for any oriented cycle of edges $\gamma\subset\Gamma$ the total voltage drop is equal to the
electromotive force along the cycle:

\begin{equation}\label{loopvoltage}
\sum_{e\in \gamma} c_{e} \rho_{|e|} =\calE_{\gamma}:=\langle\calE,\gamma\rangle.
\end{equation}
\end{itemize}
\end{notat}
In modern terminology, the Kirchhoff problem is to find a one-form on the graph with prescribed periods over cycles. In physics, this is the problem of determining the flow of the electrical current. Physically it is classically known that the current flows, and in a unique way:
\begin{fact}\label{fact:unique}
For any connected graph $\Gamma$, if the sum of all $f_\ell$ is equal to zero, the general Kirchhoff problem has a unique solution.
\end{fact}

We think of the Kirchhoff problem as a system of non-homogeneous linear equations on $\uc=\lbrace c_e\rbrace$, with the right-hand-side given by the flows $f_\ell$ and electromotive force $\calE$. As such, its solution is linear in the initial data, and is given as the sum of the solutions of the problem with only $f$'s or only $\calE$ present; we study these two special cases separately.
\begin{df}
The {\em flow Kirchhoff problem} is the special case of the general Kirchhoff problem when the electromotive force $\calE$ is zero.

The {\em (electromotive) force Kirchhoff problem} is the special case of the general Kirchhoff problem when all the in/outflow is zero, i.e.~when all $f_\ell=0$.
\end{df}

\begin{rem}\label{rem:homogeneity}
The solution of the flow Kirchhoff problem is unchanged if  $\urho=\lbrace\rho_{|e|}\rbrace$ is rescaled by some $\mu\in\RR_+$, while $\lbrace f_\ell\rbrace$ are unchanged. Thus it is natural to think of the initial data of the flow Kirchhoff problem as a point $\PP\urho\in S_{>0}^{\# |E|-1}$, where we denote by $S_{>0}^{\# |E|-1}:=\RR_+^{\#|E|}/\RR_+$ the positive octant of the real sphere.

The solution of the force Kirchhoff problem is homogeneous under rescaling resistances: if all resistances $\rho_{|e|}$ are rescaled by $\mu\in\RR_+$ while $\calE$ is unchanged, then $\uc$ is rescaled by $\mu^{-1}$. Thus if thinking of the initial data as $\PP\urho\in S_{>0}^{\#|E|-1}$, then the solution is only defined as a projective point, also.
\end{rem}

\medskip
One crucial feature of our setup is that since all resistances are positive reals, the currents solving the Kirchhoff problem can be a priori bounded. As hinted at by the homogeneity, it is natural to expect a bound for the flow Kirchhoff problem independent of $\urho$, and a bound for the force problem that is linear in $1/\rho$. We prove these two a priori bounds --- which we could not find in the literature --- by elementary arguments.
\begin{lm}\label{lm:aprioribound}
For a given graph $\Gamma$ and given inflows $\lbrace f_\ell\rbrace$, for any edge $e$ of $\Gamma$ the solution $c_e$ of the corresponding flow Kirchhoff problem satisfies
\begin{equation}\label{C-bound}
|c_e|\le \frac12 \sum_{\ell} |f_\ell|
\end{equation}
for any resistances $\urho\in \RR_+^{\#|E|}$.
\end{lm}
\begin{proof}
We prove the statement by induction on the number $k$ of vertices of $\Gamma$. If $k=1$, then every edge $e$ is a loop, and thus by condition (2) of the Kirchhoff problem (i.e.~equation~\eqref{loopvoltage}) $c_e=0$, so that the inequality is trivially satisfied. Suppose now that the statement holds for any graph with $k$ vertices. For a graph with $k+1$ vertices we claim that there must exist a vertex $v$ such that $c_e\ge 0$ for any $e\in E_v$. Indeed, suppose for contradiction that such a vertex did not exist. Then starting from an arbitrary vertex we follow some edge originating from it such that the current is negative, get to another vertex, and repeat. Then eventually we must return to a vertex that we have already visited, and thus we will have constructed an oriented cycle of edges in $\Gamma$ such that $c_e<0$ for any edge in the cycle. However, since all $\rho_{|e|}$ are positive real numbers, the sum $\sum c_e\rho_{|e|}$ over this cycle would be negative, contradicting condition (2) of the flow Kirchhoff problem, as there is no force on the right-hand-side there.

Thus there exists a vertex $v\in V(\Gamma)$ such that $c_e\ge 0$ for any $e\in E_v$. Condition (1) of the Kirchhoff problem (i.e.~equation~\eqref{noflow}) at $v$ then gives
\begin{equation}\label{induc1}
\sum_{e\in E_{v}} |c_e|=\sum_{e\in E_{v}} c_e=-\sum_{\ell: p_\ell \in C^v}f_\ell
\end{equation}
Since the sum of all inflows equals zero we have
$$-\sum_{\ell: p_\ell \in C^v}f_\ell=-\frac 12\left(\sum_{\ell: p_\ell \in C^v}f_\ell-\sum_{\ell: p_\ell \notin C^v}f_\ell\right)\le \frac 12\sum_{\ell}|f_\ell|$$
Hence  inequality~\eqref{C-bound} holds for any $e\in E_v$.

The currents $\lbrace c_e:{e\notin E_v}\rbrace$ are a solution of the flow Kirchhoff problem on the graph $\Gamma'$ whose vertices are $V(\Gamma)\setminus v$, and with additional legs obtained by replacing each oriented edge $e\in E_v$ by a leg attached to $v(-e)$, with inflow $c_e$ in that new leg. By the inductive assumption for the graph $\Gamma'$ we have for any $e\in \Gamma'$ the inequality
$$
|c_e|\le \frac12\left(\sum_{\lbrace \ell :p_\ell \notin C^v\rbrace }|f_\ell|+\sum_{e\in E_v} c_e\right).
$$
holds. Combining this with~\eqref{induc1} implies~\eqref{C-bound} for all edges of the original graph $\Gamma$.
\end{proof}
The bound for solutions of the force Kirchhoff problem is as follows.
\begin{lm}\label{lm:aprioriforce}
For a given graph $\Gamma$ and given electromotive force $\calE$, for any edge $e$ of $\Gamma$ the solution $c_e$ of the force Kirchhoff problem for any  resistances $\urho\in \RR_+^{\#|E|}$ satisfies
\begin{equation}\label{C-bound1}
|c_e|\le \frac{N\cdot|\calE|}{\min_{|e|\in |E|}\rho_{|e|}},
\end{equation}
where $|\calE|$ denotes the maximum value of $\calE$ on simple loops in $\Gamma$, and $N$ is the rank of $H_1(\Gamma)$.
\end{lm}
\begin{proof}
We will prove the lemma by induction on $N$. If $N=0$ then there are no cycles and it is easy to see that in this case all currents $c_e$ are zero. In order to prove the induction step, first note that by condition (1) of the force Kirchhoff problem, in the absence of in and outflows, for any vertex $v\in V(\Gamma)$ there must exist some edge $e_+$ such that $c_{e_+}\ge 0$. As in the proof of the previous lemma, going along such edges we must eventually return back to a vertex already passed, and the first time we do so, we have constructed a simple oriented loop $\gamma\subset \Gamma$ such that $c_{e'}\ge 0$ for any $e'\in\gamma$. Thus for any $e_0\in\gamma$ we have the estimate
\begin{equation}\label{e1}
c_{e_0}\rho_{|e_0|}\le \sum_{e'\in\gamma} c_{e'}\rho_{|e'|}=\calE_\gamma\le |\calE|,
\end{equation}
which is stronger than the required bound~\eqref{C-bound}.

Consider the graph $\Gamma'$ obtained from $\Gamma$ by cutting the edge $e_0$ and attaching to the vertices $v(e_0)$ and $v(-e_0)$ new legs with inflows $c_{e_0}$ and $c_{-e_0}$, respectively. The solution of the force Kirchhoff problem on $\Gamma$ restricted to all edges of $\Gamma'$ coincides with the solution $\tilde c_e$ of the general Kirchhoff problem on $\Gamma'$ with the same force as before on all cycles that did not pass through $e_0$, and with these inflows in the two new legs. Since the general Kirchhoff problem is the linear combination of the flow and force Kirchhoff problems, we can write $\tilde c_e=c_e'+c_e''$, where $c_e'$ and $c_e''$ are the solutions of the corresponding flow and force problems. For $c_e'$ we can use the previous lemma, while for $c_e''$ we use the inductive assumption, obtaining respectively the bounds
\begin{equation}
  |c'_e|\le c_{e_0},\ \ |c_e''|\le \frac{(N-1)|\calE|}{\min_{|e|\in |E|}\rho_{|e|}}
\end{equation}
Combining these estimates with~\eqref{e1} implies the needed bound~\eqref{C-bound1}.
\end{proof}
\begin{rem}\label{rem:size}
We note that as resistances $\rho_{|e|}$ go to infinity, the bound for solutions of the force problem goes to zero, which implies that for the general Kirchhoff problem the limit of solutions is given by the solutions to the corresponding flow Kirchhoff problem. This explains why only the solution of the flow Kirchhoff problem appears in our statement of theorem~\ref{thm:limits} on limits of RN differentials, while the force Kirchhoff problem is used to construct the RN differential explicitly in plumbing coordinates, in section~\ref{sec:RNplumbing}, essentially as corrections to the solution of the flow Kirchhoff problem.
\end{rem}
\begin{rem}
If $\uc$ is the solution of the flow Kirchhoff problem, then every vertex $v$ can be assigned a voltage potential $V_v\in\RR$ such that  Ohm's law $V_{v(e)}=V_{v(-e)}+c_e\rho_{|e|}$ holds for any edge $e$. The voltage potential on a connected graph is unique up to a global additive constant, while its existence is equivalent to condition (2) of the Kirchhoff problem. The voltage potential then induces a full (non-strict) order on the vertices of the graph, which it is natural to call the chronological order (motivated by construction of operator quantization of bosonic string in~\cite{kr-nov}).

This order is very different then the one considered in~\cite{strata} and the order of vanishing stratification that we introduce in definition~\ref{df:jetconv} below. The chronological order is a weak full order on the set $C^{(0)}$ of non-null irreducible components --- i.e.~on what would be the set of top level components in the terminology of~\cite{strata}. The chronological ordering is only present in our RN setup, when all the currents are real.
\end{rem}
\medskip
We now investigate the limits of solutions of the Kirchhoff problem as resistances vary. The flow Kirchhoff problem is a system of inhomogeneous linear equations on the currents $\uc$ with coefficients $\urho$, invariant under scaling $\urho$ by $\RR_+$. Thus the solution of the flow Kirchhoff problem depends continuously on $\PP\urho$. Given a sequence of resistances $\PP\urho_k$ that converges in $S_{>0}^{\#|E|-1}$, it thus follows that the solutions of the corresponding flow Kirchhoff problems converge. Since $S_{>0}^{\#|E|-1}$ is not compact, we will also need to investigate when the solutions of the flow Kirchhoff problems converge if $\PP\urho_k$ do not converge in $S_{>0}^{\#|E|-1}$. The simplest compactification is $S_{\ge 0}^{\#|E|-1}$ --- the closed octant of the sphere where the coordinates are required to be non-negative. However, convergence of $\PP\urho_k$ in $S_{\ge 0}^{\#|E|-1}$ does not guarantee convergence of the corresponding solutions of the Kirchhoff problem: to see this we note that if for some oriented cycle of edges all resistances are zero, an arbitrary constant can be added to all the flows in a cycle. This indicates that convergence of resistances in a certain blowup of $S_{\ge 0}^{\#|E|-1}$ is required to guarantee convergence of solutions of the flow Kirchhoff problem. The necessary blowup is in fact the real oriented blowup of the union of real coordinate planes intersected with the non-negative sector of the real sphere. We refer to~\cite{gillam} for a detailed definition and a survey of properties of the real oriented blowup of complex manifolds. For our purposes we give a direct iterative definition, which will also allow us to write down explicitly the analytic conditions for a sequence to converge in the blowup.
\begin{df}\label{df:PS}
We denote $S_{\ge 0}^{N-1}:=(\RR_{\ge 0}^N\setminus\lbrace 0\rbrace)/\RR_+$ the non-negative sector of the real sphere. The  {\em \pbs}, denoted $\PS$, is the blowup $\pi:\PS\to S_{\ge 0}^{N-1}$ defined recursively in $N$ as follows. We let $\PS[0]$ be a point. Given the definition of $\PS[j]$ for all $0<j<N$, we define $\PS[N]$ to be the result of blowing up every coordinate subspace $\lbrace 0\rbrace^j\times S_{>0}^{N-j-1}$ to $\PS[j]\times S_{>0}^{N-j-1}$ (for all possible renumberings of coordinates).
\end{df}
Recursively, this means that $\PS[j]$ is the disjoint union over all subsets $P\subset\lbrace 1,\dots,N\rbrace$ of the products $S_{>0}^{\# P-1}\times \PS[N-\#P-1]$, where the sphere records those coordinates that are non-zero, and the second factor records the corresponding recursive blowup.

Explicitly, denote $Z_N$ the set of all partitions of the set $\lbrace 1,\dots,N\rbrace$ into numbered subsets: $\lbrace P\rbrace\in Z_N$ is a decomposition $\lbrace 1,\dots,N\rbrace=P_1\sqcup\dots\sqcup P_l$. Then, as a set, $\PS$ is the disjoint union over all $\lbrace P\rbrace\in Z_n$ of the products of positive sectors of the sphere:
\begin{equation}\label{eq:PSset}
 \PS=\sqcup_{\lbrace P\rbrace\in Z_N}\prod_{j=1}^l S_{>0}^{N_j-1}
\end{equation}
where $N_j:=\# P_j$. The topology on $\PS$ is such that a point
$$\PP\urho=(x_1^{(1)}:\dots:x_{N_1}^{(1)})\times\dots\times (x_1^{(l)}:\dots:x_{N_l}^{(l)})\in \prod_{j=1}^l S_{>0}^{N_j-1}\subset \PS
$$ is the limit as $k\to \infty$ of a sequence of points $(y_1[k]:\dots:y_N[k])\in S_{>0}^{N-1}$ if and only if the following conditions hold:
\begin{equation}\label{eq:PS}
\begin{aligned}
 \lim_{k\to\infty} &(y_a[k]x_b^{(j)} - y_b[k]x_a^{(j)})=0 \qquad\mbox{for any } a,b\in P_j\ \mbox{for any }j,\mbox{ and}\\
 \lim_{k\to\infty} &\frac{y_a[k]}{y_b[k]}=0\quad\qquad\qquad\qquad\mbox{for any } a\in P_j, b\in P_{j'}\ \mbox{for any }j>j'.
\end{aligned}
\end{equation}

The case $l=1,\ P_1=\lbrace 1,\dots,N\rbrace$ corresponds to the open dense subset $S_{>0}^{N-1}\subset\PS$. The contraction $\pi:\PS\to S_{\ge 0}^{N-1}$ is defined by sending $\PP\urho$ to a point where all $x_i^{(j)}$ for $j>1$ are replaced by zeroes, while all $x_i^{(1)}$ are unchanged. The map $\pi$ is thus an isomorphism on $S_{>0}^{N-1}$, and we think of $\PS$ as a recursive real oriented blowup of $S_{\ge 0}^{N-1}$ (see eg.~\cite[Sec.~X.9]{acgh2} for a discussion of real oriented blowups). It can be seen that in fact $\PS$ is a real manifold with corners, but all that matters for us is that $\PS$ is a compact topological space containing $S_{>0}^{N-1}$ as a dense open subset.

We will show that convergence of resistances in $\PS[\#|E|-1]$ implies convergence of solutions of the flow Kirchhoff problem, and that the limits of solutions are solutions of the multi-scale Kirchhoff problem, which we now define. For a given point $\PP\urho\in\PS[\#|E|-1]$, let $x:=(x_{i_1}:\dots :x_{i_{\#P_1}}) \in S_{>0}^{\#P_1}$ be the ``largest factor'' in~\eqref{eq:PSset}, corresponding to $P_1$, and let $\PP\urho'\in\PS[\#|E|-\#P_1]$ correspond to the product of all other factors, so that we think of $\PP\urho\in\PS[\#|E|-1]$ as $x\times\PP\urho'\in S_{>0}^{\#P_1}\times\PS[\#|E|-\#P_1]$.
\begin{df}\label{df:multiKirchhoff}
The {\em multi-scale flow Kirchhoff problem} on a graph $\Gamma$ with inflows $\lbrace f_\ell\rbrace$, and generalized resistance $\PP\urho\in\PS[\#|E|-1]$, is posed recursively as follows.

Let $\Gamma_1$ be the graph obtained from $\Gamma$ by contracting all edges $|e|$ for $|e|\notin P_1$. Let $\lbrace c_e\rbrace_{|e|\in P_1}$ be the solution of the flow Kirchhoff problem on $\Gamma_1$ with inflows $f_\ell$ and resistances $x$.

Let $\Gamma_2$  be the (possibly disconnected) graph whose edges are all edges $|e|\in |E|(\Gamma) \setminus P_1$, whose vertices are all the endpoints of such edges, and whose legs are the original legs that connect at these vertices, together with a new leg for each edge $e$ such that $|e|\in P_1$ and $v(e)\in V(\Gamma_2)$. Then on every connected component of  $\Gamma_2$ we recursively pose the multi-scale flow Kirchhoff problem with the inflows being $\lbrace f_\ell\rbrace$ for the original legs attached to $\Gamma_2$, and $c_e$ for each new leg, and with resistances $\PP\urho'$ (note that condition (1) for the flow Kirchhoff problem on $\Gamma_1$ ensures that the sum of the inflows for every connected component of $\Gamma_2$ is then equal to zero).

The {\em solution} to the multi-scale flow Kirchhoff problem on $\Gamma$ is then defined to be the union of $\lbrace c_e:|e|\in P_1\rbrace$ and of the recursively defined solution of the multi-scale flow Kirchhoff problem on each connected component of $\Gamma_2$.
\end{df}
\begin{rem}\label{rem:multiKirchhoffmeaning}
In terms of stable curves, if $\Gamma$ is the dual graph of a stable curve $C$, then $\Gamma_1$ is the graph of the smoothing of $C$ at all the nodes except those indexed by $P_1$, while $\Gamma_2$ is the graph of the partial normalization of all the nodes except those indexed by $P_1$.
\end{rem}

We now prove that if resistances converge in $\PS[\#|E|-1]$, then solutions of the flow Kirchhoff problem converge to the solution of the multi-scale Kirchhoff problem.
\begin{lm}\label{lm:multiKirchhoff}
For a fixed graph $\Gamma$, if a sequence of inflows $f_{\ell,k}$ converges to $f_\ell$, and for a sequence of non-zero resistances $\lbrace\urho_k\rbrace$, the projectivized resistances $\PP\urho_k\in S_{>0}^{\# |E|-1}$ converge to some $\PP\urho\in\PS[\#|E|-1]$, then the solutions $\uc_k$ of the flow Kirchhoff problems with resistances $\urho_k$ and inflows $f_{\ell,k}$ converge. Moreover, the limit of $\uc_k$ is the solution of the multi-scale Kirchhoff problem with the inflows $f_\ell$ and generalized resistance $\PP\urho$ .
\end{lm}
\begin{proof}
We will prove the lemma by induction on the number of levels of the multi-scale problem (that is on the number $l$ of factors in~\eqref{eq:PSset}). If $\PP\urho_k$ converge to some
$\PP\urho\in S_{>0}^{\# |E|-1}$, the statement is obvious, since the flow Kirchhoff problem is simply a system of non-degenerate linear equations, and solutions depend continuously on the parameters $\urho_k$ and inflows $f_{\ell,k}$.

Now, suppose $\pi(\urho)=(x,0)$, where $x$ corresponds to the $P_1$ factor, and denote by $|x|$ the minimal absolute value of coordinates of $x$. Then by rescaling each $\urho_k$ by a suitable positive real number, we can assume that $\urho_k=(x_k,\urho'_k)$ converge to $(x,0)$, while $\PP\urho_k$ converge to $\PP\urho$. Then for $k$ sufficiently large we know that the absolute value of each coordinate of $x_k$ is  bounded below by $|x|/2$, while for any $t>0$ there exists a $K$ sufficiently large such that for any $k>K$ the absolute value of each coordinate of $\urho'_k$ is less than $t$. Given any simple oriented loop $\gamma\subset\Gamma$, let $\gamma_1\subset\Gamma_1$ be the loop obtained by contracting those edges that are not in $P_1$. Then equation (2) of the flow Kirchhoff problem on $\Gamma$ reads
\begin{equation}\label{eq:2g}
 \sum_{e\in\gamma_1} c_{e,k} x_{e,k}+\sum_{e'\in \gamma\setminus\gamma_1} c_{e',k}\rho_{e',k}=0.
\end{equation}
Let $\lbrace\tilde c_{e,k}: e\in E(\Gamma_1)\rbrace$ be the solution of the flow Kirchhoff problem on $\Gamma_1$ with inflows $f_{\ell,k}$ and resistances $x_k$. Then condition (2) of the Kirchhoff problem gives
\begin{equation}\label{2g1}
  \sum_{e\in\gamma_1} \tilde c_{e,k} x_{e,k}=0.
\end{equation}
From~\eqref{eq:2g} and~\eqref{2g1}) it follows that
\begin{equation}\label{eq:2g2}
\sum_{e\in\gamma_1} (c_{e,k}-\tilde c_{e,k})\, x_{e,k}=-\sum_{e'\in \gamma\setminus\gamma_1} c_{e',k}\rho_{e',k}.
\end{equation}
The set $(c_{e,k}-\tilde c_{e,k})$ for $e\in E(\Gamma_1)$ is the solution of the general Kirchhoff problem on $\Gamma_1$ with
the electromotive force defined by the right hand side of \eqref{eq:2g} and inflows at every vertex
$v\in \Gamma_1$ such that $v=v(e), |e|\notin P_1$ equal to $\sum_{e\notin P_1, v(e)=v} c_{e,k}$.
Let $d_{e,k}'$ and $d_{e,k}''$ be the solutions of the corresponding flow and force problems. Since $|c_{e,k}|\le 1/2\sum_\ell |f_{\ell,k}|$ we can use \eqref{C-bound} and \eqref{C-bound1} to conclude that there is a constant $M$ such
that for any $e\in E(\Gamma_1)$
\begin{equation}
  |c_{e,k}-\tilde c_{e,k}|< Mt|x|^{-1}
\end{equation}
Since $t>0$ could be chosen arbitrary, and the rest of the right-hand-side is a constant, this implies that as $k\to\infty$, the solutions $c_{e,k}$ and $\tilde c_{e,k}$ on the edges of $\Gamma_1$ have the same limit. By the inductive assumption, the solutions of the Kirchhoff problem on $\Gamma_2$ converge to the solutions of the multi-scale problem on $\Gamma_2$ with the additional inflows equal $\tilde c_{e\in \Gamma_1}$.
\end{proof}

\section{Notation for RN differentials and moduli of jet curves}\label{sec:notation}
We follow the (slightly adjusted, in anticipation of~\cite{grkrcusps}) notation and the real-normalized differentials setup of~\cite{grkrwhitham,grkrcm}, which we now review.
\begin{df}
The {\em singular part} of a meromorphic differential at a point $p$ on a Riemann surface $C$ is the equivalence class of meromorphic differentials $\omega$ in a neighborhood of $p$, with the equivalence $\omega\sim \omega'$ if and only if $\omega'-\omega$ is holomorphic at $p$.
\end{df}
\begin{df}\label{def:M}
For $m_1,\dots,m_n\in\ZZ_{\ge 0}$ we denote $\calM_{g,n}^{m_1,\dots,m_n}$  the moduli space of smooth genus $g$ complex curves $C$ with $n$ distinct labeled marked points $p_1,\dots,p_n$ together with a singular part $\sigma_\ell$ of a meromorphic differential with pole of order exactly $m_\ell+1$ at each point $p_\ell$, such that each residue $r_\ell$ is purely imaginary, and the sum of all residues is equal to zero.

We similarly denote $\calM_{g,n}^{\le m_1,\dots,\le m_n}$ the moduli space where each $\sigma_\ell$ is a singular part of order {\em up to} $m_\ell+1$, and at least one of the singular parts is non-zero, with the same condition on the residues.

We call points of $\calM_{g,n}^{m_1,\dots,m_n}$ or of $\calM_{g,n}^{\le m_1,\dots, \le m_n}$ smooth {\em jet curves}.
\end{df}
We will always denote jet curves $X$, with $C$ denoting the underlying smooth curve. To keep the notation manageable, we will always suppress the marked points in our notation for curves and families of curves.

The reason for the name of a jet curve is that the datum of a singular part is equivalent to the datum of a jet of a local coordinate, in which the meromorphic differential can be written in the standard form as $(z^{-m}+rz^{-1})dz$.

We think of $\calM_{g,n}^{m_1,\dots,m_n}\subset \calM_{g,n}^{\le m_1,\dots,\le m_n} $ as fibrations over $\calM_{g,n}$ with fibers $\prod_\ell(\CC^{m_\ell}\setminus\CC^{m_\ell-1})\times \RR^{n-1}$ and $\left(\CC^{\sum m_\ell}\times\RR^{n-1}\right)\setminus\lbrace 0,0\rbrace$, respectively.

\smallskip
As easily follows from the positive-definiteness of the imaginary part of the period matrix, for any $X\in\calM_{g,n}^{\le m_1,\dots,\le m_n}$ there exists a unique meromorphic differential $\Psi_X\in H^0(C,K_C+\sum (m_\ell+1) p_\ell)$ with prescribed singular parts $\sigma_\ell$ at $p_\ell$, with residues $r_\ell\in\mathbb R$ at $p_\ell$, and with all periods real.
\begin{df}
For any $X\in\calM_{g,n}^{\le m_1,\dots,\le m_n}$ we call $\Psi=\Psi_X$ the associated {\em real-normalized (RN for short) differential}.
\end{df}
\begin{rem}
The datum of a real-normalized differential is equivalent to the datum of the harmonic function $F(p):=\Im\int^p\Psi$ on the punctured Riemann surface $C\setminus\{p_1,\dots,p_n\}$, defined up to an additive constant. Indeed, given any such harmonic function $F$, the RN differential is given by $d(F^*+iF)$, where $F^*$ denotes the harmonic conjugate function to $F$.
\end{rem}
\begin{notat}
From now on, we will fix $g,n,m_1,\dots,m_n$, with all $m_\ell\geq 0$, and write simply $\calM$ for $\calM_{g,n}^{\le m_1,\dots,\le m_n}$.
\end{notat}
Since in the Deligne-Mumford compactification the marked points on stable nodal curves are not allowed to coincide with the nodes, the holomorphic fibration $\calM\to\calM_{g,n}$ extends to a holomorphic fibration over the Deligne-Mumford compactification $\overline\calM_{g,n}$, which we denote $\overline\calM$. We will call $X\in\overline\calM$ {\em stable jet curves}.

\section{Statement of results: limits of RN differentials}\label{sec:resultslimits}
Our first goal is to give the precise statement of the theorem on limits of RN differentials --- this will be theorem~\ref{thm:limits}, which is the precise version of theorem~\ref{thm:limitsrough}. As the RN differential does not depend holomorphically on the moduli, we work with sequences of smooth curves degenerating to a stable curve, rather than with algebraic families of smooth curves degenerating to a stable one.

\begin{notat}
For a stable curve $(C,p_1,\dots,p_n)\in\overline\calM_{g,n}$, its {\em dual graph} $\Gamma$ has vertices $v$ that correspond to {\em normalizations} $C^v$ of irreducible components of $C$,  edges $|e|$ that correspond to nodes $q_{|e|}$ of $C$, oriented edges $e$ that correspond to preimages $ q_e$ of the nodes (as points on the normalization $\tilde C$ of $C$), and legs $\ell$ that correspond to the marked points $p_\ell$, attached to the vertex $v$ such that $p_\ell\in C^{v}$. So $E_v$ is the set of all preimages of the nodes that are contained in $C^v$, and $q_e$ and $q_{-e}$ are the two preimages on $\tilde C$ of a node $q_{|e|}$ of $C$.
\end{notat}
\begin{notat}
From now on, we always work with a sequence $\lbrace X_k\rbrace\subset \calM$ of smooth jet curves such that $X_k$ converge as $k\to \infty$ to some stable jet curve $X\in\partial\overline\calM$. We denote $\lbrace C_k\rbrace\subset\calM_{g,n}$  the underlying sequence of smooth curves with distinct marked points (which, recall, we systematically suppress in notation), which then must converge to the stable curve $C\in\partial\overline\calM_{g,n}$ underlying the stable jet curve $X$.

The limit $\Psi:=\lim_{k\to \infty}\Psi_{X_k}$, if it exists, we call the {\em limit RN differential} in such a sequence.  By abuse of notation we will speak of the singularities of differentials at points $p_\ell\in C_k$ without labeling the dependence of $p_\ell$ on $k$. We will write $\Psi_k$ for $\Psi_{X_k}$.
\end{notat}
If the limit RN differential $\Psi$ exists in a given sequence, then $\Psi\in H^0(C,\omega_{C}(\sum (m_\ell+1) p_\ell))$ is a RN differential. This is to say, $\Psi$ has prescribed singularities at every point $p_\ell\in C$, and has at most simple poles at the nodes of $C$, with opposite residues. By abuse of notation, we denote by $\Psi^v$ the pullback to the normalization $C^v$ of the restriction of $\Psi$ to the corresponding irreducible component of $C$. Thus each $\Psi^v$ is determined uniquely by its residues at $q_e$ for all $e\in E_v$, and its singular parts at those $p_\ell$ that are contained in $C^v$.

In~\cite[Sec.~5]{grkrwhitham} we showed that limits of RN differentials whose only singularity is one double pole do not develop residues at the nodes of the stable curve. The proof applies verbatim to the case of RN differentials with a single pole of arbitrary order, and by $\RR$-linearity (of the dependence of the RN differential on its singular parts) it further extends to the general case of any differential without residues, i.e.~``of the second kind'' in classical terminology, giving the following result:
\begin{thm}[\cite{grkrwhitham}]\label{thm:2ndkindlimit}
If all the residues $r_\ell(X_k)$ are zero, then the limit RN differential $\Psi$ exists in any degenerating sequence $\lbrace X_k\rbrace\to X$, and on any $C^v$ the restriction $\Psi^v$ of the limit RN differential is the RN differential on $C^v$ with prescribed singular parts at those marked points $p_\ell$ that lie on $C^v$, and no other singularities, including at the nodes.
\end{thm}
\begin{rem}
This statement is a priori surprising, as for example it follows that the limit RN differential of the second kind is identically zero on any $C^v$ that contains no marked points. This is clearly false for general $r_\ell$, as one sees by considering the case of $n=2$, with two simple poles that are on different components of the stable curve: then by the residue theorem there must appear a simple pole at some node between these components.
\end{rem}
For limits of RN differentials ``of the third kind'' --- that is, with arbitrary residues $r_\ell$ --- one can easily see that the residues of the limit RN differential may depend on the degenerating sequence, and our main theorem on limit RN differentials is a necessary condition for existence of a limit RN differential, and the determination of its residues. Such an explicit construction is not available in the literature for the closures of the strata studied in~\cite{gendron,chen,fapa,strata}.

\medskip

We work in plumbing coordinates near the boundary of the moduli space, which are recalled and discussed in detail in section~\ref{sec:plumbing}. To state the results, recall that the plumbing parameter $s_{|e|}$ corresponds to locally opening up the node $q_{|e|}\in C$ given in local coordinates by $xy=0$ to $xy=s_{|e|}$. The plumbing parameters for every node, together with the coordinates on the moduli space where $\tilde C$ lives give local coordinates near the boundary point $C$ of $\overline\calM_{g,n}$. Since $\overline\calM$ is a bundle over $\overline\calM_{g,n}$, local coordinates near $X\in\partial\overline\calM$ are given by the local coordinates on $\overline\calM_{g,n}$ near $C$, together with local coordinates for the fiber of $\overline\calM\to\overline\calM_{g,n}$.

\begin{df}
The {\em log plumbing coordinates} of a smooth point $C'$ in a neighborhood of $C\in\partial\overline\calM_{g,n}$ is the point $\urho(C'):=\lbrace -\ln |s_{|e|}(t)|\rbrace\in \RR_+^{\# |E|}$. The {\em projectivized log plumbing coordinates} of $C'$ is the point $\PP\urho(C')\in\RR_+^{\# |E|}/\RR_+=S_{>0}^{\#|E|-1}$.
\end{df}
Recall that in Lemma~\ref{lm:multiKirchhoff} we proved that convergence of projectivized resistances in the blowup $\PS[\#|E|-1]$ of $S_{\ge 0}^{\#|E|-1}$ implies convergence of solutions of the flow Kirchhoff problem; we thus make the following.
\begin{df}\label{df:admissible}
A sequence $\lbrace C_k\rbrace \subset\calM_{g,n}$ converging to $C\in\partial\overline\calM_{g,n}$ is called {\em admissible} if there exists a limit $\PP\urho:=\lim_{k\to \infty}\PP\urho(C_k)\in\PS[\#|E|-1]$ of the projectivized log plumbing coordinates $\PP\urho(C_k)$ of $C_k$ as $k\to \infty$.
The point $\PP\urho$ is then called the {\em rates of degeneration} of the sequence $\lbrace C_k\rbrace$.
\end{df}
Our main result on limit RN differentials is that their residues are given by limits of solutions of the flow Kirchhoff problem, which by lemma~\ref{lm:multiKirchhoff} is the solution of the multi-scale Kirchhoff problem.
\begin{thm}\label{thm:limits}
Let $\lbrace X_k\rbrace \subset \calM$ be a sequence of smooth jet curves converging to a stable jet curve $X$. Then the limit RN differential $\Psi=\lim_{k\to\infty}\Psi_{X_k}$ exists if and only if the solutions $c_{e,k}$ of the flow Kirchhoff problems with inflows $ir_{\ell,k}$ and resistances $\urho_k$ converge. If the limit RN differential exists, then on any $C^v$ the limit $\Psi^v$ is the RN differential with prescribed singularities at the marked points $p_\ell$ contained in $C^v$, and with simple poles of residue $i$ times the limit of the solutions of the flow Kirchhoff problem.
\end{thm}
This theorem will be proven in section 6.

Since we have studied the limits of solutions of the Kirchhoff problem in section~\ref{sec:kirchhoff}, lemma~\ref{lm:multiKirchhoff} implies the following
\begin{cor}
If the sequence $\lbrace X_k\rbrace$ is admissible with rates of degeneration $\PP\urho\in\PS[\#|E|-1]$, then the limit RN differential $\Psi$ exists, and its residues are given by the solution of the multi-scale Kirchhoff problem with generalized resistances $\PP\urho$, for each $e\in E_v$.
\end{cor}
\begin{rem}
For the case of limit RN differentials of the second kind, all $f_\ell$ are zero, and thus for any $\PP\urho_k$ the set of all currents $c_e=0$ is the unique solution of the flow Kirchhoff problem, so that the solutions converge for any sequence, and in the limit all currents are still equal to zero. In particular, in this case the limit is the same in all admissible sequences; indeed, while convergence of resistances in $\PS[\#|E|-1]$ implies convergence of solutions of the Kirchhoff problem, many such limits may be the same. We do not claim that $\PS[\#|E|-1]$ is the minimal blowup of $S_{\ge 0}^{\#|E|-1}$ onto which the solutions of the flow Kirchhoff problem extend continuously.
\end{rem}
\begin{rem}
The special case of this theorem when $C$ is geometric genus zero (i.e.~each $C^v$ is a rational curve), and the rates of degeneration lie in $S_{>0}^{\#|E|-1}$ (i.e.~no blowup to $\PS[\#|E|-1]$ is necessary), was studied by Lang~\cite{lang}, who obtained for this case a version of this theorem, from a completely different viewpoint and with very different methods.
\end{rem}
\begin{rem}
For degenerating algebraic 1-parameter families such as used in~\cite{strata}, each plumbing coordinate $s_e$ has the form $t^{n_e}$ for some integer $n_e>0$, and thus any subsequence of such a family is admissible, with rates of degeneration $\PP\underline{n}\in S_{>0}^{\#|E|-1}$. In particular, for such an algebraic family, there is no need to blow up the sphere. 
\end{rem}
\begin{rem}\label{rem:meanloopvoltage}
The meaning of condition (1) of the Kirchhoff problem in terms of differentials is clear: it serves to ensure that the residue theorem is satisfied for each $\Psi^v$. The meaning of condition (2) is less transparent. In fact if a collection of RN differentials $\Psi$ on the components $C^v$ were to have arbitrary residues at the nodes, the imaginary parts of its periods over cycles passing through the nodes will diverge logarithmically, as computed in lemma~\ref{periodsofpsi}. Condition  (2) is precisely to guarantee that the logarithmic divergences cancel, so that the imaginary parts of periods of the limit RN differential on the singular stable curve $C$  are {\it finite}.
\end{rem}
Since the space $\PS[\#|E|-1]$ is compact, it follows that all possible limit RN differentials on $C$ are obtained this way.
\begin{lm}\label{lm:subseq}
Any sequence of smooth jet curves $\lbrace X_k\rbrace$ converging to a stable jet curve $X\in\partial\overline\calM$ contains an admissible subsequence.
\end{lm}
\begin{proof}
The space $\PS[\#|E|-1]$ is compact, and thus the sequence $\lbrace\PP\urho(C_k)\rbrace\subset\PS[\#|E|-1]$ must contain a convergent subsequence, which by definition corresponds to an admissible sequence of smooth curves.
\end{proof}
\begin{prop}\label{prop:compactness}
Let $\lbrace X_k\rbrace \subset \calM$ be a sequence of smooth jet curves converging to a stable jet curve $X$. If the limit RN differential exists, it is given by a collection of RN differentials on $C^v$ with prescribed singularities  at $p_\ell$ and with the residue at $q_e$ being $i$ times the solution of the multi-scale flow Kirchhoff problem for some $\PP\urho\in\PS[\#|E|-1]$.
\end{prop}
\begin{proof}
By lemma~\ref{lm:subseq}, the sequence $\lbrace X_k\rbrace$ must contain an admissible subsequence, with resistances converging to some $\PP\urho\in\PS[\#|E|-1]$. By lemma~\ref{lm:multiKirchhoff} in such a subsequence there exists a limit of solutions of the flow Kirchhoff problem, and it is given by the solution of the multi-scale Kirchhoff problem with resistances $\PP\urho$. Finally, by theorem~\ref{thm:limits}, the convergence of solutions of the flow Kirchhoff problem implies the existence of the limit RN differential in this subsequence, of the form claimed. Since the limit RN differential is assumed to exist for all of the original sequence, it must be of the form claimed.
\end{proof}

This completes the statement of our results on limits of RN differentials. The full details and statements of our results on limits of zeroes of RN differentials will be given in section~\ref{sec:zeroes}, after the main technical tool of solving the \RHP is introduced.

\section{Plumbing setup for Riemann surfaces}\label{sec:plumbing}
We now recall the full details of the plumbing construction discussed in the introduction, and fix the notation that will be used throughout the rest of the paper and in all the proofs.
\begin{df}[Standard plumbing]\label{df:stdplumbing}
Let $q_1,q_2\in C$ (with $C$ a possibly disconnected Riemann surface) be two distinct points. Let $z_1,z_2$ be local coordinates on $C$ near $q_1,q_2$ such that $z_j(q_j)=0$ and furthermore sufficiently small so that the maps $z_j$ embed the unit disk in the complex plane as disjoint neighborhoods  $V_j:=\lbrace |z_j|<1\rbrace\subset C$ of $q_j$. Then for any $s\in\CC$ with $|s|<1$ we denote $U_j=U_j^s:=\lbrace |z_j|<\sqrt{|s\,|}\rbrace\subset V_j$ the corresponding disks, and  denote $\gamma_j:=\partial U_j$ their boundary circles, which we orient negatively with respect to $U_j$. The {\em standard plumbing $C_s$ with parameter $s$} is the Riemann surface
$$
C_{s}:=\left[C\setminus (U_{1}\sqcup U_{2})\right] /( \gamma_{1}\sim\gamma_{2})
$$
where $\gamma_{1}$ is identified with $\gamma_{2}$ via the diffeomorphism $I(z_1):=s/z_1$. The structure of a Riemann surface on $C_s$ is defined by saying that a function on $C_s$ is holomorphic, if it is holomorphic on the complement of the {\em seam} $\gamma$ (the image of $\gamma_1$ and $\gamma_2$) and continuous along the seam.
\end{df}

\begin{df}[Plumbing coordinates on moduli]\label{df:plumbing}
Local plumbing coordinates on $\overline\calM_{g,n}$ near a stable curve $C\in\partial\overline\calM_{g,n}$ are defined as follows. Let $\tilde C$ be the normalization of $C$, which is a smooth (possibly disconnected) Riemann surface with marked points $p_\ell$, and also with all the preimages of the nodes as marked points.

We think of $\tilde C$ as a point in a suitable Cartesian product of moduli spaces of curves with marked points. Let $u=(u_1,\dots,u_x)$ be some local coordinates on this product of moduli spaces; we write $\tilde C_u$ for the (possibly disconnected) curve in this moduli space with coordinates $u$, so that the coordinates of $\tilde C$ are all $u_i=0$. Choose, for all $u$ sufficiently small, a holomorphically varying family of local coordinates $z_e$ in the neighborhood on $\tilde C_u$ of every preimage $q_e$ of every node of $C_u$, scaled (by dividing by a large real number) to be sufficiently small so that the unit disks in these coordinates are all disjoint on $\tilde C_u$.

Then $u$ together with a set of plumbing parameters $\us:=\lbrace s_{|e|}\rbrace\in\Delta^{\# |E|}$, for $\Delta\subset\CC$ a sufficiently small disk, give local coordinates on $\overline\calM_{g,n}$ near $C$ (see~\cite{bers}).
\end{df}
\begin{rem}
Different versions of plumbing are available in the literature. First of all, one usually considers the neighborhoods $V_j=\lbrace |z_j|<\epsilon\rbrace$ in the local coordinates, for some sufficiently small $\epsilon$; by rescaling $z_j$ by a real number this is of course equivalent to our setup.

The plumbing that we use, by identifying the boundaries of two cut out disks directly, is perhaps the earliest one, going back to \cite{bers}. It is clearly seen to be equivalent to cutting out closed disks of radii $|s|$ around $q_j$, and then identifying two boundary annuli in this open Riemann surface: if one has identified along the annuli, then one can alternatively cut the glued surface along the middle circle of the resulting glued annulus, and switch to our viewpoint. In \cite{acgh2},\cite{wolpertplumbing}, \cite{strata}, plumbing using a plumbing fixture is performed --- which is the analytic description of the algebraic versal deformation coordinates. This third kind of plumbing can also easily be seen to be equivalent to the original version that we use, by cutting the plumbing fixture $xy=t$ along the circle $|x|=|y|$. The advantage of the approach using a fixed plumbing fixture is the ability to see explicitly the algebraic structure of the degenerating family of Riemann surfaces as the node forms, and to interpret plumbing coordinates as versal deformations of nodal curves.
\end{rem}
\begin{rem}\label{rem:meaninglimit}
The version of plumbing that we use is most suited to understanding limits of 1-forms under degeneration. Indeed, in our setup if $\lbrace C_k\rbrace$ is a sequence of smooth Riemann surfaces converging to $C$, then each $C_k$ is obtained by identifying the boundaries of a subset of $C$. Thus we can interpret a sequence $\Phi_k$ of meromorphic differentials on $C_k$ as a sequence of differentials on a sequence of growing subsets of $C$, tending to all of $C$ as $k\to\infty$. Thus the limit $\lim_{k\to\infty}\Phi_k$, if it exists, automatically makes sense as a collection of meromorphic differentials $\Phi^v$ on the irreducible components $C^v$ of $C$.
\end{rem}

To keep the notation manageable, we will write $\tilde C:=\tilde C_{u,\underline{0}}$ for the normalization of a nodal curve, and will consistently drop $u$ when no confusion is possible. We write $\tilde C$ as the union of its connected components $\tC=\cup \,C^v$ indexed by vertices $v\in V(\Gamma)$ of the dual graph $\Gamma$ of $C$. Recall that we write $e$ for an {\em oriented} edge of $\Gamma$ and $|e|$ for the unoriented edge. We write $q_e\in C^{v(e)},\ v(e):=\operatorname{target}(e)$ for the corresponding preimage $q_e$ of the node $q_{|e|}=q_e\sim q_{-e}$ of $C$. To simplify notation we will also write $s_e=s_{-e}=s_{|e|}$.
We then write $\widehat C_{\us}:=\widetilde C\setminus (\cup_e U_e^{s_e})$ for the closed Riemann surface with boundary obtained by removing these open disks from $\tilde C$. Identifying for each $|e|\in|E|(\Gamma)$ the boundaries $\gamma_e^{s_e}$ and $\gamma_{-e}^{s_e}$ of $\widehat C_{\us}$ via the map $I_e$ sending $z_e$ to $s_e/z_e$ gives precisely the plumbed Riemann surface $C_{u,\us}$. When speaking of a one-form $\omega$ on $\tC$ or on $\widehat C$, we mean a collection of one-forms $\omega^v$ on the set of connected components of $\tC$ or of $\widehat C$.

Since $\overline\calM$ is the total space of a fibration over $\overline\calM_{g,n}$, local coordinates on it near some $X\in\partial\overline\calM$ are given by $u,\us$, together with some local coordinates $w$ for the fiber of the fibration. We will thus write a stable jet curve with these coordinates as $X_{w,u,\us}$, with $C_{u,\us}$ as the underlying stable curve.

A meromorphic differential $\Phi$ on $\widehat C_\us$ (which, recall, is the shorthand for a collection of meromorphic differentials $\Phi^v$ on $\widehat C_\us^v$) glues to define a meromorphic differential on $C_\us$ if and only if
\begin{equation}\label{eq:nojump}
\left.\Phi^{v(e)}\right|_{\gamma_e}=I_e^*\left(\left.\Phi^{v(-e)}\right|_{\gamma_{-e}}\right)
\end{equation}
for all $e$.
\begin{rem}
Of course not every differential $\Phi$ on $\widehat C_\us$ satisfies~\eqref{eq:nojump} and glues to a differential on $C_\us$. One standard setup is for differentials with simple poles at preimages of the nodes, with opposite residues. Choosing coordinate $z_e$ near $q_e$ such that locally $\Phi^{v(e)}=a_e dz_e/z_e$, with $a_e=-a_{-e}$, and performing plumbing in these coordinates, one constructs a glued differential on $C_\us$. More generally, one can choose standard coordinates associated to a differential to glue a zero of order $k$ to a pole of order $k+2$ with no residue, as discussed and applied in \cite{gendron,chen,strata}. As a result, and with much further work to deal with the residues appearing, one constructs a meromorphic differential on some smooth Riemann surface $C_\us$ near $C$, in plumbing coordinates. However, since the local coordinates $z_e$ depend on the differential, it is hard to ensure from this viewpoint that all suitable differentials on all smooth Riemann surfaces in a neighborhood can be obtained in this way.
\end{rem}
Our approach is direct and analytic. We start with any collection of fixed local coordinates $z_e$ near $q_e$ (for any $u$), and thus with fixed plumbing coordinates on the moduli space. Given any $\Phi$ on $\widehat C_\us$, we will subtract from it another differential $\omega$ on $\widehat C_\us$ such that their difference satisfies~\eqref{eq:nojump}, and thus defines a differential on $C_\us$. The condition for $\omega$ must then be that its ``jumps'' on $\gamma_e$ are the same as for $\Phi$, and we construct it by explicitly solving the \RHP.

\section{The \RHP}\label{sec:rhp}
Given a compact Riemann surface with a collection of closed loops in it, the \RHP is the problem of constructing a holomorphic differential on the complement of these loops, such that it extends continuously to each loop from the two sides, and its boundary values there differ by a prescribed jump. Equivalently, we think of the \RHP as posed on a Riemann surface with boundary, where the boundary components are identified pairwise, and the solution of the \RHP is a differential on the interior that extends continuously to the boundary, and such that the differences of its boundary values are the prescribed ``jumps". The classical approach to solving the \RHP on is surveyed in~\cite{zverovich}, and explained in full detail in~\cite{rodin}. The \RHP is solved by integrating the jumps with respect to the Cauchy kernel on the Riemann surface.

We are interested in constructing RN differentials in plumbing coordinates, and thus in solving the \RHP on $C_\us$. Since the Cauchy kernel on $C_\us$ depends on $\us$, determining the behavior of the solution of the \RHP under degeneration as $\us\to 0$ is hard, and has not been accomplished in the literature. Instead, we use the Cauchy kernel on the normalization $\tilde C$ of the nodal curve, to construct differentials on $\tilde C$ with prescribed jumps along the seams, considered as closed loops on $\tilde C$. By an explicit control of the constructed solution of the \RHP in the small neighborhoods of the nodes, we can then correct this solution, in an iterative way, to eventually construct the desired solution of the \RHP on $C_\us$. As this only uses the Cauchy kernel on $\tilde C$, which is independent of $\us$, we can determine the behavior of the solution under degeneration. Our interest in the current paper is in  solving the \RHP to construct RN differentials; our method was then used by Hu and the third author in~\cite{hunorton} to study a normalized basis of differentials, which turns out to be easier as holomorphic dependence on parameters can be used.

Throughout this section, we will only work on smooth jet curves, i.e.~all $s_e$ are always assumed to be non-zero. For convenience, we denote $|\us\,|:=\max_e |s_e|$.

The \RHP is an additive analog of the (multiplicative) Riemann-Hilbert problem posed on a Riemann surface $\widehat C_{u,\us}$ with $\#E$ boundary components. The initial data for the \RHP is a set $\uphi$ of complex-valued smooth 1-forms $\phi_e$ on $\gamma_e$, which we call {\em jumps}. The jumps are required to satisfy $\phi_e=-I_e^*(\phi_{-e})$ and  $\int_{\gamma_e}\phi_e=0$ for all $e\in E$.
\begin{df}\label{prob:rhp}
The {\em \RHP} is to find a holomorphic 1-form $\omega$ on the interior of $\widehat C_{u,\us}$ that extends continuously to every boundary component $\gamma_e$ of $\widehat C_{u,\us}$, and such that the boundary extensions have jumps $\phi_e$, i.e.~satisfy for any $e$ the equation
$$
  \omega|_{\gamma_e}-I_e^*\left(\omega|_{\gamma_{-e}}\right)=\phi_e.
$$
\end{df}
Equivalently, the \RHP is the problem of constructing differentials on $C_{u,\us}$ continuous away from the seams $\gamma_{|e|}$ and with prescribed differences of boundary values on the two sides of each seam.

The solution of the \RHP is never unique: the pullback to $\widehat C_{u,\us}$ of any holomorphic differential on $C_{u,\us}$ has zero jumps, and can be added to any solution to produce another solution. Our main technical tool is an explicit construction of a suitably normalized solution, which we will call ARN, with explicit bounds for it.
\begin{prop}\label{prop:RH}
There exists a constant $t$ independent of $u$, such that for any $|\us|<t$ and any $\uphi$, the \RHP has a unique solution $\omega$ on $\widehat C_{u, \us}$ satisfying
\begin{itemize}
\item $\int_{\gamma_e}\omega=0$ for any $e$;
\item $\int_\gamma\omega\in\RR$ for any cycle $\gamma\in H_1(\tC_{u},\ZZ)$.
\end{itemize}
The solution $\omega$ is given explicitly as the sum $\xi+\chi$, with $\xi$ and $\chi$ being the restrictions to $\widehat C_{u,\us}$ of the integrals~\eqref{int},\eqref{int1}, where the smooth real 1-forms $h_e,g_e$ on $\gamma_e$ are defined as the sums of the series~\eqref{rec}.
\end{prop}
We will call this $\omega$ the {\em almost real-normalized (ARN)} solution of the \RHP.
Note that the condition on $\gamma$ in the second statement is equivalent to taking $\gamma\in H_1(C_{u,\us},\ZZ)$ not intersecting any of the seams.
The importance of the proposition is the explicit construction, which will eventually allow us to give estimates for the ARN solution as the curve degenerates.
\begin{proof}
The proof of this proposition will occupy the bulk of this section.

\subsubsection*{Uniqueness of the ARN solution}
Suppose $\omega_1$ and $\omega_2$ were two different ARN solutions of the \RHP with the same initial data. Then $\omega:=\omega_1-\omega_2$ would have zero jumps, and would thus be a holomorphic 1-form on $C_{u,\us}$ with zero integral over any seam $\gamma_e$, and with real integrals over any path contained in $\widehat C^v_{u,\us}$.

To deduce that $\omega$ is identically zero we use the Stokes' theorem on each $\widehat C^v_{u,\us}$, and then sum over the components (a similar method will be used again later on). Choose an arbitrary point $p_0\in C^v$; then $F^v(p):=\Im\int_{p_0}^p\omega|_{C^v}$ is a single-valued real harmonic function on $\widehat C^v_{u,\us}$, since all the periods of $\omega^v$ are real. The harmonic conjugate function $F^{v*}(p)=\Re\int_{p_0}^p\omega$ is multiple-valued, but locally defined up to an additive constant, and thus we can still write $\omega^v=dF^{v*}+idF^v$. We use Stokes' theorem to compute the $L^2$ norm of $\omega^v$ on $C^v$:
\begin{equation}\label{stokes}
\begin{aligned}
\frac{i}{2}\int_{\widehat C^v_{u,\us}}\omega^v\wedge\bar\omega^v&=\frac{i}{2}\int_{\widehat C^v_{u,\us}}(dF^{v*}+idF^v)\wedge(dF^{v*}-idF^v)\\ &=\int_{\widehat C^v_{u,\us}}dF^{v*}\wedge dF^v=-\sum_{e\in E_v}\int_{\gamma_e}F^vdF^{v*},
\end{aligned}
\end{equation}
where we have used the fact that $F^v$ is a well-defined single-valued function on $\widehat C^v_{u,\us}$ and that the boundary of $\widehat C^v_{u,\us}$ is the collection of $\gamma_e$ for all $e\in E_v$.

We now take the sum of these equalities over all $v$; the summands on the right come in pairs $\lbrace e,-e\rbrace$. Since $\omega$ is a holomorphic differential on the plumbed surface $C_{u,\us}$, the restrictions $\left.F^{v(e)}\right|_{\gamma_e}$ and $I_e^*\left(\left.F^{v(-e)}\right|_{\gamma_{-e}}\right)$ on any seam differ by some constant of integration $C_e$, while the restrictions of the differentials $\left.dF^{v(e)*}\right|_{\gamma_e}$ and $\left.dF^{v(-e)*}\right|_{\gamma_{-e}}$ are equal under pullback by $I_e$. We thus compute
$$
 \int_{\gamma_e}F^{v(e)}dF^{v(e)*}+\int_{\gamma_{-e}}F^{v(-e)}dF^{v(-e)*}=C_e\int_{\gamma_e}dF^{v(e)*}\,,
$$
where we recall that the map $I_e:\gamma_e\to\gamma_{-e}$ is orientation-reversing. Since $\int_{\gamma_e}\omega^{v(e)}=0$ by the definition of the ARN solution, it follows that also $\int_{\gamma_e}dF^{v(e)*}=0$, and thus finally the sum of~\eqref{stokes} for all $v$ vanishes. Altogether it then follows that $\int_{C_{u,\us}}\omega\wedge\bar\omega=0$, which  implies that $\omega$ is identically equal to zero.

\subsubsection*{Construction of the ARN solution}
We first recall the notion of Cauchy kernels, then use the appropriate versions of the kernel to essentially deal with the real and imaginary parts of the initial data, and then construct the ARN solution by explicitly writing the inverse of the relevant integral operator as a sum of a convergent series.

\subsubsection*{The Cauchy kernels}
Recall that for a genus $g$ compact Riemann surface $\calC$ with a symplectic basis $\lbrace A_k,B_k\rbrace$ of $H_1(\calC,\ZZ)$, the normalized basis of the space of holomorphic differentials on $\calC$ is prescribed by the condition $\int_{A_k}\omega_j=\delta_{j,k}$.
Denote $\tau_{j,k}=\int_{B_k}\omega_j$  the period matrix of  $\calC$, and  $\theta(z)=\theta(\tau,z)$ the corresponding theta function. We denote by $A$ the Abel map of $\calC$ to $\CC^g$ sending $q_0$ to zero, and denote $\kappa$ the corresponding Riemann constant. Then for any sufficiently general collection of fixed $g$ points  $q_0,\dots, q_{g-1}\in \calC$ the {\em normalized Cauchy kernel} is defined as
\begin{equation}\label{cauchy-ker}
K_\calC(p,q):=\frac 1 {2\pi i}d_p\ln \frac{\theta(A(p)-A(q)-Z)}{\theta(A(p)-Z)}\,,\ \  Z:=\sum_{j=1}^{g-1}. A(q_j)+\kappa,
\end{equation}
where by $d_p$ we mean the exterior differential with respect to~$p$, for $q$ fixed, so that the result is a differential form in~$p$.
For $q$ fixed, $K_{\calC}$ is then a meromorphic differential in the variable $p$ with simple poles with residues $\pm (2\pi i)^{-1}$ at $p=q$ and at $p=q_0$, respectively. For $p$ fixed, $K_\calC$ is a multi-valued meromorphic function of $q$ with the only pole at $q=p$.
The Cauchy kernel $K_\calC$ is normalized in the sense that all its $A$-periods are zero: $\int_{p\in A_k}K_\calC(p,q)=0$. 
On the sphere, the Cauchy kernel becomes simply $K_{\CC\PP^1}(p,q)=\frac{1}{2\pi i}\frac{d p}{p-q}$.

The Cauchy kernel is used to solve the~\RHP, as we now recall (see~\cite{rodin}). Given a smooth closed simple curve $\gamma\subset\calC$ and a smooth 1-form $\phi$ on $\gamma$, the Cauchy integral transform $\int_{q\in\gamma}K_\calC(p,q)\phi(q)$ defines a holomorphic 1-form in the variable $p$, for $p\not\in\gamma$ (the holomorphicity for $p\ne q_0$ follows from the holomorphicity of the kernel $K_\calC$ in $p$, so that we can then differentiate under the integral sign, while the holomorphicity at $p=q_0$ follows from the fact that the residue of the Cauchy kernel is independent of~$q$, while $\int_{\gamma_e}\phi=0$). The Sokhotski-Plemelj formula is the statement that the boundary values of this expression for $z\in\gamma$ are
\begin{equation}\label{eq:SP}
  \lim_{p'\to p\in\gamma} \left(\int_{q\in\gamma}K_\calC(p',q)\phi(q)\right)=\pm\frac12 \phi(p)+\fint_{q\in\gamma}K_C(p,q)\phi(q),
\end{equation}
where we used the classical notation $\fint$ for the Cauchy principal value of the singular integral. Locally the neighborhood of $\gamma\subset \calC$ looks like an annulus with $\gamma$ being the middle circle, and then the limit is taken for $p'$ lying in a fixed component of the complement of $\gamma$ in this annulus. The sign in the Sokhotski-Plemelj formula is the orientation of the contour $\gamma$ as the boundary component of the corresponding half of the annulus. The Sokhotski-Plemelj formula implies that the integral transform with respect to the Cauchy kernel solves the \RHP.

\smallskip
To continue the construction of the RN differential, instead of the $A$-normalized Cauchy kernel we will now introduced the suitably real-normalized Cauchy kernel, defined as follows:
\begin{equation}\label{cauchy-re}
  K^{re}_\calC(p,q):=K_\calC(p,q)-\sum_{k=1}^g\alpha_k (q)\,\omega_k(p),
\end{equation}
where $\alpha_k$ are the coordinates of the vector $\alpha(q):=(\Im\tau)^{-1}\, (\Im A(q))$. For $q$ fixed, $K_\calC^{re}$ is also a meromorphic differential in $p$ with simple poles at $p=q$ and $p=q_0$ with residues $\pm (2\pi i)^{-1}$. For $p$ fixed, $K_\calC^{re}$ is a {\em single-valued real-analytic} function of the variable $q$, away from $q=p$. As can be easily checked from the monodromy properties of the theta function, all periods of $K_\calC^{re}$ are real: $\int_{p\in\gamma}K^{re}_{\calC}(p,q)\in\RR,\quad \forall \gamma\in H_1(\calC,\ZZ).$

Similarly, we introduce
\begin{equation}\label{cauchy-im}
K^{im}_\calC(p,q):=iK_\calC(p,q)-i\sum_{k=1}^g\beta_k (q)\,\omega_k(p),
\end{equation}
where $\beta_k$ are the coordinates of the vector $\beta(q):=(\Im\tau)^{-1}\, (\Re (A(q))$.
For $q$ fixed, $K_\calC^{im}$ is a meromorphic differential in $p$ with simple poles  at $p=q$ and $p=q_0$ with residues $\pm (2\pi )^{-1}$. For $p$ fixed, $K_\calC^{im}$ is a multi-valued real-analytic function of $q$ away from $q=p$. We note that $\Im \int_{p\in\gamma}K_\calC^{im}(p,q)\in \ZZ,\quad \forall \gamma\in H_1(\calC,\ZZ).$

Analogues of the Sokhotski-Plemelj formula hold for $K^{re}$ and $K^{im}$. For these analogues, in the right-hand-side of~\eqref{eq:SP} the Cauchy principal value of the integral transform of $\phi$ with respect to $K^{re}_\calC$ or $K^{im}_\calC$, correspondingly, should be taken, while for $K^{im}_\calC$, the $\phi$ is also multiplied by $i$. To prove this, note that $K^{re}_\calC$ and $K^{im}_\calC$ differ from $K_\calC$ by adding some differential $\omega(p,q)$ holomorphic in $p$ (while $K^{im}$ further multiplies by $i$), so that the Cauchy principal value of the integral transform of $\phi$ with respect to this added holomorphic differential is simply the value of the integral, which thus contributes $\int_{q\in\gamma} \omega(p,q)\phi(q)$ on both sides of Sokhotski-Plemelj formula~\eqref{eq:SP}.

\subsubsection*{The integral transform with respect to $K^{re}$}
We now  apply the Cauchy kernels to obtain a solution of the \RHP, using $K^{re}$ and $K^{im}$ to ensure the ARN condition.

For any set $\uh=\{h_e\}$ of real-valued smooth 1-forms on $\lbrace\gamma_e\rbrace$ such that $h_e=-I^*_e(h_{-e})$ for any $e$, and all periods are zero:
\begin{equation}\label{zero-period}
 \int_{q\in\gamma_e} h_e(q)=0,
\end{equation}
we define for $p\in C^v$ the Cauchy integral transform
\begin{equation}\label{int}
 \xi^v(p):=\sum_{e\in E_v} \int_{q\in \gamma_e} K^{re}_{C^v}(p,q) h_e(q).
\end{equation}
We emphasize that the formula above is an integral transform on the normalization $C^v$ of the irreducible component of $C_{u,0}$. In particular the kernel $K_v^{re}$ is {\em independent} of $\us$.
A priori $\xi^v$ can have a pole at $p=q_0$. However, since the residue of $K^{re}_v$ at $p=q_0$ is equal to $-(2\pi i)^{-1}$ for any $q$, the residue of $\xi^v$ at $p=q_0$ is equal to $-(2\pi i)^{-1} \sum_{e\in E_v}\int_{q\in \gamma_e}  h_e(q)$, which vanishes by~\eqref{zero-period}. We thus view $\xi^v$ as a collection of holomorphic differentials $\xi_{e}$ on each $U_e$ for $e\in E_v$, and the holomorphic differential $\widehat\xi^v$ on $\widehat C^v_\us$. We observe that for any closed path $\gamma\subset \widehat C^v$ the integral
$$
 \int_{p\in\gamma}\widehat\xi^v=\sum_{e\in E_v} \int_{p\in\gamma}\int_{q\in \gamma_e} K^{re}_{C^v}(p,q) h_e(q)=\sum_{e\in E_v} \int_{q\in \gamma_e} (\int_{p\in\gamma}K^{re}_{C^v}(p,q)) h_e(q)
$$
is real, since $\int_{p\in\gamma}K^{re}_{C^v}(p,q)\in\RR$l, and $h_e$ real-valued.

\smallskip
We now study the singular part of the kernel in more detail.
For $e,e'\in E_v$, $e\ne e'$ we write $z_e=z_e(p)\in V_e$ and $w_{e'}=w_{e'}(q)\in V_{e'}$ for the local coordinates in these disks, and denote
\begin{equation}\label{tildeK}
{\bf K}^{re}_v(z_e,w_{e'})\,dz_e:=K_v^{re}(p,q)\,.
\end{equation}
For $p,q\in V_e$, let ${\bf K}^{re}_v$ be the {\em holomorphic} part of $K^{re}_v$:
\begin{equation}\label{Cauchy-local}
{\bf K}^{re}_v(z_e,w_e)\,dz_e:=K^{re}_{v}(p,q)-\frac{dz_e}{2\pi i(z_e-w_e)}.
\end{equation}

We define the integral operator  $\calK^{re}$ by
\begin{equation}\label{calK}
  (\calK^{re}\uh)_e(z_e):=dz_e\sum_{e'\in E_{v(e)}} \int_{w_{e'}\in\gamma_{e'}} {\bf K}^{re}_{v(e)}(z_e,w_{e'})h_{e'}(w_{e'}).
\end{equation}
Since each ${\bf K}^{re}_v$ is holomorphic in $z_e$, the operator $\calK^{re}$ sends a collection $\uh$ of {\em real-valued} 1-forms on the set of seams $\gamma_e$ to a collection of {\em holomorphic} 1-forms on the disks $V_e$ --- again, the holomorphicity is simply due to the fact that each ${\bf K}^{re}_v$ is holomorphic in $z_e$. By definition, for any $z_e\in V_e$ we have
\begin{equation}\label{int_loc}
  \xi^v(z_e)=(\calK^{re}\uh)_e(z_e)+dz_e\int_{w_e\in\gamma_e}\frac{h_e(w_e)}{2\pi i(z_e-w_e)}
\end{equation}
The integral in this expression is singular, and its boundary values for $z_e\in\gamma_e$ are given by the Sokhotski-Plemelj formula~\eqref{eq:SP}:
\begin{equation}\label{eq:sokhot}
\widehat\xi^v(z_e)=(\calK^{re}\uh)_e(z_e)+\frac 12 h_e(z_e)+dz_e \fint_{w_e\in\gamma_e}\frac{h_e(w_e)}{2\pi i(z_e-w_e)}
\end{equation}
and
\begin{equation}
\xi_{e}(z_e)=(\calK^{re}\uh)_e(z_e)-\frac 12 h_e(z_e)+dz_e \fint_{w_e\in\gamma_e}\frac{h_e(w_e)}{2\pi i(z_e-w_e)}.
\end{equation}

Since $\xi_e$ is holomorphic on $U_e$, its integral over $\gamma_e=\partial U_e$ vanishes by the residue theorem. By using assumption~\eqref{zero-period}, we thus obtain
\begin{equation}\label{vanishing}
  \int_{p\in\gamma_e}\widehat\xi^v(p)=\int_{p\in\gamma_e}\xi_{e}(p)+\int_{p\in\gamma_e}h_e(p)=0
\end{equation}

To use this for solving the \RHP on $\widehat C_{\us}$, we need to compare $\widehat\xi^{v(e)}|_{\gamma_e}$ to $I^*_e(\widehat\xi^{v(-e)})|_{\gamma_e}$. Using $I^*_e(h_{-e})=-h_e$, and recalling that $I_e^*$ is orientation-reversing
we compute
\begin{equation}
\frac {dz_e}{2\pi i }\cdot \fint_{w_e\in\gamma_e}\frac{h_e(w_e)}{z_e-w_e}-\frac{d(sz_e^{-1})}{2\pi i} \cdot \fint_{w_e\in\gamma_{e}}\frac{h_e(w_e)}{sz_e^{-1}-sw_e^{-1}}=
\end{equation}
$$=\frac {dz_e}{2\pi i z_e} \int_{w_e\in\gamma_e} h_e(w_e)=0$$
for the jump of the singular integral in~\eqref{eq:sokhot}. Using again that $h_e=-I_e^*(h_{-e})$, we can write the jump of $\widehat\xi$ at a point $z_e\in\gamma_e$ as
\begin{equation}\label{eq:jumpxi}
\left(\widehat\xi^{v(e)}-I_e^*(\widehat\xi^{v(-e)})\right)(z_e)=\left(h_e+(\calK^{re}\uh)_e- I_e^*\left(\left(\calK^{re}\uh\right)_{-e}\right)\right)(z_e).
\end{equation}
We define the matrix-valued (with indices $e\in E$) operator $\underline\calK^{re}$ by
\begin{equation}\label{ucalK}
 (\underline\calK^{re}\uh)_e:=(\calK^{re}\uh)_e- I_e^*\left(\left(\calK^{re}\uh\right)_{-e}\right),
\end{equation}
and then reinterpret equation~\eqref{eq:jumpxi} as saying that the jump of $\widehat\xi$ across $\gamma_e$ is equal to
$\left[(\underline I+\underline\calK^{re})\uh\right]_e$, where $\underline I$ is the identity matrix.

\subsubsection*{The integral transform with respect to $K^{im}$}
We now perform an analogous construction starting from $K^{im}$, to deal with the imaginary part of the jumps while preserving the reality of the periods.
For any set $\ug$ of real-valued smooth 1-forms $g_e$ on the seams $\gamma_e$ such that $g_e=-I^*_e(g_{-e})$ and $\int_{\gamma_e}g_e=0$ for any $e$, we define the Cauchy integral
\begin{equation}\label{int1}
\chi^v(p):=\sum_{e\in E_v} \int_{q\in\gamma_e} K^{im}_v(p,q)g_e(q).
\end{equation}
Recall that, unlike $K^{re}_v$, the kernel $K^{im}_v$ is a multi-valued function of $q$. However,  we claim that $\chi^v$ is well-defined. Indeed, the difference of any two values of $K^{im}_v$ is some holomorphic differential $\omega(p,q)$, and from the definition of $K^{im}_v$ it follows that $\Im\int_{p\in\gamma} \omega(p,q)\in\ZZ$ for any cycle $\gamma\subset C^v$ and any $q$. Thus for $\gamma$ fixed and $q$ varying, $\Im\int_{p\in\gamma} \omega(p,q)\in\ZZ$  is locally constant in $q$, and so $\Im\int_{p\in\gamma}\partial_q \omega(p,q)=0$ for any $\gamma$. But then $\partial_q\omega(p,q)$, considered as a function of $p$, is a holomorphic real-normalized differential --- which is thus identically zero, so that $\omega(p,q)$ is independent of $q$. It then follows that $\int_{q\in\gamma_e}\omega(p,q)g_e(q)=\omega(p)\int_{q\in\gamma_e}g_e(q)=0$. Thus finally the multivaluedness of $K^{im}_v$ cancels in the definition of $\chi^v(p)$.

We again think of each $\chi^v$ as a collection of holomorphic 1-forms $\chi_e$ on each disk $U_e$ for $e\in E_v$, and a holomorphic 1-form $\widehat \chi^v$ on $\widehat C^v_\us$. The Sokhotski-Plemelj formula in this case yields for $z_e\in \gamma_e$
\begin{equation}\label{eq:sokhot1}
  \widehat\chi^v(z_e)=(\calK^{im}\ug)_e(z_e)+\frac i2 g_e(z_e)+dz_e \fint_{w_e\in\gamma_e}\frac{ig_e(w_e)}{2\pi(z_e-w_e)}
\end{equation}
and
\begin{equation}
  \chi_{e}(z_e)=(\calK^{im}\ug)_e(z_e)-\frac i2 g_e(z_e)+dz_e \fint_{w_e\in\gamma_e}\frac{ig_e(w_e)}{2\pi(z_e-w_e)}
\end{equation}

Similarly to the case of $\calK^{re}$, this  implies $\int_{\gamma_e}\chi^v=0$ for any $e\in E_v$, analogously to~\eqref{vanishing}. Moreover, although the periods of $K^{im}_v$ are not real, the period of $\widehat\chi^v$ over any cycle $\gamma\subset\widehat C^v$ is real by the same argument as above:
the imaginary part of any period of $K^{im}_v$ is an integer. This integer is then independent of $q$, and thus when it is multiplied by $g_e(q)$, the integral over $\gamma_e$ vanishes, since $\int_{q\in \gamma_e}g_e(q)=0$.

Similarly to~\eqref{eq:jumpxi}, we compute the jump of $\widehat\chi$ at $z_e\in\gamma_e$ to be
\begin{equation}\label{eq:jumpchi}
(\,\widehat\chi^{v(e)}-I_e^*(\widehat\chi^{v(-e)}))(z_e)=(ig_e+(\calK^{im}\ug)_e-I^*_e(\calK^{im}\ug)_{-e})(z_e),
\end{equation}
and we interpret the right-hand-side as the operator $i\underline I+\underline\calK^{im}$ applied to $\ug$.

\subsubsection*{The ARN solution as a recursively defined series}
We now combine the pieces above to construct the ARN solution of the original \RHP on $C_{u,\us}$, with initial data $\uphi$. The jumps of $\widehat\xi$ and $\widehat\chi$ are given by~\eqref{eq:jumpxi} and~\eqref{eq:jumpchi}. Thus $\omega:=\widehat\xi+\widehat\chi$ is the ARN solution of the \RHP on $C_{u,\us}$ with initial data $\uphi$ if the jumps are correct, i.e.~if $\uh$ and $\ug$ satisfy  the linear integral equation
\begin{equation}\label{phi-h-g}
  \uphi=(\underline I+\underline\calK^{re})\uh+(i\underline I+\underline\calK^{im})\ug.
\end{equation}
To keep track of the reality of $\uh$ and $\ug$, we write out the real and imaginary parts of this equation separately, so it becomes
\begin{equation}\label{eq:phifinal}
\left(\begin{array}{c} \Re\, \uphi \\ \Im\, \uphi
\end{array} \right) = \left(\underline I+\underline\calK\right)\left(\begin{array}{c} \uh \\ \ug
\end{array} \right)
\end{equation}
where
\begin{equation}\underline \calK:=\left(\begin{array}{cc} \Re\, \underline\calK^{re} &\Re\, \underline\calK^{im}\\
\Im\, \underline\calK^{re}& \Im\, \underline\calK^{im}\end{array} \right)
\end{equation}
is now a real matrix-valued integral operator.
Thus finally our goal is to show that for any given $\uphi$, the linear integral equation~\eqref{phi-h-g} has a solution $\uh,\ug$. If the norm of the operator $\underline\calK$ is sufficiently small, then the inverse of the operator $\underline I+\underline\calK$ is given by the sum of the convergent series $\underline I+\sum_{l=1}^\infty (-\underline\calK)^l$, so that~\eqref{eq:phifinal} is solved by the convergent series
\begin{equation}\label{h-f}
  \uh:=\sum_{l=0}^\infty (-1)^l\uh^{(l)},\quad \ug:=\sum_{l=0}^\infty (-1)^l \ug ^{(l)}
\end{equation}
with the leading terms $\uh^0:=\Re \uphi$ and $\ug^0:=\Im\uphi$,  and the higher order terms defined recurrently by
\begin{equation}\label{rec}
  \uh^{(l)}+i\ug^{(l)}:=\underline\calK^{re}\uh^{(l-1)} +\underline\calK^{im}\ug^{(l-1)}.
\end{equation}

\subsubsection*{Bounds the terms of the series}
To complete the construction of the ARN solution, it thus remains to show that the series~\eqref{h-f} with terms defined by~\eqref{rec} are indeed convergent. For this, we will estimate the norm of the integral operator $\underline\calK$. The explicit recursive bounds that we obtain for the terms of the series~\eqref{rec} will be crucial in our further analysis of the behavior the ARN solution as $\us\to 0$, yielding eventually proposition~\ref{ARNPJP}, going beyond proving the convergence of the series~\eqref{rec}.

\smallskip
We define the $L^\infty$ norm of a collection of one-forms  $\phi_e=\tilde\phi_e dz_e$ on the seams by
\begin{equation}\label{normLinf}
  |\uphi|_{\us}:=\max_e \sup_{z_e\in \gamma_e^{s_e}}\left|\frac{2\pi \phi_{e}}{d\ln z_e}\right|=\max_e \left(2\pi \sqrt{|s_e|}\cdot \sup_{z_e\in \gamma_e^{s_e}} |\tilde\phi_{e}|\right),
\end{equation}
and our goal is to bound $|\uh^{(l)}+i\ug^{(l)}|_\us$ for the terms recursively defined by~\eqref{rec}.

Since ${\bf K}^{re}(p,q)$ is a real-analytic function of $q$, there exists a constant $M_1$ such that
\begin{equation}\label{eq:KK}
 \left|{\bf K}_v^{re}(z_e,w_{e'})-{\bf K}^{re}_v(z_e,0)\right|<M_1\,|w_{e'}|
\end{equation}
holds for all $v$, all $e,e'\in E_v$, and any $z_e\in V_e,w_{e'}\in V_{e'}$.
Since ${\bf K}^{re}(p,q)$ depends real-analytically on the moduli coordinate $u$, the constant $M_1$ can be chosen to be the same for all $u$ in some neighborhood of $u=0$. As we emphasized above and proved in equation~\eqref{vanishing}, $\int_{p\in\gamma_e}(\underline \calK^{re}\uh)_e(p)=0$ for any $\uh$ such that $\int_{p\in\gamma_e}h_e(p)=0$ for any $e$. Similar analysis applies for $\underline \calK^{im}\ug$, and it thus follows from the recurrent definition~\eqref{rec} that $\int_{p\in\gamma_e}h_e^{(l)}(p)=\int_{p\in\gamma_e}g_e^{(l)}(p)=0$ for any $e$ and for all $l$. Multiplying~\eqref{eq:KK} by $h_{e'}^{(l)}$ and integrating over $\gamma_{e'}$, the contribution from
${\bf K}(z_e,0)$ vanishes, so that we obtain for any $z_e\in V_e$
\begin{equation}\label{est1}
  \left|\int_{w_e'\in\gamma_{e'}} {\bf  K}^{re}(z_e,w_{e'})h_{e'}^{(l)} (w_{e'})\right|< M_1 \sqrt{|s_e}|\, \sup_{w_{e'}\in\gamma_{e'}} \left|\frac{2\pi i h_{e'}^{(l)}}{d\ln w_{e'}}\right|\le M_1\sqrt{|\us|}|\uh^{(l)}|_\us
\end{equation}
By recalling the definition of $\underline\calK^{re}$, the bound above implies the same bound for it; the bound for $\underline\calK^{im}$ is obtained analogously. Adding these inequalities and applying them recursively in $l$, it follows that there exists a constant $M_2$ such that
\begin{equation}\label{est11}
  |\uh^{(l)}+i\ug^{(l)}|_{\us}<(M_2\sqrt{|\us|})^l|\uphi|_\us
\end{equation}
for any $l$.
Thus for $\sqrt{|\us|}<(2M_2)^{-1}$ the terms of the series~$\uh$ and $\ug$ are bounded by the geometric sequence with ratio less than $1/2$. Since the sum of such a geometric series is less than 2, for further use we record that we have proven the crucial inequality
\begin{equation}\label{eq:hg}
|\uh+i\ug|_\us<2 |\uphi|_\us,
\end{equation}
for any $|\us|<(2M_2)^{-2}$. To finish the proof of convergence we note that in particular the sum of the left-hand-sides of~\eqref{est11} for all $l$ converges, and thus the convergent series~\eqref{h-f} give a solution of the linear integral equation~\eqref{eq:phifinal}. Thus the sum of the corresponding Cauchy integrals $\widehat\xi+\widehat\chi$ gives the ARN solution of the original \RHP with initial data $\uphi$, finally proving proposition~\ref{prop:RH}.
\end{proof}

\medskip
Our main interest is the behavior of the ARN solution as $\us\to0$, and the setup is as follows. Let  $\uf=\lbrace f_e\rbrace=\lbrace \tilde f_e(z_e) dz_e\rbrace$ be a collection of holomorphic 1-forms, defined on the unit disks $V_e$. Then for any $|\us|<1$ we consider the~\RHP with initial data $\lbrace\phi_e\rbrace:=\left\lbrace \left.\left(f_e-I^*_ef_{-e}\right)\right|_{\gamma_e^{s_e}}\right\rbrace$. Note that by the residue theorem,  $\int_{\gamma_e} f_e=0$ for any $e$, so this $\phi_e$ can be used as initial data for the \RHP.  The fact that $\phi_e$ is a restriction to the seam of a holomorphic 1-form will allow us to explicitly compute the Cauchy principal value appearing in the Sokhotski-Plemelj formula, while we will crucially use the fact that our Cauchy kernels $K^{re}_v$ and $K^{im}_v$ are taken on $C^v$, and independent of $\us$.

We define the $L^\infty$ norm of $\uf$ by
\begin{equation}
|\uf|:=2\pi \max_e \sup_{|z_e|=1} \left| \frac{f_e}{d\ln z_e}\right|=2\pi \max_e \sup_{|z_e|=1} | \tilde f_e|.
\end{equation}
The Schwartz lemma on the disk $U_e=\lbrace |z_e|<\sqrt{|s_e|}\rbrace$  implies
\begin{equation}\label{Schwartz}
|\uf|_{\us}:=\left|\uf|_{\gamma_e^{s_e}}\right|\le  |\uf| \,(\sqrt{|\us|})^{\ord \uf},
\end{equation}
where we have denoted $\ord \uf:=\min_e (\ord_{q_e} f_e)$. We further denote, for any irreducible component $C^v$, $|\us_{\,v}|:=\max_{e\in E_v}|s_e|$. We finally denote the usual $L^2$ norm of a differential on $\widehat C_\us^v$ by
$$
||\omega||_{\widehat C_\us^v}^2:=\frac{i}{2}\int_{\widehat C^v_\us}\omega^v\wedge\overline\omega^v.
$$

Our main  bound is then the following
\begin{prop}[Bound for the ARN solution]\label{ARNPJP}
For a fixed $\uf$ and any sufficiently small $|\us|$, let $\omega_\us$ be the ARN solution of the~\RHP with initial data $\lbrace\phi_e\rbrace:=\left\lbrace \left.\left(f_e-I^*_ef_{-e}\right)\right|_{\gamma_e^{s_e}}\right\rbrace$. Then there exists a constant $M$ independent of sufficiently small $u$ and $\us$ such that for any irreducible component $C^v$ the following inequality holds:
\begin{equation}\label{L2-estim}
||\,\omega_\us||_{\widehat C^v_\us}\le M\,|\uf|\,\sqrt{|\us_v|}^{\,\,\ord \uf+1}.
\end{equation}
\end{prop}
Since $C^v_{\underline 1}:=C^v\setminus \cup_{e\in E_v}V_e$ is a compact set on which $L^2$ and $L^1$ norms can be bounded in terms
of each other, we have the following
\begin{cor}\label{cor:zero}
For any fixed path $\gamma\subset C^v_{\underline 1}$, there exists a constant $M_\gamma$ such that the following inequality holds for all sufficiently small $u$:
$$
 \left|\int_\gamma\omega_\us\right|\le M_\gamma|\uf|\sqrt{|\us_v|}^{\ord\uf+1}.
$$
\end{cor}

To prove the proposition, we first obtain a pointwise bound.
\begin{lm}
In the setup as above, there exists a constant $M_3$ independent of $u$ and $\us$, such that for any $z_e\in \gamma_e^{s_e}$, writing the ARN solution as $\omega^v_\us=\tilde \omega^v(z_e) dz_e $, the following inequality holds:
 \begin{equation}\label{25}
|\tilde \omega^v(z_e)|<M_3 |\uf|\sqrt{|\us_v|}^{\ord \uf}.
\end{equation}
\end{lm}
\begin{proof}
To obtain this bound, we  return to the details of the construction of the ARN solution; using the fact that the initial data is the restriction of a holomorphic 1-form will eventually allow us to evaluate the singular integral by the residue theorem. Writing the ARN solution for any $z_e\in V_e$ as the sum of~\eqref{int_loc} and the similar formula for $\chi$ gives
$$
\omega^v(z_e)=(\calK^{re}\uh)_e(z_e)+(\calK^{im}\ug)_e(z_e)+\frac{dz_e}{2\pi i}\int_{w_e\in\gamma_e}
\frac{h_e(w_e)+ig_e(w_e)}{z_e-w_e}.
$$
The first two summands are holomorphic functions of $z_e\in V_e$ (see the discussion after formula~\eqref{calK}). For the singular integral, recall that $\uh$ and $\ug$ are given by the sums of the series~\eqref{h-f}, with the terms defined recursively by
$$
\begin{aligned}
\Big(\uh^{(l)}&+i\ug^{(l)}\Big)_e=\left(\underline\calK^{re}\uh^{(l-1)} +\underline\calK^{im}\ug^{(l-1)}\right)_e\\
&=\left(\calK^{re}\uh^{(l-1)}+\calK^{im}\ug^{(l-1)}\right)_e-
I^*_e\left(\left(\calK^{re}\uh^{(l-1)}+\calK^{im}\ug^{(l-1)}\right)_{-e}\right),
\end{aligned}
$$
where  we have recalled the definition of the operator $\underline\calK^{re}$ in formula~\eqref{ucalK} (and $\underline\calK^{im}$ is similar). In this recursive definition we clearly see a holomorphic form on $V_e$, and a pullback of a holomorphic form on $V_{-e}$ under $I_e^*$. To make use of this, we define the series  $\varphi_e:=\sum_{l=0} (-1)^l \varphi_e^{(l)}$ with the first term $\varphi_e^{(0)}=f_e$, and the further terms simply being the holomorphic forms appearing in the recursion above:
\begin{equation}\label{holom}
\begin{aligned}
\varphi_e^{(l)}:&=(\calK^{re}\uh^{(l-1)}+\calK^{im}\ug^{(l-1)})\\&
=dz_e\sum_{e'\in E_{v}}\int_{w_{e'}\in \gamma_{e'}} \left({\bf K}_v^{re}(z_e,w_{e'})
h_{e'}^{(l-1)}(w_{e'})+{\bf K}_v^{im}(z_e,w_{e'})
g_{e'}^{(l-1)}(w_{e'})\right).
\end{aligned}
\end{equation}
Now $\varphi_e^{(0)}$ is a holomorphic form on $V_e$, and so is every $\varphi_e^{(l)}$; moreover, equation~\eqref{est11} gives a bound on $|\varphi_e^{(l)}|$ for $l\ge 1$, which shows that for $|\us|$ sufficiently small the series defining $\varphi_e$ converge uniformly --- and thus their sum $\varphi_e$ is also a holomorphic 1-form on $V_e$. Moreover, the bound~\eqref{est11} combined with the Schwarz inequality~\eqref{Schwartz} implies
\begin{equation}\label{eq:fF}
|\underline{\varphi}|_\us<2|\uf|_\us\le 2|\uf|\sqrt{|\us_v|}^{\ord\uf}.
\end{equation}
We finally rewrite the ARN solution as
\begin{equation}\label{xiLocal}
\omega^v(z_e)=(\calK^{re}\uh)_e(z_e)+(\calK^{im}\ug)_e(z_e)+\frac {dz_e}{2\pi i}\int_{|w_e|=\sqrt{|s_e|}}\frac {\varphi_e(w_e)-I_e^*(\varphi_{-e})(w_e)}{z_e-w_e},
\end{equation}
and now estimate each term in this formula. Since $\varphi_e$ is holomophic, by the residue theorem we compute for any $|z_e|> \sqrt{|s_e|}$ the singular integral to be
\begin{equation}\label{xiLocal1}
\frac {dz_e}{2\pi i}\int_{|w_e|=\sqrt{|s_e|}}\frac {\varphi_e(w_e)-(I_e^*(\varphi_{-e}))(w_e)}{z_e-w_e}=-I^*_e(\varphi_{-e}) (z_e)
\end{equation}
Indeed since each $\varphi_e$ is holomorphic for $w_e\in V_e$, the first term in~\eqref{xiLocal} has no residues in the disk $|w_e|<\sqrt {|s_e|}$. The second term is equal to (recall that $I_e$ is orientation-reversing)
$$
\frac {dz_e}{2\pi i}\int_{|w_{-e}|=\sqrt{|s_e|}}\frac {\varphi_{-e}(w_{-e})}{z_e-s_ew_{-e}^{-1}}=
-\frac {dz_e}{2\pi i z_e}\int_{|w_{-e}|=\sqrt{|s_e|}}\frac {w_{-e}\varphi_{-e}(w_{-e})}{s_ez_e^{-1}-w_{-e}}
$$
\begin{equation}\label{xiLocal2}
=\frac {s_edz_e}{z_e^2}\tilde\varphi_{-e}(s_e z_e^{-1})=-I^*_e(\varphi_{-e}).
\end{equation}
Equation~\eqref{xiLocal2} extends continuously to $|z_e|=\sqrt{|s_e|}$. Thus equation~\eqref{eq:fF} implies that the singular integral in~\eqref{xiLocal} is bounded by $2|\uf|\sqrt{|\us_v|}^{\ord\uf}$, which is the order of the bound that we want.

On the other hand, we have  the bound~\eqref{est1} for the norm of  $\calK^{re}$, and an analogous bound for the norm of $\calK^{im}$. Then combining the bound ~\eqref{eq:hg} with the Schwartz inequality gives
$$
|\calK^{re}\, \uh+\calK^{im}\,\ug|_\us\le M'\sqrt{|\us_v|}^{\ord\uf +1},
$$
which is of smaller order than the bound for the singular integral. We have thus obtained bounds for both summands giving $\omega^v$ in formula~\eqref{xiLocal}, and the lemma is proven.
\end{proof}
Now we are ready to prove the bound for the $L^2$ norm of $\omega$.
\begin{proof}[Proof of proposition~\ref{ARNPJP}]
As in the proof of uniqueness of the ARN differential, we use the Stokes' theorem expression~\eqref{stokes} for the $L^2$ norm of $\omega^v$. For a given $e$, we thus need to compute
\begin{equation}\label{eq:eq}
 \int_{z_e\in\gamma_e}F^{v(e)}(z_e)\,dF^{v(e)*}(z_e)=\int_{\gamma_e}F^{v(e)}(z_e)\cdot
 \left(\omega^v(z_e)+\bar\omega^v(z_e)\right).
\end{equation}
Recall that since $\int_{z_e\in\gamma_e}\omega^{v(e)}(z_e)=0$, it follows that this integral does not depend on the choice of the constant of integration for the definition of $F^{v(e)}$, so we can pick any point  $z_e^0\in\gamma_e$, and replace $F^{v(e)}(z_e)$ on the right-hand-side of~\eqref{eq:eq}  by $F^{v(e)}(z_e)-F^{v(e)}(z_e^0)$.
From the lemma we have the pointwise bound $|\overline{\tilde\omega}^v|\le M_3|\uf|\sqrt{|\us|}^{\ord\uf}$ for some constant $M_3$ independent of $u$ and $\us$. As the length of the arc from $z_e^0$ to $z_e$ is at most $2\pi\sqrt{|s_e|}$, it follows that by integrating $\overline{\tilde\omega}$ that $F^{v(e)}(z_e)-F^{v(e)}(z_e^0)\le 2\pi M_3|\uf|\sqrt{|\us|}^{\ord\uf+1}$. We thus obtain
$$
 \int_{z_e\in\gamma_e}F^{v(e)}(z_e)\,dF^{v(e)*}(z_e)\le 2\pi M_3|\uf|\sqrt{|\us|}^{\ord\uf+1}\int_{\gamma_e}|2\omega^v(z_e)|.
$$
The lemma again gives a pointwise bound for the integrand, while integrating over the seam
introduces another factor of $2\pi\sqrt{|s_e|}$, so that summing over all $e$ we finally obtain
\begin{equation}\label{29}
||\omega^v||^2\le 4\pi^2M_3^2|\uf| \sqrt{\max_{e\in E_v}|\us_e|}^{\,\,2\ord f+2}\cdot \# E_v.
\end{equation}
Thus finally there exists a constant $M$ such that inequality~\eqref{L2-estim} holds.
\end{proof}

\section{The RN differential in plumbing coordinates}\label{sec:RNplumbing}
In this section we construct explicitly the RN differential on any smooth jet curve $X_{w,u,\us}$ in plumbing coordinates. The construction starts with a collection of RN differentials on $C^v$ with prescribed singular parts at $p_\ell$ and with residues at the nodes given by a solution of the {\em flow} Kirchhoff problem. Note that this differential is not strictly speaking real-normalized: already for the case of an irreducible nodal curve and a differential with no residues, when there are no residues in the limit either, the period over a loop passing through a node is equal to the integral of the RN integral on the normalization from one preimage of the node to the other, which may not be real.

We then use the ARN solution of the \RHP with the initial data matching the jumps of this collection of differentials on $C^v$ to construct a differential on $C_{u,\us}$. Since the ARN solutions have zero periods, the resulting differential will not be real-normalized. We thus proceed recurrently, by using RN differentials with residues solving the {\em force} Kirchhoff problem, with the electromotive force being equal to the imaginary part of the periods of the differential constructed at the previous step. The bound of the ARN solution of the \RHP from proposition~\ref{ARNPJP} is used crucially to show that this construction converges.

We continue to use the notation for plumbed surfaces and the \RHP as in the previous section.
\begin{notat}
For a fixed smooth jet curve $X_{w,u,\us}$, given any collection of real numbers $\uc=\lbrace c_e\rbrace_{e\in E(\Gamma)}$ satisfying conditions (0) and (1) of the Kirchhoff problem, we denote $\Phi(\uc)=\lbrace \Phi^{v}(\uc)\rbrace_{v\in V(\Gamma)}$ the collection of RN differentials on $C^v_u$ with prescribed singular parts $\sigma_\ell$ and prescribed residues $r_\ell$ (encoded by the coordinate $w$) at the marked points $p_\ell$,  with residue $ic_e$ at the preimage $q_e$ of every node, and holomorphic elsewhere.
\end{notat}
\begin{lm}\label{lm:psiphi}
For any $\uc$ satisfying conditions (0) and (1) of the Kirchhoff problem, let $f_e(\uc):=\left.\Phi^{v(e)}(\uc)\right|_{V_e}-ic_e dz_e/z_e$ be the collection of holomorphic differentials on $V_e$. Let $\omega(\uc)=\lbrace \omega^v(\uc)\rbrace$ be the ARN solution on $\widehat C_{u,\us}$ to the \RHP with initial data $\left.\left(f_e(\uc)-I_e^*f_{-e}(\uc)\right)\right|_{\gamma_e^{s_e}}$. Then the collection of differentials
\begin{equation}\label{psiR}
  \Psi^v(\uc):=\Phi^{v}(\uc)-\omega^{v}(\uc)
\end{equation}
is the unique meromorphic differential on $C_{u,\us}$ with singularities $\sigma_\ell$ and residues $r_\ell$ at $p_\ell$ prescribed by the coordinate $w$, and holomorphic elsewhere, such that $\int_{\gamma_e}\Psi(\uc)=2\pi c_e$ for any $e$, and $\int_\gamma\Psi(\uc)\in\RR$ for any cycle $\gamma\in H_1(\widehat C_{u,\us},\ZZ)$.
\end{lm}
Again, recall that the last condition is equivalent to $\int_\gamma\Psi(\uc)\in\RR$ for any $\gamma\in H_1(C_{u,\us},\ZZ)$ not intersecting the seams. As in the previous section, we will drop $w$ and $u$ in the notation from now on.
\begin{proof}
Since $I_e^*(ic_{-e}dz_{-e}/z_{-e})=-ic_{-e}dz_e/z_e=ic_edz_e/z_e$, the jumps of $\omega(\uc)$ on the seams are equal to those of $\Phi(\uc)$, and thus by construction the differential $\Psi(\uc)$ on $\widehat C_\us$ has no jumps. Thus $\Psi(\uc)$ defines a meromorphic differential on $C_\us$ with prescribed singularities at $p_\ell$, and holomorphic elsewhere (the simple pole of $\Phi^{v(e)}(\uc)$ at $q_e$ is cut out by plumbing). Since $\int_{\gamma_e}\omega^{v(e)}(\uc)=0$ for any $e$, by definition of the ARN solution, the residue theorem on $U_e$ yields
$$\int_{\gamma_e}\Psi(\uc)=-2\pi i\Res_{q_e}\Phi^v(\uc)=2\pi c_e$$
(for the sign, recall that $\gamma_e$ is oriented as the boundary of $\widehat C^{v(e)}$).
Since $\Phi^v(\uc)$ is real-normalized on $C^v$, while the period of $\omega^{v}(\uc)$ over any cycle on $\widehat C^v_\us$ is real, it follows that the integral of $\Psi(\uc)$ over any cycle on $\widehat C_\us$ is also real.

The uniqueness of $\Psi(\uc)$ follows from the uniqueness of the ARN solution of the \RHP. Indeed, if $\Psi'$ is another such differential, then $\Psi(\uc)-\Psi'$ is a holomorphic differential on $C_{\us}$ with zero periods over the seams, and with real periods over all cycles not intersecting the seams. Thus it is the ARN solution of the \RHP with zero initial data, which must then be identically zero.
\end{proof}
\begin{notat}\label{df:Omega}
The construction above can be applied for the case when there are no prescribed singularities. In this case suppose $\uc'=\lbrace c_e'\rbrace$ satisfy conditions (0) and (1) of the Kirchhoff problem with no inflow, i.e.~when all $f_\ell$ are equal to zero. We denote $\Omega(\uc')$ the differential constructed on $C_{\us}$ by the above procedure. Then $\Omega(\uc')$ is the unique {\em holomorphic} differential on $C_\us$ such that $\int_{\gamma_e}\Omega(\uc')=2\pi c_e'$ for any $e$ and $\int_\gamma\Omega(\uc')\in\RR$ for any  $\gamma\in H_1(\widehat C_\us,\ZZ)$.
\end{notat}
From uniqueness it  follows that
\begin{equation}\label{eq:OmegaPsi}
\Psi(\uc+\uc')=\Psi(\uc)+\Omega(\uc')
\end{equation}
for any such $\uc$ and $\uc'$, since the right-hand-side is a differential on $C_\us$ satisfying the same conditions as the left-hand-side. This relation will be used in the next section.

\smallskip
Consider  the RN differential $\Psi_X=\Psi_{w,u,\us}$ on $X_{w,u,\us}$: it has prescribed singularities at $p_\ell$, is holomorphic elsewhere, and its periods over all cycles are real. Thus by uniqueness of the latter, $\Psi_{w,u,\us}$ is equal to $\Psi(\uc)$, for $c_e:=(2\pi)^{-1}\int_{\gamma_e}\Psi_{w,u,\us}$. Thus constructing the RN differential $\Psi_X$ is equivalent to determining $\uc$ (depending on $w,u,\us$) such that $\Psi_X=\Psi(\uc)$. We will construct $\uc$ recursively, starting with the solution of the {\em flow} Kirchhoff problem, and then recursively solving the {\em force} Kirchhoff problem, with the force being the imaginary parts of the periods of the previous term in the series. The difficulty is that the periods over the paths passing through the nodes may diverge. We first show that they have well-controlled logarithmic divergences.
\begin{notat}
For any closed path $\gamma$ on $C$, let $\widehat\gamma_\us$ be the collection of paths that are the intersections $\gamma\cap\widehat C_\us^v$ (conveniently, in our plumbing setup $\widehat C^v_\us$ is a subset of $C^v$ for any $\us$). These $\widehat\gamma_\us$ do not form a closed path on $C_\us$, as the point $\gamma\cap\gamma_e$ may not be the preimage of  $\gamma\cap\gamma_{-e}$ under the identification $I_e$.
We choose a starting point on $\gamma$ arbitrarily, and then denote $\gamma_{e_1},\dots,\gamma_{e_N}$ (possibly with repetitions) the set of seams that $\gamma$ intersects (oriented so that $\gamma$ crosses from $\gamma_{e_j}$ to $\gamma_{-e_j}$), and denote by $C^{v_j}$ the component of $C$ that $\gamma$ lies on after crossing $\gamma_{e_j}$. Let then $\widehat\gamma_\us^j$ be the segment of $\gamma$ contained in $C^{v_j}$ going from $\gamma_{e_{j-1}}$ to $\gamma_{e_j}$. Let $\delta_j$ be an  arc of  $\gamma_{e_j}$ connecting the point $\gamma\cap\gamma_{e_j}$ to $I_{e_j}^{-1}(\gamma\cap\gamma_{-e_j})$. Finally let $\gamma_\us$ be the closed path on $C_\us$ obtained by traversing $\lbrace \widehat\gamma_\us^1,\delta_1,\dots, \widehat\gamma_\us^N,\delta_N\rbrace$ in this order.
\end{notat}
\begin{lm}\label{periodsofpsi}
For any closed path $\gamma$ on $C$ and any $\uc$ satisfying conditions (0) and (1) of the Kirchhoff problem, there exist constants $\Pi_\gamma(\uc)$ (independent of $\us$, but depending on $w,u$), and $M_5(\uc)$ (independent of $w,u,\us$)  such that
\begin{equation}\label{intgamma}
\left|\,\Im \int_{\gamma_\us} \Psi(\uc)-\sum_{j=1}^Nc_{e_j}\ln |s_{e_j}|-\Pi_\gamma(\uc)\right|\le M_ 5(\uc)\sqrt{|\us|}.
\end{equation}
for any $w,u,\us$ sufficiently small.

Furthermore, the constant $\Pi_\gamma(\uc)$ only depends on the class of $\gamma$ in the cohomology of the dual graph $\Gamma$ of $C$.
\end{lm}
We note that for each $j$ there are two choices of $\delta_{e_j}$, by going in opposite directions around $\gamma_{e_j}$; however, since $\Im\int_{\gamma_e}\Psi(\uc)=0$ by definition, the imaginary part of the integral of $\Psi(\uc)$ is the same for either choice.

For the second statement, note that there is a map $H_1(C,\ZZ)\to H_1(\Gamma,\ZZ)$. This map is obtained by ``extending'' every node of $C$ to a segment, and then contracting every irreducible component of $C$ to a point. Geometrically, the image of $\gamma$ in $H_1(\Gamma,\ZZ)$ encodes simply the sequence in which $\gamma$ passes through the nodes of $C$.
\begin{proof}
We recall from lemma~\ref{lm:psiphi} that $\Psi(\uc)=\Phi(\uc)-\omega(\uc),$ and start by estimating the periods of $\Phi(\uc)$. We break each path $\gamma^j_\us$ into three parts: $\gamma_1$ contained in $V_{e_{j-1}}$ and connecting a boundary point of $U_{e_{j-1}}$ to the boundary of $V_{e_{j-1}}$, a similar path $\gamma_2$ contained in $V_{e_j}$, and the path $\gamma_3$ connecting a boundary point of $V_{e_{j-1}}$ to a boundary point of $V_{e_j}$. Since the neighborhoods $V_e$ are independent of $\us$,  $\int_{\gamma_3}\Phi(\uc)$ is independent of $\us$. For $\gamma_2$, we write $\Phi(\uc)=ic_{e_j}dz_{e_j}/z_{e_j}+\Phi'(\uc)$ in $V_{e_j}$, where $\Phi'(\uc)$ denotes the regular part. Integrating the singular part over $\gamma_2$, which goes from some $z'\in \partial V_{e_j}$ (i.e.~$|z'|=1$) to some $z\in\partial U_{e_j}$ (i.e.~$|z|=\sqrt{|s_{e_j}|}$) yields $\Im\int_{\gamma_2} -ic_e dz_{e_j}/z_{e_j}=-c_e\ln |z/z'|=-c_e\ln\sqrt{|s_{e_j}|}$. To estimate $\int_{\gamma_2}\Phi'(\uc)$, we can add to $\gamma_2$ a path connecting $z$ to $0$. The resulting path then does not depend on $\us$, and thus the integral of $\Phi'(\uc)$ over it is independent of $\us$. Since $\Phi'(\uc)$ is regular in $V_{e_j}$, its norm there is bounded, and thus the integral of $\Phi'(\uc)$ over a path in $V_{e_j}$ connecting $z$ to $0$ is bounded by a constant times $|z|=\sqrt{|s_{e_j}|}$. We thus see that there exist constants $\alpha$ and $\beta$ independent of $\us$ (but depending on $\uc$, and on $w,u$) such that
$$
   \left|\Im\int_{\gamma_2}\Phi(\uc)-\alpha-c_{e_j}\ln\sqrt{|s_{e_j}|}\right|<\beta\sqrt{|s_{e_j}|}.
$$
An analogous estimate holds for $\Im\int_{\gamma_1}\Phi(\uc)$, where we note that the path is now oriented from the boundary of $U_{e_{j-1}}$ to the boundary of $V_{e_{j-1}}$, but also that the residue is equal to $ic_{-e_{j-1}}=-ic_{e_{j-1}}$. Combining $\gamma_1,\gamma_2,\gamma_3$ together, we thus finally see that there exist a constant $\Pi_\gamma^{j}(\uc)$ independent of $\us$ (and depending smoothly on $w,u$), and a constant $M_6(\uc)$ independent of $w,u,\us$, such that
\begin{equation}\label{intpart}
\left|\,\Im\,\int_{\widehat\gamma^j_\us}\Phi^{v_j}(\uc)-\frac12 c_{e_{j-1}}\ln|s_{e_{j-1}}|-\frac12 c_{e_j} \ln|s_{e_j}|-\Pi_\gamma^j(\uc)\,\right|<M_6(\uc)\sqrt{|\us|}
\end{equation}
for all $w,u,\us$ sufficiently small.

We now estimate $\int_{\gamma^j_\us}\omega(\uc)$.  By corollary~\ref{cor:zero},  $|\int_{\gamma_3}\omega(\uc)|$ is bounded by a constant times $|\uf|\cdot\sqrt{|\us|}$, which is the size of the bound we want. For $\int_{\gamma_1}\omega(\uc)$ and $\int_{\gamma_2}\omega(\uc)$ corollary~\ref{cor:zero} does not apply, as these paths depend on $\us$. However, $\gamma_2$ is contained in $V_{e_{j}}$, and formula~\eqref{xiLocal} with the singular integral computed in~\eqref{xiLocal1} expresses $\omega|_{V_{e_{j}}}$ as a sum of a holomorphic form on $V_{e_j}$ and $-f_{-e}(s_ez_e^{-1})$. Similarly to the case of $\Phi'(\uc)$, the integral of the holomorphic form has the form required, while $-\int_{z'}^z f_{-e}(s_ez_e^{-1})=\int_{s_e/z}^{s_e/z'} f_{-e}(z_{-e})$ is also the integral of a holomorphic form on $V_{-e}$ from a point $s_e/z$ of absolute value $\sqrt{|s_e|}$ to the point $s_e/z'$ of absolute value $|s_e|$. Thus the same argument applies, and altogether $\int_{\gamma_2}\omega(\uc)$ differs from some constant $\alpha'$ by less than $\beta'\sqrt{|s_{e_j}|}$, for some other constant $\beta'$. The integral $\int_{\gamma_1}\omega(\uc)$ is completely analogous, so at the end we obtain the bound
\begin{equation}\label{intpart1}
\left|\,\Im\,\int_{\widehat\gamma^j_\us}\omega^{v_j}(\uc)-E^j(\uc)\,\right|<M_{7}(\uc)\sqrt{|\us|}
\end{equation}
with some constant $E^j$ independent of $\us$ (but depending on $w,u$), and $M_7$ independent of $w,u,\us$.

To estimate the integral of $\Psi(\uc)$ over the segments $\delta_j$, we first note that $\Im\int_{\delta_j}ic_{e_j}d\ln z_{e_j}=0$ for the singular part of $\Phi(\uc)$. Then $\int_{\delta_j}\Phi'(\uc)$ is bounded by a constant independent of $\us$ times the $L^2$ norm of $\Phi'(\uc)$ on $V_{e_j}$ (which is also a constant independent of $\us$) times the length of $\delta_j$, which is at most $2\pi\sqrt{|s_{e_j}|}$. Similarly $\Im\int_{\delta_j}\omega(\uc)$ is bounded by a constant independent of $\us$, times $\sqrt{|s_{e_j}|}$, since $\omega(\uc)$ restricted to the seam is a sum of a form holomorphic on $V_e$, and a pullback of a form $f_{-e}$ holomorphic on $V_{-e}$, restricted to the seam. Thus altogether the lemma follows by subtracting equation~\eqref{intpart1} from equation~\eqref{intpart}, and summing over all $j$ (notice that in the sum each $e_j$ appears twice, and thus the two $1/2$ in front of $c_{e_j}\ln|s_{e_j}|$ appearing in~\eqref{intpart} for $\hat\gamma_j$ and $\hat\gamma_{j+1}$ add up to the coefficient 1 in~\eqref{intgamma}).

\smallskip
To prove that $\Pi_\gamma(\uc)$ only depends on the class of the image of $\gamma$ in $H_1(\Gamma,\ZZ)$, we note that on each $C^v$ the differential $\Psi(\uc)$ is real-normalized, so that $\Im\int_{\gamma^v}\Psi(\uc)=0$ for any closed loop $\gamma^v\in H_1(C^v,\ZZ)$. Thus $\Pi_\gamma(\uc)=\Pi_{\gamma'}(\uc)$ for any paths $\gamma$ and $\gamma'$ that only differ on $C^v$, which is exactly to say that $\Pi$ only depends on $\gamma\in H_1(\Gamma,\ZZ)$.
\end{proof}
\begin{rem}
The above proof uses the details of the estimates of the ARN solution. Obtaining similar results using the classical approach to the \RHP, with the Cauchy kernel $K_{C_\us}$ on a varying Riemann surface, appears much harder.
\end{rem}
We now construct the RN differential $\Psi_{w,u,\us}$ on any smooth jet curve explicitly, as $\Psi(\uc)$, with $\uc$ defined as a sum of a series. We let $\uc^{(0)}(\us):=\lbrace c_{e}^{(0)}(\us)\rbrace$ be the solution of the flow Kirchhoff problem on the dual graph of $C_{u,\underline{0}}$ with inflows $ir_\ell$, and resistances $\rho_{|e|}(\us)=-\ln |s_e|$ being the log plumbing coordinates. While our construction depends on $w,u,\us$ --- and here we write this out explicitly --- note that the Kirchhoff problem and its solution are independent of $u$. We only work with smooth curves here, i.e.~all $s_e$ are non-zero.

A priori by lemma~\ref{periodsofpsi} the imaginary parts of the integrals of $\Psi(\uc^{(0)})$ may have logarithmic divergences. However, the sum of these logarithmic divergences on a closed path is $\sum c_e^{(0)}\ln |s_e|=-\sum c_e^{(0)}\rho_{|e|}$ is precisely the left-hand-side of condition (2) of the Kirchhoff problem (equation~\eqref{loopvoltage}). Since $\uc^{(0)}$ is the solution of the flow Kirchhoff problem, this sum is equal to zero, as there is no electromotive force. Thus the corresponding logarithmic divergences cancel, so that we have
\begin{equation}\label{intPsic0}
\left|\Im\int_{\gamma_\us} \Psi(\uc^{(0)})-\Pi_\gamma(\uc^{(0)})\right|<M(\uc^{(0)})\sqrt{|\us\,|}
\end{equation}
with the constant $M$ independent of $w,u,\us$, and $\Pi_\gamma$ independent of $\us$. To deal with those periods of $\Psi(\uc^{(0)})$ that have a non-zero imaginary part, we introduce the correction $\uc^{(1)}(\us)$ to be the solution of the {\em force} Kirchhoff problem on the dual graph of $C_{u,\underline 0}$, with electromotive force
\begin{equation}\label{eq:E0}
\calE_{\gamma_{\us}}^{(0)}:=-\Im\int_{\gamma _\us}\Psi(\uc^{(0)}).
\end{equation}
Recall that by~\eqref{eq:OmegaPsi}
$$
 \Psi\left(\uc^{(0)}(\us)+\uc^{(1)}(\us)\right)-\Psi(\uc^{(0)}(\us))=\Omega\left(\uc^{(1)}(\us)\right),
$$
where $\Omega$ is defined in  notation~\ref{df:Omega}. Proposition~\ref{periodsofpsi} also applies to bound the periods of $\Omega\left(\uc(\us)\right)$, where we change notation to emphasize that the constants here are for the holomorphic differential,
$$
\left|\Im\int_{\gamma_{\us}}\Omega(\uc^{(1)}(\us))-\sum_{e|q_e\in\gamma} c_e^{(1)}\ln |s_e\,|-\tilde{\Pi}_\gamma(\uc^{(1)})\right|\le \tilde{M}(\uc^{(1)})\sqrt{|\us|}.
$$
Thus altogether we can estimate the imaginary parts of periods
\begin{equation}\label{eq:correction}
\left|\Im\int_{\gamma_{\us}}\Psi(\uc^{(0)}(\us)+\uc^{(1)}(\us))-\tilde{\Pi}_\gamma(\uc^{(1)})\right|\le \tilde{M}(\uc^{(1)})\sqrt{|\us|}.
\end{equation}
We therefore proceed recursively, defining for any $l>1$ the electromotive force
\begin{equation}\label{eq:El}
\calE_{\gamma_{\us}}^{(l)}:=\sum_{\lbrace e|q_e\in\gamma\rbrace} c_e^{(l)}\ln |s_e\,|-\Im\int_{\gamma_{\us}} \Omega(\uc^{(l)})
\end{equation}
and letting $c_{e}^{(l+1)}(\us)$  be the solution of the {\em force} Kirchhoff problem with electromotive force $\calE_{\gamma_{\us}}^{(l)}$, so that the estimate analogous to~\eqref{eq:correction} holds:
\begin{equation}
\left|\Im\int_{\gamma_{\us}}\Psi(\sum_{i=0}^l\uc^{(i)}(\us))-\tilde{\Pi}_\gamma(\uc^{(l)})\right|\le\tilde{M}(\uc^{(l)})\sqrt{|\us|}
\end{equation}
Thus once we show that $c_e^{(l)}(\us)$ converge to zero as $l\to\infty$, since $\tilde{\Pi}_\gamma(\uc^{(l)})$ and $\tilde{M}(\uc^{(l)})$ depend continuously on $\uc^{(l)}$ and vanish for $\uc=0$ (the differential $\Omega(\underline{0})$ vanishes as it is the ARN solution of the jump problem with zero jump), it will follow that the imaginary parts of that the periods  of $\Psi(\sum_{l=0}^\infty \uc^{(l)})$ are equal to zero --- provided that the sum of the series converges, which we will now prove.
\begin{prop}\label{prop:PsiR}
For any stable jet curve $X$  there exists a $t\in\RR_+$ such that for any $\us$ satisfying $|\us|<t$ the series
\begin{equation}\label{c-series}
 c_e(\us):=\sum_{l=0}^\infty c_{e}^{(l)}(\us),
\end{equation}
with terms recursively defined above, converge, and the differential $\Psi(\uc)$ is the RN differential on $X_{w,u,\us}$.
\end{prop}
\begin{proof}
Since $\Omega(\uc)$ depends linearly on $\uc$, the map sending $\uc^{(l)}$ to $\uE^{(l)}$ is a linear map of finite-dimensional real vector spaces. Denoting by $M$ the norm of this linear map, it follows that $|\uE^{(l)}|\le M |\uc^{(l)}|$, where $|\uc|:=\max_e|c_e|$. By the construction $\uc^{(l+1)}(\us)$ is the solution of the force Kirchhoff problem with the electromotive force $\uE^{(l)}$. The a priori bound for the solutions of the force Kirchhoff problem given by lemma~\ref{lm:aprioriforce} yields
\begin{equation}\label{tt}
  \left|\uc^{(l+1)}(\us)\right|\le(-\ln |\us|)\cdot N\cdot\left|\calE^{(l)}\right|\le M\cdot N\cdot (-\ln |\us|)\cdot\left|\uc^{(l)}(\us)\right|
\end{equation}
(where we recall that $N$ is the rank of $H^1(\Gamma)$, and is thus some constant).
Thus for $|s|<t=e^{-MN}$ the terms $\uc^{(l)}$ are bounded by a geometric series with ratio less than 1. Thus the individual terms go to zero, while the sum of the series~\eqref{c-series} converges, and as explained above, this implies that $\Psi(\uc)$ constructed from the sum of these series is the RN differential.
\end{proof}

\section{The limit RN differential: proof of theorem~\ref{thm:limits}}\label{sec:prooflimits}
We now prove the main result on limits of RN differentials. The proof will crucially use the a priori bounds on solutions of the Kirchhoff problem given by lemmas~\ref{lm:aprioribound} and~\ref{lm:aprioriforce}, which in particular imply that the residues of RN differentials are a priori bounded {\em uniformly} in a neighborhood of any given stable curve.

\begin{proof}[Proof of theorem~\ref{thm:limits}]
Let $\lbrace X_k\rbrace$ be a sequence of smooth jet curves converging to a stable jet curve $X$. Let $\uc_{\,k}^{(0)}$ be the solution of the flow Kirchhoff problem with resistances $\rho_{|e|,k}=-\ln |s_{e,k}|$ and inflows $f_{\ell,k}$, which we assume converge to some limit $\uc^{(0)}$. Let $\Psi_k$ denote the RN differential on $X_k$, which by proposition~\ref{prop:PsiR} is equal to $\Psi_k(\uc_k)$, where $\uc_k$ is the sum of the series constructed there, of which $\uc_k^{(0)}$ is the first term. As in the proof of uniqueness of the ARN solution, and as in the beginning of proof of proposition~\ref{ARNPJP}, we  apply the Stokes' theorem as in formula~\eqref{stokes} for the norm of $\Psi_k-\Psi_k(\uc_k^{(0)})$ . As before, we now sum the result over all $e$, and look at the pairs of terms corresponding to $e$ and $-e$. Since all periods of $\Psi_k$ are real, it follows that $F_k=\Im\int\Psi_k$ is a single-valued global function on $C_\us$, and thus its values on $\gamma_e$ and the pullback of its values on $\gamma_{-e}$ under $I_e$ are equal. On the other hand, $F_k(\uc_k^{(0)})=\Im\int\Psi_k(\uc_k^{(0)})$ is multi-valued on $C_\us$, as the integrals over cycles intersecting the seams may not be real. Thus the difference of the values of $F_k(\uc_k^{(0)})$ on $\widehat C^{v(e)}$ restricted to $\gamma_e$, and on $\widehat C^{v(-e)}$ restricted to $\gamma_{-e}$ and pulled back under $I_e$, is equal to the imaginary part of the integral of $\Psi_k(\uc_k^{(0)})$ over some cycle $\gamma_e^\checkmark$ on $C_\us$ intersecting $\gamma_e$. Thus altogether we see that
\begin{equation} \label{psibound}
\begin{aligned}
||\Psi_k-\Psi_k(\uc_{\,k}^{(0)})||^2&=\sum_e \int_{\gamma_e}\left(F_k-F_k(\uc_k^{(0)})\right)\,d\left(F_k^*-F_k^*(\uc_k^{(0)})\right)\\
&=\sum_e\left|\Im\int_{\gamma_e^\checkmark}\Psi_k(\uc_k^{(0)})\right|\cdot \left|\int_{\gamma_e}d\left(F_k^*-F_k^*(\uc_k^{(0)})\right)\right|\\
&=\pi \sum_e|c_{e,k}-c_{e,k}^{(0)}|\cdot \left|\Im\int_{\gamma_e^\checkmark}\Psi_k(\uc_k^{(0)})\right|,
\end{aligned}
\end{equation}
where we recall that by lemma~\ref{periodsofpsi} the period over $\gamma_e^\checkmark$ only depends on the class of this cycle in $H_1(\Gamma,\ZZ)$, i.e.~only on the sequence of nodes that the path passes through. We now take $k$ sufficiently large so that $w_k,u_k,\us_k$ (which all converge to zero as $k\to\infty$) are sufficiently small for all the bounds in all the previous results to apply.

The integral on the right-hand-side in the last line of~\eqref{psibound} is  by definition $-\calE_{\gamma_e^\checkmark,k}^{(0)}$ used in the construction of the RN differential in the previous section. The class $\gamma_e^\checkmark$ does not depend on $\us$, and thus by the proof of proposition~\ref{prop:PsiR} we have $|\calE_{\gamma_e^\checkmark,k}^{(0)}|<M_{\gamma_e^\checkmark}|\uc_k^{(0)}|$.

We now let $\uc_k=\lbrace c_{e,k}\rbrace$ be those constructed in the previous section, i.e.~these are given by the sums of the series for $X_k$, and such that $\Psi_k=\Psi_{X_k}(\uc_k)$.  Then by construction and by~\eqref{tt} we have the estimate
$$
  |c_{e,k}-c_{e,k}^{(0)}|=\left|\sum_{l=1}^\infty c_{e,k}^{(l)}\right|\le -|\uc^{(0)}|/\ln |\us_k|.
$$
Since $|\uc^{(0)}|$, for $w,u,\us$ by~\eqref{C-bound} is uniformly bounded for all $w,u,\us$ sufficiently small (this is a crucial use of a priori bounds for solutions of the Kirchhoff problem!) --- and thus for all $k$ sufficiently large, altogether~\eqref{psibound} implies the bound
\begin{equation}\label{psibound1}
||\Psi_k-\Psi_k(\uc_{\,k}^{(0)})||< M\,(-\ln|\us_k|)^{-1/2}
\end{equation}
for some constant $M$ independent of sufficiently large $k$.

Recalling from lemma~\ref{lm:psiphi} that $\Psi_k(\uc_{\,k}^{(0)})-\Phi_k(\uc_{\,k}^{(0)})$ is the ARN solution of the \RHP, the norm of which is bounded by proposition~\ref{ARNPJP}, we obtain
\begin{equation}\label{psibound2}
||\Psi_k(\uc_{\,k}^{(0)})-\Phi_k(\uc_{\,k}^{(0)})||_{C^v_s}< M_1|\us_k|^{1/2}
\end{equation}
for some constant $M_1$. Thus finally the convergence of $\uc_k^{(0)}$ implies the convergence of $\Phi_k(\uc_k^{(0)})$, which depend on them continuously, and the two bounds above then imply the convergence of $\Psi_k=\Psi_{X_k}(\uc_k)$, which is to say that the limit RN differential exists.

Now for the second part of the statement, if the sequence $\lbrace X_k\rbrace$ is admissible, by lemma \ref{lm:multiKirchhoff}
$\uc_{\,k}^{(0)}$ converge to the solution $\uc^{(0)}$ of the multi-scale Kirchhoff problem with generalized resistance $\PP\urho=\lim\PP\urho_k\in\PS[\#|E|-1]$. By the above argument the RN differentials $\Psi_k$ then converge to $\Phi(\uc^{(0)})$.
\end{proof}

\section{Limits of zeroes of RN differentials}\label{sec:zeroes}
In this section we finally state and prove our main result describing limits of zeroes of RN differentials. We will show that if the limits of zeroes of $\Psi_{X_k}$ exist, as a collection of points on $C$ with multiplicities, then these limits are the divisor of zeroes of a suitable ``twisted" collection of RN differentials on the components $C^v$, which may have higher order poles at some of the nodes where we specify that the residues are given by the series~\ref{c-series}. This twisted differential $\Phi$ arises as the limit of restrictions of $\Psi_k$ to $C^v$, scaled by some sequence of positive reals $\mu_k$ depending on $v$.

\smallskip
Suppose $X_k$ is a sequence of smooth jet curves such that $C_k$ converge to a stable curve $C$, and that in this sequence the limit of the divisor of zeroes of $\Psi_k$ exists, as a collection of points of $C$, with multiplicity. Since the space of singular parts $\sigma_\ell$, considered up to scaling all of them at once by $\RR_+$, is compact, there exists a subsequence in which the singular parts $\lbrace\sigma_{\ell,k}\rbrace$, considered up to scaling by $\RR_+$, converge. Since the zeroes of a differential are preserved under such scaling, we can further rescale all singular parts  in the sequence so that the singular parts themselves converge (not just up to scaling). Thus without loss of generality we can assume that the sequence of smooth jet curves $X_k$ converges to a stable jet curve $X$. Furthermore, since every such convergent sequence contains an admissible subsequence, without loss of generality for the rest of this section we will fix once and for all an admissible sequence $\lbrace X_k\rbrace=\lbrace X_{w_k,u_k,\us_k}\rbrace$ converging to $X$. By theorem~\ref{thm:limits} there exists a limit RN differential $\Psi$ on $C$, and our goal is to investigate the limits of zeroes of $\Psi_k$ on those components $C^v$ where $\Psi$ is identically zero.

\begin{notat}
By a {\em subcurve $D$ of a stable curve} $C$ we mean the combinatorial data of a subset of the set of irreducible components $C^v$ of $C$. Geometrically we think of $D$ as the union of the corresponding components, which in particular may be disconnected. The data of a subcurve $D\subset C$ then also defines subcurves $D_u\subset C_{u,\underline{0}}$ of nearby stable curves. We will call two subcurves of $C$ disjoint if no irreducible component of $C$ is contained in both of them. In particular, subcurves that we call disjoint may still intersect at the nodes.

Given a subcurve $D\subset C$, we let $I_D$ be the set of internal nodes of $D$, i.e.~the set of $e\in E(\Gamma)$ such that $q_e,q_{-e}\in D$. We denote by $E_D$ the set of nodes where $D$ meets its complement, i.e.~the set of $e\in E(\Gamma)$ such that $q_e\in D$ but $q_{-e}\notin D$.
\end{notat}
\begin{df}\label{df:PRN}
Given a sequence of meromorphic differentials $\nu_k$ on smooth curves $\lbrace C_k\rbrace$ converging to $C$, and a sequence of positive reals $\mu_k$, we say that {\em there exists a scale-$\mu$ limit of $\nu_k$ on a subcurve $D$} if for any $C^v\subset D$ there exists a not identically zero meromorphic differential $\nu^v$ on $C^v$, such that for any compact set $K\subset C^v\setminus\cup_{e\in E_v}\lbrace q_e\rbrace$, the sequence of differentials $\left.\mu_k\nu_k\right|_K$ converges to $\left.\nu^v\right|_K$.

More generally, given a decomposition $D=D^{(0)}\cup D^{(1)}\dots\cup D^{(L)}$ into disjoint subcurves, and given sequences of positive reals $\mu_k^{(0)},\ldots,\mu_k^{(L)}$, we say that {\em there exists a multi-scale-$\mu$ limit of $\nu_k$ on $D$} if for any $0\le \lambda\le L$ there exists a scale-$\mu^{(\lambda)}$ limit of $\nu_k$ on $D^{(\lambda)}$. We denote the collection of limits $\nu^v$ for all $C^v\subset D$ by $\PP_D^{\mu}\nu$.
\end{df}
Of course there is never a unique choice of a sequence $\mu_k$ such that the scale-$\mu$ limit exists; any other sequence $\mu_k'$ such that there exists a finite non-zero limit of $\mu_k/\mu_k'$ yields the same notion of existence of scaled limits. For multi-scale limits we will thus always number the subcurves $D^{(\lambda)}$ in such a way that $\lim_{k\to\infty} (\mu_k^{(\lambda')}/\mu_k^{(\lambda)})=0$ for any $\lambda'<\lambda$ (if some such limit is finite, we then consider the union of $D^{(\lambda)}$ and $D^{(\lambda')}$ as one subcurve). We think of the multi-scale limit $\PP_D^{\mu}\nu$ as a collection of meromorphic differentials on all $C^v\subset D$, such that on each subcurve $D^{(\lambda)}$ the collection of differentials $\PP_{D^{(\lambda)}}^{\mu^{(\lambda)}}\nu$ is defined up to rescaling all of it by a positive real constant.

In this terminology, the main theorem will consist of arguing that any admissible sequence has a subsequence, such that for this subsequence once can define a stratification $C=C^{(0)}\cup\dots\cup C^{(L)}$ and sequences $\mu_k^{(0)},\ldots,\mu_k^{(L)}$, such that there exists a multi-scale-$\mu$ limit $\PP_C^\mu\Psi$ of $\Psi_k$. The proof will use auxiliary RN differentials constructed on plumbed subcurves; to define them, we introduce more notation.

\begin{notat}
For our fixed admissible sequence $\lbrace X_k\rbrace$, given a subcurve $D\subset C$, we denote  $D_k$ the (possibly disconnected) smooth curve obtained by taking the union of all irreducible components of $C_{u_k,0}$ that are contained in $D$, and plumbing them at every internal node $e\in I_D$, with plumbing parameter $s_{e,k}$. Denote $\PsiDP$ the RN differential on $D_k$, whose only singularities are:
\begin{itemize}
\item $\sigma_{\ell,k}$ at all the points $p_\ell\in D_k$, and
\item simple poles at external nodes $e\in E_D$, with residue $ic_{e,k}$
\end{itemize}
(where we recall that the $\uc_k$ are such that $\Psi_k=\Psi_{X_k}(\uc_k)$, which is to say that $2\pi c_{e,k}=\int_{\gamma_e} \Psi_k$).
\end{notat}
Given a subcurve $D\subset C$, we will also consider the differential on the complementary subcurve $D':=\overline{C\setminus D}$, defined similarly, except that the singularities at the external nodes $E_{D'}=-E_D$ will be prescribed by ``balancing'' the jets of $\PsiDP$ at the nodes $e\in E_D$.

\begin{notat}
Given a meromorphic differential $\nu^v$ on $C^v$ and given $e\in E_v$, we denote by $m_e:=\ord_{q_e}\nu^v$ (which is negative if $\nu^v$ has a pole at $q_e$), and denote $u_{j,e}$ the coefficients of the Laurent expansion of $\nu^v$ near $q_e$, so that
$$
 \left.\nu^v\right|_{V_e}=:\sum_{j=m_e}^\infty u_{j,e}z_e^jdz_e.
$$

We fix once and for all a positive integer $m$, which will eventually be assumed to be sufficiently large. We then denote by $J_e(\nu^v)$ the sum of the first order polar part and the holomorphic $m$-jet of the differential $\nu^v$ near $q_e$, that is we define
$$
 J_e(\nu^v):=\sum_{j=-1}^{m-1} u_{j,e}z_e^jdz_e
$$
(the inclusion of $u_{-1,e}z_e^{-1}dz_e$ in the jet is for convenience, so that the notation below is simplified; by abuse of notation we call $J_e$ the $m$-jet).
\end{notat}

\begin{notat}
Given a subcurve $D\subset C$, with complementary subcurve $D'$, the {\em balancing differential} is the RN differential $\PsiDM$ on $D'_k$, whose only singularities are:
\begin{itemize}
\item $\sigma_{\ell,k}$ at all the points $p_\ell\in D'_k$, and
\item $\sigma_{-e,k}:=I_{-e}^*(J_e(\PsiDP))$ at each external node $e\in E_{D'}=-E_D$. Explicitly, this is to say that the singular part $\sigma_{-e,k}$ of $\PsiDM$ at $q_{-e}$ is
\begin{equation}\label{sigma-ek}
      \sigma_{-e,k}:=-\left(s_e\sum_{j=-1}^{m-1}s_e^ju_{j,e,k}z_{-e}^{-j-2}\right)dz_{-e},
\end{equation}
    where $u_{j,e,k}$ are the coefficients of the $m$-jet  $J_e(\PsiDP)$, and we recall that $u_{-1,e,k}=ic_{e,k}$ by the definition of $\PsiDP$.
\end{itemize}
\end{notat}
We call the last condition of this definition the {\em balancing} condition, as it requires the singular part of $\PsiDM$ at $q_{-e}$ to ``balance'' the $m$-jet of $\PsiDP$ at $q_e$.

From now on we will denote
\begin{equation}
\calS_{D,k}:=\left(\{\sigma_{\ell,k}\}_{p_\ell\in D'};\{\sigma_{-e,k}\}_{e\in E_D}\right)
\end{equation}
the collection of all prescribed ``balancing'' singular parts of $\PsiDM$.  Thus $\calS_{d,k}$ is a point in the vector space $\CC^{(m+1)\#E_d+\sum_{p_\ell\in D'}(m_\ell+1)}$, and we denote by $\PP\calS_{D,k}$ the corresponding point in the sphere, which is its quotient by $\RR_+$. In particular $\PP\calS_{D,k}$ only makes sense if at least one prescribed singular part of $\PsiDM$ is non-zero.

\begin{df}
Given a subcurve $D\subset C$, we call the admissible sequence $\lbrace X_k\rbrace$ {\em jet-admissible on $D$} if at least one singular part in $\calS_D$ is non-zero, and there exists a limit $\PP \calS_D:=\lim_{k\to\infty} \PP \calS_{D,k}$ in the sphere.
\end{df}
We now define recursively a stratification of $C$ and the corresponding multi-scale; we first simplify notation.
\begin{notat}
Given some disjoint subcurves $C^{(0)},\dots,C^{(\lambda)}$ of $C$, we denote $C^{(\le\lambda)}:=C^{(0)}\cup\dots\cup C^{(\lambda)}$, and denote $C^{(>\lambda)}:=\overline{C\setminus C^{(\le\lambda)}}$ the complementary subcurve. We further
write for brevity  $\calS_k^{(\lambda)}=\calS_{C_k^{(\le\lambda-1)}}$, $E^{(\lambda)}:=E_{C^{(\le\lambda)}}$, and denote $\Psi_k^{(\le\lambda)}:=\Psi_{C^{(\le\lambda)}_k}^+$ and $\Psi_k^{(>\lambda)}:=\Psi_{C^{(\le\lambda)}_k}^-$.
\end{notat}
\begin{df}\label{df:jetconvlambda}
Suppose that for some $\lambda\ge 0$ the disjoint subcurves $C^{(0)},\ldots, C^{(\lambda)}$ are already given. Suppose moreover that for some given multi-scale $\mu_k^{(0)},\ldots ,\mu_k^{(\lambda)}$ there exists a multi-scale-$\mu$ limit $\PP_{C^{(\le\lambda)}}^\mu\Psi$ of the differentials $\Psi_k$ on $C^{(\le\lambda)}$. Suppose furthermore that $\lim_{k\to\infty}\mu_k^{(\lambda)}\Psi_k|_{C_k^{(>\lambda)}}=0$. We then say that $\lbrace X_k\rbrace$ is {\em jet-convergent at step $\lambda$} if it is jet-admissible on $C^{(\le\lambda')}$ for any $\lambda'<\lambda$. We call an admissible sequences {\em jet-convergent} if $C^{(\le L)}=C$ for some $\lambda=L$. In this case we call the decomposition $C=C^{(0)}\cup\dots\cup C^{(L)}$ the {\em order of vanishing stratification}.
\end{df}
\begin{rem}
The notion of order of vanishing stratification is closely related to the notion of a weak full order induced by a level function, as defined in~\cite{strata}.
\end{rem}
The definition of jet-convergenet sequences is motivated by the fact (which will be proven below) that for any sequence jet-convergent at level~$\lambda$, both~$\mu^{(\lambda)}$ and the scale-$\mu^{(\lambda)}$ limit of the sequence of differentials on $C^{(\lambda)}$ are determined uniquely by the behavior at level $\lambda-1$.

Indeed, by definition of jet-admissibility, for a sequence that is jet-convergent at step $\lambda$  there must exist a  sequence of positive reals $\mu_k^{(\lambda)}$ such that there exists a non-zero limit
\begin{equation}\label{mun}
\calS^{(\lambda)}:=\lim_{k\to\infty} \mu_k^{(\lambda)} \calS_k^{(\lambda)}.
\end{equation} The fact that $X_k$ is an admissible sequence means that $\PP\urho_k$ converge in $\PS[\#|E|-1]$, which implies that any subset of coordinates of $\PP\urho_k$ also converges in the corresponding blowup of the sphere. Thus it follows that the sequence of smooth curves $C_k^{(>\lambda-1)}$ is also admissible. Denoting $X_k^{(>\lambda-1)}$ the smooth jet curve with underlying curve $C_k^{(>\lambda-1)}$ and with prescribed singular parts $\mu_k^{(\lambda)}\calS_k^{(\lambda)}$, we see that the sequence converges to a stable jet curve $C$ with singular parts $\calS^{(\lambda)}$. Thus theorem~\ref{thm:limits} implies that there exists the limit RN differential
\begin{equation}\label{philambda}
\Phi^{(\lambda)}:=\lim_{k\to\infty}\mu_k^{(\lambda)}\Psi_k^{(>\lambda)}\,.
\end{equation}
In our recursive construction of jet-convergent sequences the subcurve $C^{(\lambda)}\subset C^{(>\lambda-1)}$ will be defined as the subcurve consisting of all irreducible components of $C$ on which the differential $\Phi^{(\lambda)}$ is not identically zero.

Then the proof of the main theorem on limits of zeroes of RN differentials essentially reduces to proving the equality
\begin{equation}\label{varPhiPhi}
\Phi^{(\lambda)}=\lim_{k\to\infty}\mu_k^{(\lambda)}\Psi_k
\end{equation}
on every irreducible component of the subcurve $C^{(\lambda)}$.

Assuming that this equality holds, we make the following definition:
\begin{df}\label{zeroestwisted}
For any irreducible component $C^v\subset C^{(\lambda)}$ denote $\PP\Phi^v:=\left.\PP\Phi^{(\lambda)}\right|_{C^v}$, and call the collection of all such $\PP\Phi^v$ the {\em twisted limit differential} on $C$. The {\em divisor of zeroes} of the twisted limit differential is defined to be the set of zeroes of all $\PP\Phi^v$, with multiplicity, away from all the nodes, together with every node $q_{|e|}$ of $C$ counted with multiplicity $\ord_{q_e}\PP\Phi^{v(e)}+\ord_{q_{-e}}\PP\Phi^{v(-e)}+2$, and together with every marked point $p_\ell$ counted with multiplicity $m_\ell+1-\ord_{p_\ell}\PP\Phi^v$, where $C^v$ is the component containing $p_\ell$.
\end{df}
\begin{rem}
By theorem \ref{thm:limits} the differential $\PP\Phi^v$ is a RN differential on $C^v$, whose singular parts are the scaled limits of the singular parts $\sigma_{-e,k}$ given by~\eqref{sigma-ek}, and of the singularities $\sigma_{\ell,k}$ at the marked points $p_\ell$ that lie on $C^v$. Thus $\PP\Phi^v$ may have higher order poles at the nodes.
\end{rem}

We are now ready to state the main theorem on the limits of zeroes of RN differentials. Until now the integer $m$ in the definition of the balancing differential was arbitrary. Now we will choose it to be sufficiently large, in order to guarantee existence of jet-convergent sequences (and to ensure that the inequality \eqref{boundomegaelll} holds). We denote $$m_0:=2g-2+\sum_\ell (m_\ell+1).$$
\begin{thm}\label{thm:zeroes}
For any fixed  $m>2m_0$, any admissible sequence $X_k$ contains a jet-convergent subsequence. For any jet-convergent sequence of smooth jet curves, equality~\eqref{varPhiPhi} holds, where $\Phi^{(\lambda)}$ is defined by~\eqref{philambda}. For any jet-convergent sequence, the limits of zeroes of $\Psi_k$ on $C_k$ exist, and form the divisor of zeroes of the twisted limit differential, counted with multiplicities.
\end{thm}

We will prove the theorem by induction in the number of levels of the order of vanishing stratification. The base case of induction is $L=0$, in which case $C=C^{(0)}$, $\Psi_k=\Psi_k^{(\le\lambda)}$, and the limit RN differential $\Psi$ does not vanish identically on any irreducible component $C^v$ of $C$. Thus the theorem in this case reduces to showing that the limit RN differential $\Psi$ is given by $\Phi=\Phi^{(0)}$, which is precisely the statement of the theorem on limit RN differentials, theorem~\ref{thm:limits} in this case.

{\bf Inductive assumption at step $\lambda\ge 0$.} {\em Assume that for a sequence $\lbrace X_k\rbrace$ that is jet-convergent at step $\lambda$, equality~\eqref{varPhiPhi} holds, where $\Phi^{(\lambda)}$ is defined by~\eqref{philambda}. Assume moreover that if $C^{(\le\lambda)}\subsetneq C$, then the limit $\lim_{k\to\infty}\left.\mu_k^{(\lambda)}\Psi_k\right|_{C^{(>\lambda)}}$ is identically zero.}

To deduce the inductive assumption at step $\lambda+1$ from the inductive assumption at step $\lambda$, we will need the following two lemmas. First we prove the lemma showing that multi-scale-$\mu_k^{(0)},\dots\mu_k^{(\lambda)}$ limits of $\Psi_k|_{C^{(\le\lambda)}}$ and $\Psi_k^{(\le\lambda)}$ are equal.
\begin{lm}\label{lm:comparePsis}
If the inductive assumption at step $\lambda$ holds, then for any $\lambda'\le \lambda$ the following equality holds:
$$
\Phi^{(\lambda')}=\lim_{k\to\infty}\left. \mu_k^{(\lambda')}\Psi_k\right|_{C^v}=\lim_{k\to\infty}\left. \mu_k^{(\lambda')}\Psi_k^{(\le\lambda)}\right|_{C^v}
$$
\end{lm}

\begin{proof}
By the inductive assumption, the multi-scale limit of $\Psi_k|_{C^{(\le\lambda)}}$ is equal to $\Phi$, and thus we need to show that the multi-scale limit of $\Psi_k^{(\le\lambda)}$ is the same. Let $\upsilon_k^{(\lambda)}$ be the ARN solution of the following jump problem on $C^{(\le\lambda)}$: $\upsilon_k^{(\lambda)}$ has zero jumps on the seams $\gamma_e, e\in I_{C^{(\le\lambda)}}$, and
 on $\gamma_e, e\in E^{(\lambda)}$ has the jump equal to
\begin{equation}\label{newjump}
\left(\Psi_k-ic_{e,k}z_e^{-1}dz_e)\right|_{\gamma_e}
\end{equation}
Formally it is a new type jump problem since the collection of the initial data are set not only on the seams at nodes of $C^{(\le \lambda)}$ but also on the seams $\gamma_e, e\in E^{(\lambda)}$ which are boundaries of the neighborhoods $U_e^{s_{e,k}}$ of the points $q_e\in C^{(\le \lambda)}$. Since $\int_{\gamma_e} (\Psi_k-ic_{e,k}d\ln z_e)=0$ the solution of this jump problem is verbatim the same and is given by the Cauchy integrals which now contain integration over $\gamma_e, e\in E^{(\lambda)}$. The same bounds hold, i.e. the $L^2$ norm of $\upsilon_k^{(\lambda)}$ is bounded by the $L^\infty$ norm of the initial data which is the $L^\infty$ norm of $\Psi_k-ic_{e,k}d\ln z_e$ on $\gamma_e$. Then by the assumption of the lemma there exists a constant $M$ such that
\begin{equation}\label{newbound}
||\upsilon_k^{(\lambda)}||_{C_k^{(\le \lambda)}}<M (\mu_k^{(\lambda)})^{-1}\sqrt {|s_k|}
\end{equation}
Consider now the differential $\widetilde\Psi_k^{(\le\lambda)}$ which is equal to $\Psi_k-\upsilon_k^{(\lambda)}$ on $C^{(\le \lambda)}\setminus \cup_{e\in E^{(\lambda)}} U_e^{s_{e,k}}$ and equals $ic_{e,k}d\ln z_e - \upsilon_k^{(\lambda)}$ inside $U_e^{s_{e,k}}$ . By the definition of $\upsilon_k^{(\lambda)}$ it has zero jumps on all the seams including $\gamma_e, e\in E^{(\lambda)}$, i.e. it is a meromorphic differential on $C^{(\le \lambda)}$, has the same singularities as $\Psi_k^{(\le\lambda)}$ and is real normalized. Hence, $\widetilde\Psi_k^{(\le\lambda)}=\Psi_k^{(\le\lambda)}$. Then \eqref{newbound} implies that on any compact set $K\subset C^v\subset C^{(\lambda')}$ not containing nodes
$$\lim_{k\to\infty} \left.\mu_k^{(\lambda')}\left(\Psi_k-\Psi_k^{(\le\lambda)}\right)\right|_K=0.$$
\end{proof}
We now obtain some bounds for the orders of zeroes and poles of $\Phi^{(\lambda)}$.
\begin{lm}\label{mem}
Suppose the inductive assumption holds at step $\lambda$. For any $e\in E^{(\lambda)}$ denote $m_e:=\ord_{q_e}\Phi^{(\lambda)}$. Then the following inequality holds:
\begin{equation}\label{meineq}
\sum_{e\in E^{(\lambda)}}m_e\le m_0.
\end{equation}
\end{lm}
\begin{proof}
Recall that $\Psi_k^{(\le\lambda)}$ is a meromorphic differential on $C_k^{(\le\lambda)}$, whose only singularities are poles of orders $m_\ell+1$ at the marked points $p_\ell$, and possibly simple poles at $E^{(\lambda)}$. Thus the total number of zeroes of $\Psi_k^{(\le\lambda)}$, counted with multiplicity, is at most equal to $2g(C_k^{(\le\lambda)})-2+\sum_\ell (m_\ell+1)+\#E^{(\lambda)}$. Furthermore, from lemma~\ref{lm:comparePsis} we know that $\lim_{k\to\infty}\left.\mu_k^{(\lambda)}\Psi_k^{(\le \lambda)}\right|_{C^{(\lambda)}}=\Phi^{(\lambda)}$, which is a differential that is regular at all $q_e$ for $e\in E^{(\lambda)}$. Since for $k$ sufficiently large the total number of zeroes and poles of $\Psi_k^{(\le \lambda)}$ within $V_e$, counted with multiplicity, is independent of $k$, and is equal to $m_e$ for the limit differential $\Phi^{(\lambda)}$, it follows that for any $k$ sufficiently large  $\Psi_k^{(\le \lambda)}$ has $m_e+1$ zeroes in $V_e$, and one simple pole there. Thus altogether we obtain for the number of zeroes of $\Psi_k^{(\le\lambda)}$ the inequality
$$
 \sum_{e\in E^{(\lambda)}}(m_e+1)\le 2g(C_k^{(\le\lambda)})-2+\sum_\ell (m_\ell+1)+\#E^{(\lambda)},
$$
which gives the statement of the lemma upon canceling $\#E^{(\lambda)}$ that appears on both sides, and noticing that $g(C_k^{(\le\lambda)})\le g(C_k)$.
\end{proof}

We are now ready to prove the main result of this section.
\begin{proof}[Proof of theorem \ref{thm:zeroes}] Assume that for an admissible sequence $\{X_k\}$ the inductive assumption is satisfied at step $\lambda$. Our first goal is to show that there exists a subsequence for which we can choose a scale $\mu^{(\lambda+1)}$, a differential $\Phi^{(\lambda+1)}$ and a subcurve $C^{(\lambda+1)}$.

Indeed, consider the set of singular parts $\calS_k^{(\lambda+1)}$ of differentials $\Psi_k^{(>\lambda)}$. At points  $q_{-e}, e\in E^{(\lambda)}$ these singular parts are defined by the balancing condition $\sigma_{-e,k}=I^*(J_e(\Psi_k^{(\leq \lambda)}))$. By the inductive assumption, the differentials $\Psi_k^{(\leq \lambda)}$ multiplied by $\mu_k^{(\lambda)}$ converge to $\Phi^{(\lambda)}$. Lemma \ref{mem} then implies that the $m$-jet of $\Phi^{(\lambda)}$ is non-zero, and thus not all singular parts $ \calS_k^{(\lambda+1)}$ are zero. Thus the projectivization $\PP\calS_k^{(\lambda+1)}$ of this set of singular parts is well-defined. This projectivization is a point on a sphere, and since the sphere is compact, there exists a subsequence of $\{X_k\}$ such that $\PP\calS_k^{(\lambda)}$ converge on the sphere. Then for this subsequence there exists a sequence of positive real numbers $\mu_k^{(\lambda+1)}$ such that there exists the limit
\begin{equation}\label{mun+1}
\calS^{(\lambda+1)}:=\lim_{k\to\infty} \mu_k^{(\lambda+1)} \calS_k^{(\lambda+1)}.
\end{equation}
Since $X_k$ is an admissible sequence, the sequence $\PP\urho_k$ converges in $\PS[\#|E|-1]$, and thus every subset of components of $\PP\urho_k$ also converges in the iterated real oriented blowup of the corresponding sphere. Thus the sequence of jet curves $C_k^{(>\lambda)}$ is also admissible. Denote $X_k^{(>\lambda)}$ the smooth jet curve with the underlying smooth curve $C_k^{(>\lambda)}$, and with prescribed singular parts $\mu_k^{(\lambda+1)}\calS_k^{(\lambda+1)}$. This sequence of jet curves must then converge to a stable jet curve $X^{(>\lambda)}$ with some prescribed singular parts $\calS^{(\lambda+1)}$. By theorem \ref{thm:limits} there exists a limit RN differential in this sequence, and thus we can define $\Phi^{(\lambda+1)}$ by equation~\eqref{philambda}, with $\lambda$ replaced by $\lambda+1$. We can then finally define the subcurve $C^{(\lambda+1)}$ to be the union of all irreducible components of $C\setminus C^{(\le \lambda)}$ on which $\Phi^{(\lambda+1)}$ is not identically zero.

Recall now that the equality $\lim_{k\to\infty}\mu_k^{(\lambda)}\Psi_k|_{C_k^{(>\lambda)}}=0$ is also a part of the inductive assumption; it immediately follows that $\lim_{k\to\infty} \mu_k^{(\lambda)}\calS_k^{(\lambda)}=0$. Since there exists a finite $\mu_k^{(\lambda+1)}$-scaled limit $\calS^{(\lambda+1)}$ of singular parts $\calS_k^{(\lambda+1)}$, as defined above, it follows that $\lim_{k\to\infty} \left(\mu_k^{(\lambda)}/\mu_k^{(\lambda+1)}\right)=0$.

Thus, upon passing to a subsequence, we will from now on assume that $\{X_k\}$ is jet-admissible at step $\lambda$, i.e.~ that equality
\eqref{mun} holds.

The following lemma proves the crucial part of the step of induction.
\begin{lm}\label{lm:step1}
If the inductive assumption at step $\lambda$ holds, then on any compact subset $K\subset C^{(\lambda+1)}$ that does not contain any nodes, the differentials $\mu_k^{(\lambda+1)}\Psi_k$ restricted to $K$  converge to $\Phi^{(\lambda+1)}|_K$.
\end{lm}
Before proving this essential lemma, we will analyze the behavior of  $\Phi^{(\lambda+1)}$ in a neighborhood of point $q_{-e}$, for $e\in E^{(\lambda)}$.

\begin{lm} \label{lm:onemore}
Suppose the inductive assumption at step $\lambda$ holds. Then there exists a constant $M$ such that for any $e\in E^{(\lambda)}$ and any $k$ the inequality
\begin{equation}\label{onemore}
\mu_k^{(\lambda+1)}|s_{e,k}|^{m_e+1}<M\mu_k^{(\lambda)}
\end{equation}
holds.
\end{lm}
\begin{proof}
Let $u_{j,e}:=\lim_{k\to\infty} \mu_k^{(\lambda)}u_{j,e,k}$ be the limits of the scaled coefficients of the Taylor expansions of $\Psi_k^{(\le\lambda)}$ at $q_e$, if such exist. Since $\ord_{q_e}\Phi^{(\lambda)}=m_e$ by definition, the limit $u_{m_e,e}$ exists and is non-zero, while $u_{j,e}=0$ for any $0\le j<m_e$. Jet-convergence at step $\lambda$ means that there exists the scale-$\mu^{(\lambda+1)}$ limit of the singular part $\sigma_{-e,k}$. Thus for all $e\in E^{(\ell)}$ and $-1\le j\le m-1$ (where we denote $u_{-1,e,k}:=-ic_{e,k}$) there exist limits $\lim \mu_k^{(\lambda+1)}u_{-j-2,-e,k} =-\lim \mu_k^{(\lambda+1)} s_{e,k}^{j+1}u_{j,e,k}$ --- where we have used the balancing condition. In particular, for $j=m_e$ there exists a finite limit
\begin{equation}\label{eq:limit}
a:=\lim_{k\to\infty} \mu_k^{(\lambda+1)} s_{e,k}^{m_e+1}u_{m_e,e,k}
\end{equation}
of this sequence.
Since the limit $u_{m_e,e}=\lim\mu_k^{(\lambda)}u_{m_e,e,k}$ also exists and is finite and non-zero, the ratio of these two sequences, which is  $\mu_k^{(\lambda+1)} s_{e,k}^{m_e+1}\left(\mu_k^{(\lambda)}\right)^{-1}$, tends to the finite non-zero limit $a/u_{m_e,e}$, and in particular is bounded above by some constant $M$ independent of $k$.
\end{proof}
We can extend this analysis to bound the pole order of $\Phi^{(\lambda+1)}$ at $q_{-e}$, which will be used below. Recall that we have denoted $m_e=\ord_{q_e}\Phi^{(\lambda)}$.
\begin{lm}\label{lm:polebound}
Suppose the inductive assumption at step $\lambda$ holds. Then for any $e\in E^{(\lambda)}$
\begin{equation} \label{ordsigma}
\ord_{q_{-e}}\Phi^{(\lambda+1)} \ge -m_e-2.
\end{equation}
\end{lm}
\begin{proof}
Continuing in the setup of the proof of the previous lemma, note that since there exists a scale-$\mu^{(\lambda)}$ limit of $\Psi_k^{(\le\lambda)}$ on $C^{(\lambda)}$, which by lemma~\ref{lm:comparePsis} is equal to $\Phi^{(\lambda)}$, it follows that for any $m_e<j<m$ the coefficients $\mu_k^{(\lambda)} u_{j,e,k}$ are bounded independent of $k$. Hence, there exist a $k_0$ and a constant $M$ such that  for all $k>k_0$ the inequality
\begin{equation}\label{M1}
|u_{m_e,e,k}/u_{j,e,k}|> M_1
\end{equation}
holds.

Suppose now for contradiction that the pole order of $\Phi^{(\lambda+1)}$ at $q_{-e}$ is higher that $m_e+2$, i.e.~that for some $m_e<j<m$ the scale-$\mu^{(\lambda+1)}$ limit of $u_{-j-2,e,k}$ is non-zero. By the balancing condition, this is equivalent to the limit
$$
  b:=\lim_{k\to\infty} \mu_k^{(\lambda+1)} s_{e,k}^{j+1}u_{j,e,k}
$$
being non-zero. Dividing equation~\eqref{eq:limit} by this limit and using the bound~\eqref{M1} yields for the absolute value
$$
  \left|\frac {a}{b}\right|=\lim \left|\frac{\mu_k^{(\lambda+1)} s_{e,k}^{m_e+1}u_{m_e,e,k}}{\mu_k^{(\lambda+1)} s_{e,k}^{j+1}u_{j,e,k}}\right|=\lim \left|\frac{u_{m_e,e,k}} {u_{j,e,k}} s_{e,k}^{m_e-j}\right|>M_1\left|s_{e,k}\right|^{m_e-j}=\infty,
$$
which is a contradiction.
\end{proof}
We can now show that the divisor of zeroes of $\Phi$ indeed does not include nodes with negative coefficients.
\begin{cor}
The divisor of zeroes of the twisted limit differential is a linear combination of points of the nodal curve with non-negative coefficients.
\end{cor}
\begin{proof}
By definition the statement means that we need to prove that $\ord_{q_e}\Phi^{v(e)}+\ord_{q_{-e}}\Phi^{v(-e)}\ge -2$ for every node $e$.
For a node $e$ such that there exists an $\ell$ such that $C^{v(e)},C^{v(-e)}\subset C^{(\lambda)}$, the twisted differential $\Phi^{(\lambda)}$ is a limit RN differential, and thus has at most simple poles at the nodes, so the inequality is immediate. For the other nodes there exists an $\lambda$ such that $q_e\in C^{(\lambda)}$ while $q_{-e}\not\in C^{(\le\lambda)}$. In this case lemma~\ref{lm:polebound} gives precisely the required inequality.
\end{proof}
We can now give the proof of the main technical lemma
\begin{proof}[Proof of lemma~\ref{lm:step1}]
Together, the differentials $\Psi_k^{(\le\lambda)}$ and $\Psi_k^{(>\lambda)}$ define a differential on the smoothing of $C_{u_k,\underline{0}}$ at all the nodes $E\setminus E^{(\lambda)}$, which is then real-normalized on all cycles not passing through nodes in $E^{(\lambda)}$.  We thus consider the \RHP with {\em zero jumps} at all seams corresponding to internal nodes $e\in I_{C^{(\le\lambda)}}$ and $e\in I_{C^{(>\lambda)}}$, and with the jump on seam $\gamma_e$ for any $e\in E^{(\lambda)}$ given by
\begin{equation}\label{eq:init}
  \left.\Psi_k^{(\le\lambda)}\right|_{\gamma_e}-I_e^*\left(\left.\Psi^{(>\lambda)}_k\right|_{\gamma_{-e}}\right),
\end{equation}
and let $\omega_k^{(\lambda)}$ be the ARN solution of this \RHP on $C_k$. Then the difference
$$
  \nu:=\left(\Psi^{(\le\lambda)}_k\sqcup \Psi^{(>\lambda)}_k\right)-\omega^{(\lambda)}_k
$$
is a differential on $C^{(\le\lambda)}_k\sqcup C^{(>\lambda)}_k$ that also has no jumps on any seam $\gamma_e$ for any $e\in E^{(\lambda)}$. Thus $\nu$ is a well-defined differential on the plumbed curve $C_k$, satisfying the following properties: its only singularities are at $p_\ell$, with the singular parts prescribed by the coordinates $w_k$ (the singularities at $q_e$ for $e\in E^{(\lambda)}$ are cut out by plumbing); $\nu$ has real integral over any cycle on $C_k$ not intersecting the seams
$\gamma_e$ for $e\in E^{(\lambda)}$; and by definition of the residues of $\Psi_k^{(\le\lambda)}$ at $q_e$the integral of $\nu$ over any seam $\gamma_e$ for $e\in E^{(\lambda)}$ is equal to $2\pi c_{e,k}$, which is the integral of $\Psi_k$ over that seam (the integral of the ARN solution over the seam is zero). Therefore, $\Psi_k-\nu$ is a holomorphic differential on $C_k$ such that all its periods over the cycles not intersecting seams $\gamma_e, e\in E^{(\lambda)}$ are real, and the periods over these seams are zero. Verbatim the same argument that proves uniqueness of $\Psi_X(\uc)$, by applying the Stokes' theorem, shows that such a differential is zero, and thus it follows that $\nu$ is equal to $\Psi_k$.

The advantage of this construction of~$\Psi_k$ over the one used in section 6 is in that the balancing condition~\eqref{sigma-ek} gives a much better upper bound for the initial data of the jump problem whose solution is $\omega_k^{(\lambda)}$. Indeed, for any $e\in E^{(\lambda)}$ let us define holomorphic differentials $f_{e,k}$ on $V_e$ and $f_{-e,k}$ on $V_{-e}$ by
\begin{equation}\label{fek}
  f_{e,k}:=\left.\Psi^{(\le\lambda)}_k\right|_{V_e}-J_e(\Psi^{(\le\lambda)}_k)=\sum_{j=m}^\infty u_{j,e,k} z_{e}^j dz_e,
\end{equation}
where $u_{j,e,k}$ denote the coefficients of the Taylor series expansion of $\Psi^{(\le\lambda)}_k$ at $q_e$, and respectively
\begin{equation}\label{f-ek}
  f_{-e,k}:=\Psi^{(>\lambda)}_k-\sigma_{-e,k}=\sum_{j=0}^\infty u_{j,-e,k}z_{-e}^j dz_{-e},
\end{equation}
where $u_{j,-e,k}$ denote the coefficients of the Laurent series expansion of $\Psi^{(>\lambda)}_k$ at $q_{-e}$.

From the balancing condition~\eqref{sigma-ek} it follows that
$$
  \left.\Psi^{(\le\lambda)}_k\right|_{\gamma_e}-I_e^*\left(\left.\Psi^{(>\lambda)}_k\right|_{\gamma_{-e}}\right)=
  \left.f_{e,k}\right|_{\gamma_e}-I_e^*\left(\left.f_{-e,k}\right|_{\gamma_{-e}}\right).
$$
Since $f_{e,k}$ is holomorphic in $V_e$, and $ f_{-e,k}$ is holomorphic in $V_{-e}$, we can apply proposition~\ref{ARNPJP} to bound the ARN solution $\omega^{(\lambda)}_k$ of the \RHP posed above.

Since $C^v\setminus\lbrace V_e\rbrace_{e\in E_v}$ is a compact subset of $C^v\setminus\lbrace q_e\rbrace_{e\in E_v}$, the sequence of differentials $\mu_k^{(\lambda)}\Psi_k^{(\le\lambda)}$ on it converges. Thus each individual term of their Taylor expansions also converges, and thus $\mu_k^{(\lambda)}f_{e,k}$, being the sum of all the terms of the Taylor expansion, starting from the $m$-th term, converges on $C^v\setminus\lbrace V_e\rbrace_{e\in E_v}$. In particular, $\mu_k^{(\lambda)}f_{e,k}$ converge on the circle $\lbrace |z_e|=1\rbrace=\partial V_e$ --- which is where we need to take the norm to apply the bound of proposition~\ref{ARNPJP}. From~\eqref{fek}  we see that for any $e\in E^{(\lambda)}$ the differential $\mu_k^{(\lambda)}f_{e,k}$ has a zero at $q_e$ of order at least $m$. Therefore,  there exists a constant $M$ independent of sufficiently small $w,u,\us$ (so that it works for all sufficiently large $k$) such that
$$
  |f_{e,k}|_{\us}< M\left(\mu_k^{(\lambda)}\right)^{-1} \max_{e\in E^{(\lambda)}}|s_{e,k}|^{(m+1)/2}.
$$
We now recall that by lemma~\ref{mem} $m_e\le m_0<m/2$ for any $e$, so that $|s_{e,k}|^{(m+1)/2}<|s_{e,k}|^{m_e+1/2}$ for any $k$ sufficiently large. Using this inequality for the right-hand-side above, and multiplying this bound by $|\mu_k^{(\lambda+1)}|$ yields
\begin{equation}\label{bound-fek}
  \mu_k^{(\lambda+1)} |f_{e,k}|_\us< M \max_{e\in E^{(\lambda)}} |s_{e,k}|^{1/2},
\end{equation}
by applying the bound for $\mu_k^{(\lambda+1)}/\mu_k^{(\lambda)}$ given by lemma~\ref{lm:onemore}.

We now bound $f_{-e,k}$. By assumption of jet-convergence at step $\lambda$, the differentials $\mu_k^{(\lambda+1)}\Psi_k^{(>\lambda)}$ converge to $\Phi^{(\lambda+1)}$. Since $\mu_k^{(\lambda+1)} f_{-e,k}$ is the regular part of $\mu_k^{(\lambda+1)}\Psi_k^{(>\lambda)}$, these also converge, and thus there exists a constant $M_2$ independent of sufficiently small $\us$ and sufficiently large $k$ such that
\begin{equation}\label{bound-f-ek}
  \mu_k^{(\lambda+1)}|f_{-e,k}|_{\us}<M_2 \max_{e\in E^{(\lambda)}}|s_{e,k}|^{1/2}
\end{equation}

Using for the initial data $\mu_k^{(\lambda+1)}(f_{e,k}-I_e^*(f_{-e,k}))$ of the jump problem the upper bounds provided by~\eqref{bound-fek},\eqref{bound-f-ek}, proposition~\ref{ARNPJP} finally yields the existence of a constant $M_3$ such that
\begin{equation}\label{boundomegaelll}
  \mu_k^{(\lambda+1)}||\omega^{(\lambda)}||_{C_k}<M_3|\us_k|^{1/2}
\end{equation}
Since the limit of $\mu_k^{(\lambda+1)}\Psi_k^{(>\lambda)}$  is finite and non-zero on any component $C^v\subset C^{(\lambda+1)}$, it follows that for any compact $K\subset C_k^{(>\lambda)}$ we have
$$
\lim_{k\to\infty} \mu_k^{(\lambda+1)}\Psi^k|_K=\lim_{k\to\infty} \mu_k^{(\lambda+1)}\left(\Psi_k^{(>\lambda)}-\omega_k^{(\lambda)}\right)|_K=$$
$$ \lim_{k\to\infty}\mu_k^{(\lambda+1)}\Psi_k^{(>\lambda)}|_K=\Phi^{(\lambda+1)}|_K.$$
\end{proof}
This lemma completes the proof of the inductive step. Indeed, lemma~\ref{lm:step1} shows that if the inductive assumption holds at step $\lambda$, then there exists a scale $\mu_k^{(\lambda+1)}$ such that the inductive assumption holds, for a suitable subcurve, also at step $\lambda+1$. In particular we have proven the following
\begin{cor}\label{cor:subjet}
Any sequence $\{X_k\}$ that is jet-convergent at step $\lambda$ has a subsequence that is jet-convergent at step $\lambda+1$.
\end{cor}
Since the order of vanishing stratification is finite, to complete the proof of the theorem it remains to determine the limits of zeroes of $\Psi_k$. Away from the nodes of $C$, clearly the limits of zeroes of $\Psi_k$ are the same as the limits of zeroes of $\mu_k^{(\lambda)}\Psi_k$, which are simply the zeroes of $\PP\Phi$, counted with multiplicity. We thus need to show that a node $q_{|e|}$ is the limit of
$\ord_{q_e}\Phi^{v(e)}+\ord_{q_{-e}}\Phi^{v(-e)}+2$ zeroes of  $\Psi_k$, counted with multiplicity.

Denote $K_{e,k}$ the compact set $(V_e\setminus U_e^{s_{e,k}})\sqcup (V_{-e}\setminus U_{-e}^{s_{e,k}})/(\gamma_e\sim\gamma_{-e})$, where the seams are identified via $I_e$, as usual. The differentials $dz_e$ in $V_{-e}\setminus U_{-e}^{s_{e,k}}$ and $s_{e,k}z_{-e}^{-2}dz_{-e}$ in $V_{-e}\setminus U_{-e}^{s_{e,k}}$ match each other on the seam $\gamma_e$ and hence define a holomorphic differential $\zeta_{e,k}$ on $K_{e,k}$ that is nowhere zero. Hence the number of zeros in $K_{e,k}$ of the differential $\Psi_k$ is equal to the number of zeros in $K_{e,k}$ of the function $\Psi_k/\zeta_{e,k}$. By the argument principle the latter is equal to the integral of $(2\pi i)^{-1}d\ln (\Psi_k/\zeta_{e,k})$ over the boundary $\partial K_{e,k}$, which is the union of the circles $|z_e|=1$ and $|z_{-e}|=1$, with opposite orientations. On these circles $\PP\Psi_k$ converges to $\PP\Phi^{v(e)}$ and $\PP\Phi^{v(-e)}$ respectively, and thus the integrals of $d\ln(\Psi_k/\zeta_{e,k})$ over them converge to the integrals of $d\ln(\Phi^{v(e)}/dz_e)$ and $d\ln(z_{-e}^2\Phi^{v(-e)}/dz_{-e})$, respectively. Thus the total number of zeroes of $\Psi_k$ within $K_{e,k}$, for $k$ sufficiently large, is equal to the total number of zeroes and poles of $\Phi^{v(e)}$ in $V_e$ plus the total number of zeroes and poles of $z_{-e}^2\Phi^{v(-e)}$ in $V_{-e}$, all counted with multiplicity. Since the only zeroes or poles of $\Phi$ in these neighborhoods are at the origins, the statement about the multiplicity follows.

The computation of the multiplicity of $p_\ell$ as the limit of zeroes of $\Psi_k$ is straightforward --- the point $p_\ell$ is a smooth point of the component $C^v$ that contains it, and thus the multiplicity of it as the limit of zeroes of $\Psi_k$ is precisely the difference of the pole orders of $\Psi_k$ and $\Phi$ at $p_\ell$.

Our main theorem~\ref{thm:zeroes} on limits of zeroes of RN differentials is thus proven.
\end{proof}

\begin{cor}\label{cor:cpt}
Suppose that for a sequence of smooth jet curves $\lbrace X_k\rbrace$ converging to a stable jet curve $X$, the limit of zeroes of RN differentials exists. Then the limits of zeroes are the divisor of zeroes of the twisted limit differential for some jet-convergent sequence.
\end{cor}
\begin{proof}
As argued in the proof of proposition~\ref{prop:compactness}, compactness of $\PS[\#E-1]$ implies that any sequence $\lbrace X_k\rbrace$ has an admissible subsequence. Furthermore, corollary~\ref{cor:subjet} shows that there exists a jet-convergent subsequence of this admissible subsequence. By theorem~\ref{thm:zeroes} the limit of zeroes of the RN differentials corresponding to this jet-convergent subsequence is the divisor of zeroes of the twisted limit differential in this subsequence. Since the limit of zeroes exists for the whole sequence, it must then be equal to the limit of zeroes in this subsequence.
\end{proof}

\appendix
\section[A]{$m$-balanced approximation}
As we see in the proof of the main theorem~\ref{thm:zeroes} on limits of zeroes of RN differentials, the main motivation for introducing the balancing condition~\ref{sigma-ek} for the jet of the differential at $q_e$ and the singular part at $q_{-e}$ is to ensure that the jump is sufficiently small, so that proposition~\ref{ARNPJP} applies, and yields a bound for the ARN solution. This bound shows that for~$\us$ sufficiently small this ARN solution is smaller  than the differentials themselves. In this appendix we develop this idea into a general notion of $m$-balanced approximations. The balancing construction used in the previous section will correspond to the case $m=0$ of this more general construction While not used for the proof of our main results, the notion of an $m$-balanced approximations for $m>0$ gives a general framework for future works aimed at understanding more precisely the asymptotic behavior of meromorphic differentials under degeneration.

For a fixed $m\in\ZZ_{>0}$, we continue with the notation of the previous section. Denote by $W$ the set of all collections of RN meromorphic differentials $\Phi=\lbrace \Phi^v\rbrace$ on $C^v$ that have poles of order up to $m_\ell+1$ at each marked point $p_\ell$, are holomorphic away from the marked points and the preimages of the nodes, and have poles of order at most $m+1$ at the preimages of all nodes, with opposite residues at $q_e$ and $q_{-e}$ for any $e$. We note that $W$ is a finite-dimensional real vector space, and for further use denote $W_0\subset W$ the vector subspace of those differentials that are regular at all $p_\ell$, and have zero residue at any node (while still allowed to have a higher order pole there). For $\Phi\in W$ we denote by $u_{j,e}$ the coefficients of the Laurent series of $\Phi^{v(e)}$ at $q_e$, denote by $\sigma_e$ the singular part, and by $J_e$ the $m$-jet of the holomorphic part of $\Phi^{v(e)}$ plus the polar term of order $-1$.

For a jet curve $X=X_{w,u,s}$ we continue to denote by $\uc=\uc(X)$ the sum of the series as constructed in proposition~\ref{prop:PsiR}, i.e.~the values such that $\Psi_X=\Psi_X(\uc(X))$.
The $m$-balanced approximation is then defined as follows.
\begin{df}\label{df:nondegenerate}
For a given $X$ and for a fixed integer $m\ge 0$,  an element $\Phi_{X}[m]=\lbrace \Phi[m]^v\rbrace\in W$  is called an {\em $m$-balanced approximation} if the following conditions hold.
\begin{itemize}
\item At each point $p_\ell$ the singular part of $\Phi[m]$ is as prescribed by the coordinates $w$.
\item At any preimage $q_e$ of any node, the residue of $\Phi[m]^{v(e)}$ is equal to $\Res_{q_e}\Phi[m]^{v(e)}=ic_e$.
\item At any preimage $q_{-e}$ of any node, the singular part $\sigma_{-e}$ of $\Phi[m]^{v(-e)}$ is equal to the pullback under $I_e^*$ of the $m$-jet $J_e$ of $\Phi[m]^{v(e)}$ at $q_{e}$:
    $$\sigma_{-e}=I_e^*(J_{e}(\Phi[m]^{v(e)})),$$
    which is the balancing condition explicitly written in equation~\ref{sigma-ek}.
\end{itemize}
\end{df}
We note that the last condition prescribes both $\sigma_{-e}$ in terms of the jet of $\Phi[m]^{v(e)}$ and, by choosing $q_{-e}$ instead, also the singular part $\sigma_e$ in terms of the jet of $\Phi[m]^{v(-e)}$. Thus the existence of $m$-balanced approximations cannot be argued by constructing them starting from some component, and then proceeding to define them explicitly on the adjoining component. We thus first need to prove that $m$-balanced approximations exist. To prove this, one could first argue uniqueness as we do below, and then deduce the existence by noticing that conditions imposed on an $m$-balanced approximation are a system of non-homogeneous linear equations on the singular parts, which must then have a solution. However, for possible applications it is important to be able to compute the $m$-balanced approximation, and we thus give a proof by an explicit construction of the approximation as a sum of a recursively defined series.
\begin{prop}\label{approx1}
For any fixed $m$, there exists a constant $t_m$ such that for any $|w|,|u|,|\us|<t_m$ there exists a unique $m$-balanced approximation $\Phi_{X}[m]$.
\end{prop}
\begin{proof}
Similarly to the construction of the RN differential in plumbing coordinates, we will prove the existence by constructing the approximation as a sum of a series $\Phi_{X}[m]=\sum_{l=0}^\infty \Phi^{(l)}$, now with the first term
$\Phi^{(0)}:=\Phi(\uc(X))$.

The further terms $\Phi^{(l)}$ for $l>1$ will lie in $W_0$, so that adding them to $\Phi^{(0)}$ does not change the singular parts of $\Phi$ at $p_\ell$ or the residues at the nodes. To define $\Phi^{(l)}$, we introduce a linear operator $R: W\to W_0$, which we think of as ``balancing'' the singular parts. An element of $W_0$ is prescribed by its singular parts at each $q_e$, and we define $R(\Phi)$ by prescribing its singular parts at each $q_{-e}$ to be
\begin{equation}\label{eq:RJ}
  s_e\sum_{j=0}^{m-1} s_e^j u_{j,e}z_{-e}^{-j-2}dz_{-e},
\end{equation}
where $u_{j,e}$ are the coefficients of the jet $J_e$ of $\Phi$ (note that this formula prescribes a singular part without residue, as required for an element of $W_0$).

We introduce a norm on $W$ by taking the maximum of norms of the singular parts. Let the linear operator $R':W\to W_0$ be defined by prescribing the singular parts to be $R'(\sigma)_{-e}=s_e^{-1}R(\sigma)_{-e}$, Since it depends smoothly on $u$ and the  collection of singular parts $R(\sigma)$ is obtained by multiplying the collection of singular parts $R'(\sigma)_{-e}$ by a diagonal matrix of $s_e$, it follows that there exists a constant $M_m$ such that
for any $\Phi\in W$ we have $|R(\Phi)|< |\us|\cdot|R'(\Phi)|\le M_m\cdot|\us|\cdot |\Phi|$.

We now define the terms of the series by setting $\Phi^{(l+1)}:=R(\Phi^{(l)})$ for any $l\ge 0$, so that the above bound shows
\begin{equation}\label{eq:sigmabound}
|\Phi^{(l+1)}|\le M_m\cdot |\us|\cdot|\Phi^{(l)}|
\end{equation}
for any $l$. Thus for $|\us|<M_m^{-1}$, the norms of the terms are bounded by a geometric sequence with ratio less than 1, and thus the series converge.

\medskip
To prove uniqueness of the approximation, suppose that $\Phi'$ is the difference of any two $m$-balanced approximations. Then $\Phi'\in W_0$ is a collection of RN differentials holomorphic at all the marked points, satisfying all the balancing conditions~\eqref{df:nondegenerate}. But then $R(\Phi')=\Phi'$ by definition, and thus it follows that $|\Phi'|\le M_m\cdot |\us|\cdot |\Phi'|$, which for $|\us|<M_m^{-1}$ is a contradiction unless $\Phi'$ is zero.
\end{proof}
By using techniques similar to the proof of lemma~\ref{mem} it can be shown that for $m>2m_0$ each meromorphic differential $\Phi[m]^v$ of an $m$-balanced approximation is not identically zero. To prove this, one proceeds inductively by the number of components of $C$, and uses the estimates similar to those in lemma~\ref{mem} to show that each $m$-jet $J_e(\Phi)$ is not identically zero --- which then implies that no singular part $\sigma_{-e}$ is identically zero, and thus no RN differential $\Phi[m]^v$ is identically zero.

Furthermore, consider the jump problem with initial data $$\left.\Phi[m]^{v(e)}\right|_{\gamma_e}-I_e^*\left(\left.\Phi[m]^{v(-e)}\right|_{\gamma_{-e}}\right).$$
Let $\omega_X[m]$ be the ARN solution of this jump problem, so that $\Phi_X[m]-\omega_X[m]$ then glues to define a differential on $C$, which is then easily seen to be equal to $\Psi_X$. The balancing condition then shows that the first $m$ terms of the initial data of the jump problem cancel (as in equations~\eqref{fek},\eqref{f-ek} in the proof of theorem~\ref{thm:limits}), and thus gives a bound for the ARN solution. As a result, one obtains the following
\begin{lm}\label{lm:limapprox}
For any $w,u,\us$ sufficiently small there exists a constant $M$ such that for any compact $K\subset C^v\setminus\cup_{e\in E_v}\lbrace q_e\rbrace$ the inequality
\begin{equation}\label{moder1}
  ||\Psi_X-\Phi_X[m]||_K< M|\us|^{(m+1)/2}
\end{equation}
holds.
\end{lm}
This lemma says that on those components $C^v$ where the limit RN differential $\Psi$ is not identically zero (that is, on $C^{(0)}$), $\Phi_X[m]$ approximates $\Psi_X$ up to order $|\us|^{(m+1)/2}$. In the same spirit as before, we can also approximate $\Psi_X|_{C^{(>0)}}$ with the same precision, by using the order of vanishing stratification.

More precisely, suppose $X_k$ is a jet-convergent sequence, as defined in definition~\ref{df:jetconvlambda}, with multi-scale $\mu^{(0)},\dots,\mu^{(L)}$. We define a new sequence of differentials as follows.
\begin{notat}
For any $0\le\lambda\le L$ we denote $\Psi_k[m]^{(\le \lambda)}$ and $\Psi_k[m]^{(>\lambda)}$ the RN differentials on $C_k^{(\le\lambda)}$ and $C_k^{(>\lambda)}$, respectively, whose only singular parts are as follows:
\begin{itemize}
\item $\sigma_{\ell,k}$, at each $p_\ell\in C^{(\le \lambda)}$, for the differential $\Psi_k[m]^{(\le \lambda)}$;
\item at every $q_e$,$e\in E^{(\lambda)}$, the differential $\Psi_k[m]^{(\le \lambda)}$ has singular part $\sigma_{e,k}[m]$ of order at most $m+1$, with residue $ic_{e,k}$;
\item at every $q_{-e}, e\in E^{(\lambda)}$, the differential $\Psi_k^{(>\lambda)}$ has singular part $\sigma_{-e,k}[2m]$ of order at most $2m+1$;
\item the following enhanced balancing condition holds:
\begin{equation}\label{balancingmarch}
\sigma_{e,k}[m]+J_e^{[2m]}(\Psi_k[m]^{(\le \lambda)})=I_e^*\left(\sigma_{-e,k}[2m]+J_{-e}^{[m]}(\Psi_k[m]^{(>\lambda)})\right),
\end{equation}
where $J_e^{[m]}(\cdot)$ is the polar term of order $-1$ plus the $m$-jet of the regular part of the corresponding differential.
\end{itemize}
\end{notat}
What the enhanced balancing condition says is that $\sigma_{e,k}[m]$ is the singular part that is $m$-balanced with the $m$-jet of $\Psi_k[m]^{(>\lambda)}$ at $q_{-e}$, while $\sigma_{-e,k}[2m]$ is the singular part that is $2m$-balanced with the $2m$-jet of $\Psi_k[m]^{(>\lambda)}$ at $q_e$. Viewing it this way, the original balancing condition~\eqref{sigma-ek} requires $0$-balancing in one direction, and $m$-balancing in the other direction.

The existence of such a pair of differentials $\Psi_k[m]^{(\le\lambda)}$ and $\Psi_k[m]^{(>\lambda)}$ requires a proof, as unlike the case of $\Psi_k^{(\le\lambda)}$, which is defined directly, and $\Psi_k^{(>\lambda)}$, which is defined by prescribing its singular parts, the differentials $\Psi_k[m]^{(\le\lambda)}$ and $\Psi_k[m]^{(>\lambda)}$ must satisfy the enhanced balancing condition, which restricts the singularities of both of them. However, arguing the same way as in the proof of existence of $m$-balanced approximations, one can prove that such differentials exist. Essentially the argument again boils down to noticing the extra powers of $s_{e,k}$ appearing in front of the singular parts $\sigma_{e,k}[m]$ and $\sigma_{-e,k}[2m]$, which imply the uniqueness of solution. Then since the enhanced balancing condition is a non-homogeneous system of $\RR$-linear equations, it follows that this system is non-degenerate, and has a solution. Similarly to our proof of proposition~\ref{approx1}, one can construct these differentials explicitly as sums of recursively defined series --- which essentially amounts to inverting a linear operator, as a recursively defined series with terms decaying as powers of $s$.
\begin{df}
For a given jet-convergent sequence $X_k$, we let $\Phi_k[m]^{(\lambda+1)}$ be the $m$-balanced approximation of $\Psi_k[m]^{(>\lambda)}$ on $C_k^{(>\lambda)}$. We call then the collection of differentials $\Phi_k[m]^{(\lambda+1)}$ on $C^{(\lambda+1)}$ the {\em $m$-balanced approximation} on $X_k$.
\end{df}
The use of the enhanced balancing condition is to guarantee a still better bound on the initial data for the suitable jump problem, and the name of $m$-balanced approximation is justified by the following main result about it.
\begin{prop}
For any jet-convergent sequence $X_k$, for any $\lambda\ge 0$ and for any $C^v\subset C^{(\lambda+1)}$, the following inequality holds:
$$
  \mu_k^{(\lambda+1)} ||\Psi_k- \Phi_k[m]^{(\lambda+1)}||_{C^v}< M_1 |\us_k|^{(m+1)/2}.
$$
\end{prop}
The proof is completely parallel to the proof of theorem~\ref{thm:zeroes}, and as we do not require this proposition for the proof of our main result, theorem~\ref{thm:zeroes}, we do not give the full details of the proof, just indicating the outline, for possible future applications.
\begin{proof}[Idea of the proof]
Mimicking the proof of theorem~\ref{thm:zeroes} one first needs to prove the statement analogous to lemma~\ref{lm:comparePsis}, showing that the multi-scale-$\mu$ limits of $\Psi_k$ and $\Psi_k[m]^{(\le\lambda)}$ are the same. The proof is by using the ARN solution of the \RHP, and noticing that the presence of positive powers of $s_{e,k}$ in each singular part $\sigma_{e,k}$ of $\Psi_k[m]^{(\le\lambda)}$ at $q_e$ for $e\in E^{(\lambda)}$, as given by the enhanced balancing condition, ensures that in the limit as $k\to\infty$ these singular parts go to zero. One then shows, similarly to the main part of the proof of theorem~\ref{thm:zeroes}, that the RN differential $\Psi_k$ is equal to
$$
  \Psi_k=\left(\Psi_k[m]^{(\le\lambda)}\sqcup \Psi_k[m]^{(>\lambda)}\right)-\omega_k[m]^{(\lambda)},
$$
where $\omega[m]^{(\lambda)}$ is the ARN solution of the jump problem with zero jumps at all seams
corresponding to internal nodes $e\in I_{C^{(\le\lambda)}}$ and $e\in I_{C^{(>\lambda)}}$, and with the jump on seam $\gamma_e$ for any $e\in E^{(\lambda)}$ given by
\begin{equation}
  \left.\Psi_k[m]^{(\le\lambda)}\right|_{\gamma_e}-I_e^*\left(\left.\Psi[m]_k^{(>\lambda)}\right|_{\gamma_{-e}}\right),
\end{equation}
Identically to the arguments in the proof of \eqref{boundomegaelll}, one can verify that the enhanced balancing condition~\eqref{balancingmarch} gives indeed the extra $m$-th power of $s_{e,k}$ in the bounds for the initial data for this \RHP, so that both upper bounds~\eqref{bound-fek} and~\eqref{bound-f-ek} are improved by an extra factor of $|s_{e,k}|^m$. Thus another application of the bound for the ARN solution of the \RHP, given by proposition~\ref{ARNPJP}, guarantee that there exists a constant $M$ such that for any $C^v\subset C^{(\lambda+1)}$ the inequality
$$
  \mu_k^{(\lambda+1)}||\omega[m]^{(\lambda)}||_{C^v}< M |\us_k|^{(m+1)/2}
$$
holds. The proposition thus follows.
\end{proof}
By taking the limit as $k\to\infty$, the proposition of course implies that in any jet-convergent sequence the scale-$\mu^{(\lambda)}$ limit of $\Phi_k[m]^{(\lambda)}$ on $C^{(\lambda)}$ is equal to the twisted differential $\Phi^{(\lambda)}$ that appears in theorem~\ref{thm:zeroes}. Moreover, the bound in the proposition then shows that the collection of differentials $\Phi_k[m]^{(\lambda)}$ for all $0\le \lambda\le L$ gives an approximation to $\Psi_k$ which, after scaling by the corresponding scale $\mu^{(\lambda)}$ is still within $|s|^m$. This information allows the study of differentials $\Psi_k$ in a degenerating sequence of jet curves with arbitrary precision; in more generality, the method of considering $m$-balanced approximations can also be applied to studying degenerations of other kinds of differentials on sequences of degenerating Riemann surfaces.

\end{document}